\documentclass[12pt]{article}\usepackage{float}
\usepackage{delarray}
\usepackage{amsmath, latexsym, amsfonts, amssymb, amsthm, amscd,enumerate,hyperref}
\usepackage{graphicx}
\usepackage{color}
\usepackage{amsmath, latexsym, amsfonts, amssymb, amsthm, amscd,multirow}
\usepackage{mathrsfs}
\usepackage[T1]{fontenc}
\makeatletter
\usepackage{float}
\allowdisplaybreaks
\newcommand{\1}{1\!\!{\sf I}}
\makeatother

{\DeclareGraphicsExtensions{.eps}}

\usepackage[latin1]{inputenc}
\usepackage{epsfig}

\usepackage{amssymb}
\usepackage{amsfonts}
\usepackage{amstext,amsmath}
{\DeclareGraphicsExtensions{.eps}}
\usepackage[latin1]{inputenc}
\usepackage{epsfig}
\usepackage{graphicx}

\usepackage{lipsum}
\usepackage{amssymb}
\usepackage{amsfonts}
\usepackage{amstext}

\DeclareGraphicsExtensions{.eps}

\theoremstyle{theorem}
\newtheorem{theorem}{Theorem}[section]

\newtheorem{proposition}[theorem]{Proposition}
\newtheorem*{proposition*}{Proposition}
\newtheorem{cor}[theorem]{Corollary}

\newtheorem*{theorem*}{Theorem}
\newtheorem{lemma}[theorem]{Lemma}

\newcommand*{\ensembledenombres}{\mathbb}
\newcommand*{\R}{\ensembledenombres{R}}
\newcommand*{\C}{\ensembledenombres{C}}

\newcommand*{\tr}{\mathrm{tr}}
\newcommand*{\A}{\mathcal{A}}
\newcommand*{\Tr}{\mathrm{Tr}}
\newcommand*{\G}{\mathbf{G}}
\newcommand*{\K}{\mathrm{K}}

\newcommand*{\g}{\mathbf{g}}
\newcommand*{\bfK}{\mathrm{K}}
\newcommand*{\Var}{\mathrm{Var}}
\newcommand*{\Cov}{\mathrm{Cov}}
\newcommand*{\wtbfG}{\widetilde{\mathbf{G}}}
\newcommand*{\wtG}{\widetilde{G}}
\newcommand*{\hbfG}{\hat{\mathbf{G}}}
\newcommand*{\hG}{\hat{G}}
\newcommand*{\wtcG}{\widetilde{\mathcal{G}}}
\newcommand*{\hcG}{\hat{\mathcal{G}}}
\newcommand*{\cG}{\mathcal{G}}
\newcommand*{\cK}{\mathcal{K}}
\newcommand*{\E}{\mathbb{E}}
\newcommand*{\hcK}{\hat{\mathcal{K}}}

\title{Strong Convergence of Multiplicative Brownian Motions on the General Linear Group}

 \makeatletter
\newcounter{author}
\renewcommand*\author[1]{%
  \stepcounter{author}%
  \ifnum\c@author=1
    \gdef\@author{#1}%
  \else
    \xdef\@author{\unexpanded\expandafter{\@author\and#1}}%
  \fi
  \csgdef{author@\the\c@author}{#1}}
\newcommand*\email[1]{%
  \csgdef{email@\the\c@author}{#1}}
\newcommand*\address[1]{%
  \csgdef{address@\the\c@author}{#1}}
\AtEndDocument{%
  \xdef\author@count{\the\c@author}%
  \c@author=1
  \print@authors}
\newcommand*\print@authors{%
  \ifnum\c@author>\author@count
  \else
    \print@author{\the\c@author}%
    \advance\c@author by 1
    \expandafter\print@authors
  \fi}
\newcommand*\print@author[1]{%
  \par\medskip
  \begin{tabular}{@{}l@{}}%
    \textsc{\csuse{author@#1}}\\
    \csuse{address@#1}\\
    \textit{E-mail address}:
    \href{mailto:\csuse{email@#1}}{\csuse{email@#1}}
  \end{tabular}}
\makeatother

\date{
    \today
}
\author{Marwa Banna}
\address{New York University Abu Dhabi, Division of Science, Mathematics, Abu Dhabi, UAE.}
\email{marwa.banna@nyu.edu}

\author{Mireille Capitaine}
\address{Institut de Math\'ematiques de Toulouse, UMR5219, Universit\'e de Toulouse, CNRS, France.}
\email{mireille.capitaine@math.univ-toulouse.fr}

\author{Guillaume C\'ebron}
\address{Institut de Math\'ematiques de Toulouse, UMR5219, Universit\'e de Toulouse, France.}
\email{guillaume.cebron@math.univ-toulouse.fr}

\usepackage{lipsum}
\providecommand{\keywords}[1]
{
  \small	
  \textit{Keywords:} #1
} %{} #1}
\providecommand{\subjclass}[1]
{
  \small	
  \textit{2000 Mathematics Subject Classification:} #1
}

\begin{document}

\maketitle

\begin{abstract}We consider the family of multiplicative Brownian motions $G_{\lambda,\tau}$ on the general linear group introduced by Driver-Hall-Kemp. They are parametrized by the real variance $\lambda\in \mathbb{R}$ and the complex covariance $\tau \in \mathbb{C}$ of the underlying elliptic Brownian motion. We show the almost sure strong convergence of the finite-dimensional marginals of $G_{\lambda,\tau}$ to the corresponding free multiplicative Brownian motion introduced by Hall-Ho: as the dimension tends to infinity, not only does the noncommutative distribution converge almost surely, but the operator norm does as well. This result generalizes the work of Collins-Dahlqvist-Kemp for the special case $(\lambda,\tau)=(1,0)$ which corresponds to the Brownian motion on the unitary group. Actually, this strong convergence remains valid when the family of multiplicative Brownian motions $G_{\lambda,\tau}$ is considered alongside a family of strongly converging deterministic matrices. 
\end{abstract}
\\ \subjclass{60B10, 60B20, 46L54, 22E30}
 \section{Introduction}

\paragraph*{Additive and multiplicative Brownian motions.} The \emph{Ginibre ensemble} consists of random matrices with independent and identically distributed complex Gaussian entries, each with variance~$1/N$, and the  \emph{Gaussian Unitary Ensemble} (GUE) is obtained by taking the Hermitian part of the Ginibre ensemble. 

In this paper, we are concerned with \emph{multiplicative} versions of these random matrices. More precisely, if we endow the set of Hermitian matrices with the inner product $\langle A,B\rangle_N= N\Re ( \Tr(A^*B))$, a matrix $H$ from the GUE can be seen as the end point $H_1$ of a Hermitian Brownian motion $(H_t)_{t\geq 0}$ starting at $H_0=0$. This stochastic process has additive increments $H_t-H_s$ (for $s<t$) which are stationary and stochastically independent. It is natural to convert those additive increment into multiplicative ones via the  \emph{rolling map}~\cite{mckean2024stochastic}, also known as the \emph{stochastic exponential map}~\cite{hakim2006exponentielle}: from $(H_t)_{t\geq 0}$, we define the stochastic process $(U_t)_{t\geq 0}$ via the Stratonovitch stochastic differential equation:
$$dU_t=U_t\circ i \, dH_t,\ \ \ \ U_0=I_N.$$
This procedure is the usual one to define a Brownian motion in a Lie group from a Brownian motion in its Lie algebra. Here, it yields a Brownian motion  $(U_t)_{t\geq 0}$ in the unitary group $U(N)$, with stationary and independent multiplicative increments $U_s^{-1}U_t$  (for $s<t$). Similarly, 
if we endow the set $M_N(\mathbb{C})$ of all matrices with the inner product $\langle A,B\rangle_N= N\Re ( \Tr(A^*B))$, a matrix $Z$ from the Ginibre ensemble can be seen as the end point $Z_1$ of a Brownian motion $(Z_t)_{t\geq 0}$ on $M_N(\mathbb{C})$ starting at $Z_0=0$. The stochastic exponential map yields the stochastic process $(G_t)_{t\geq 0}$ defined from $(Z_t)_{t\geq 0}$ via the Stratonovitch stochastic differential equation:
$$dG_t=G_t\circ i \, dZ_t,\ \ \ \ G_0=I_N.$$
The process $(G_t)_{t\geq 0}$ is a Brownian motion in the general linear group $GL_N(\mathbb{C})$, with stationary and independent multiplicative increments $G_s^{-1}G_t$  (for $s<t$).

\paragraph*{Large-$N$ limits.}
The famous theorem of Wigner~\cite{wigner1958distribution} states that the empirical eigenvalue distribution of the GUE converges weakly, almost surely, to the
standard semicircular distribution as $N\to\infty$. In the early~$90$'s, a deep relation between \emph{random matrices} and \emph{free probability} was
found by Voiculescu~\cite{voiculescu1991limit}: a sequence
of tuples of independent random matrices $H^{(1)},\ldots,H^{(p)}$ sampled from the GUE is \emph{asymptotically free}. In particular, the empirical eigenvalue distribution of any self-adjoint polynomial $P(H^{(1)},\ldots,H^{(p)})$ converges weakly,  almost surely,
to a probability measure as $N\to\infty$. Note that asymptotic freeness holds more generally for
independent random matrices which are unitarily invariant~\cite{voiculescu1992free}. In particular, a sequence
of tuples of independent random matrices $Z^{(1)},\ldots,Z^{(p)}$ sampled from the Ginibre ensemble is also \emph{asymptotically free}  as $N\to\infty$.
Let $(H_t
; t \in [0;1])$ be a Brownian motion taking values in the space 
of Hermitian $N \times N $ matrices normalized such that $\mathbb{E}\left((H_t)_{ij}(H_t)_{kl}\right)=\frac{t}{N} \delta_{il} \delta_{jk}$. Asymptotic freeness yields that a large-$N$ limit of this Brownian motion is provided by Voiculescu's free Brownian motion defined in Section \ref{Section: Notation} below. Namely, let $(A, \varphi)$ be a noncommutative probability space and $(s_t; t \in [0;1])$ be a free Brownian
motion in $(A, \varphi)$, then for all choices of $t_1, \cdots , t_n \in [0, 1]$, one has
\[
\tr_N\left( H_{t_1}
\dots H_{t_n}\right)
\xrightarrow[N\rightarrow +\infty]{} \varphi(s_{t_1}\dots  s_{t_n}),\]
almost surely, where $\tr_N$ denotes the normalized trace $\frac{1}{N}
\Tr$. 

In 1997, Biane~\cite{biane1997free} suggested to study the limit of the processes $(U_t)_{t\geq 0}$ and $(G_t)_{t\geq 0}$ as $N\to \infty$. He proved that $(U_t)_{t\geq 0}$ converges in $*$-distribution to a limiting process $(u_t)_{t\geq 0}$ called the \emph{free unitary Brownian motion}, and conjectured  that $(G_t)_{t\geq 0}$ converges as well to a certain process called the \emph{free multiplicative Brownian motion}. Note that the convergence of $(U_t)_{t\geq 0}$ as $N\to \infty$ has been proved rigorously at the same time in two other works by Rains~\cite{rains1997combinatorial} and Xu~\cite{xu1997random}. This fact has been known to physicists before, in connection to gauge theory~\cite{gopakumar1995mastering,kazakov1980non,singer1995master}, and the two-dimensional Yang-Mills theory with structure group $U(N)$ has been intensively studied since then from a mathematical point of view \cite{anshelevich2012quantum,cebron2017generalized,dahlqvist2016free,driver2017makeenko,driver2017three,levy2017master,levy2020two,sengupta2008traces}. The process $(G_t)_{t\geq 0}$ is also directly related to Yang-Mills theory~\cite{lohmayer2008possible}. In fact, the processes $(U_t)_{t\geq 0}$ and $(G_t)_{t\geq 0}$ are linked together via the Segal-Bargmann-Hall transform~\cite{hall1994segal}, a useful tool in two-dimensional Yang-Mills theory~\cite{albeveriosegal,driver1999yang,hall2001coherent}. Moreover, the limit of the Segal-Bargmann-Hall transform on $N\times N$ matrices as $N\to \infty$ is also a fruitful and productive area of research~\cite{biane1997segal,cebron2013free,cebron2018segal,chan2020segal,driver2013large,ho2016two}.

Concerning the large-$N$ limit of the Brownian motion $(G_t)_{t\geq 0}$, the independence, stationarity, and unitary invariance of its left multiplicative increments, together with Voiculescu's asymptotic freeness, reduce the study of the convergence of the entire process to the convergence in $*$-distribution of a marginal $G_t$, for any fixed time $t\geq 0$. This convergence was proved in 2013 by the third author in~\cite{cebron2013free}, solving the conjecture of Biane. Independently, Kemp~\cite{kemp2016large} gave a unified proof of both convergences in $*$-distribution of $(U_t)_{t\geq 0}$ and $(G_t)_{t\geq 0}$ by studying a two-parameter family of Brownian motions in $GL_N(\mathbb{C})$, with complementary estimates in~\cite{cebron2014fluctuations,kemp2017heat}; the two processes $(U_t)_{t\geq 0}$ and $(G_t)_{t\geq 0}$ corresponding to particular choices of the parameters. While the convergence in $*$-distribution of $(U_t)_{t\geq 0}$ implies the convergence of the empirical eigenvalue distribution of $U_t$ for any $t\geq 0$, it is not true for $(G_t)_{t\geq 0}$: since $G_t$ is not a normal operator ($G_t$ does not commute with $G_t^*$), the convergence in $*$-distribution of $(G_t)_{t\geq 0}$ does not imply the convergence of the empirical eigenvalue distribution of $G_t$. The convergence of the empirical eigenvalue distribution of $G_t$ as $N\to \infty$ is still an open problem, even if there is a natural candidate for the limiting measure. We expect the limit to be the Brown measure of the free multiplicative Brownian motion, and there has been a lot of effort to understand and describe this measure~\cite{driver2022brown,Eaknipitsari,gudowska2003infinite,hall2019brown,ho2019brown,lohmayer2008possible}. The state of art concerning the understanding of the Brown measure of the free multiplicative Brownian motion seems to be the results of Hall and Ho in~\cite{Hall-Ho-23}.

\paragraph*{The model.} In 2020, Driver, Hall and Kemp~\cite{driver2020complex} observed that one can naturally include a third parameter, obtaining a larger family of Brownian motions in $GL_N(\mathbb{C})$. Notably, this three-parameter family has been shown~\cite[Theorem 3.2]{driver2020complex} to be, up to multiplication by a scalar process, the most general class of Brownian motions in $GL_N(\mathbb{C})$ that preserves the unitary invariance of increments, a property that ensures asymptotic freeness as $N\to \infty$. In this paper, we shall study the large-$N$ limit of a Brownian motion from this three-parameter family. We now proceed with a precise definition.

Let $\theta\in \mathbb{R}$, $a\geq 0$ and $b\geq 0$ that are not both zero, and set 
$\lambda=a^2+b^2$ and $ \tau=\lambda-e^{2i\theta}(a^2-b^2)$. We denote by $(H_t)_{t\geq 0}$ and $(\tilde H_t)_{t\geq 0}$ two independent Brownian motions on the set of Hermitian $N\times N$ matrices endowed with the inner product $\langle \cdot,\cdot\rangle_N$. The \emph{elliptic Brownian motion} $(Z_{\lambda,\tau}(t))_{t\geq 0}$ on $M_N(\mathbb{C})$ is then defined by
\begin{equation}\label{R-E-BM}
    Z_{\lambda,\tau}(t)= e^{i \theta}(aH_t+ib\tilde H_t),
\end{equation}
whose law is uniquely defined by the parameters $\lambda\geq 0$ and $\tau\in \mathbb{C}$. 
Further details regarding these parameters are provided in Section \ref{Section: Notation}. The multiplicative $(\lambda,\tau)$-Brownian motion $(G_{\lambda, \tau}(t))_{t\geq 0}$ is the diffusion process on $GL_N(\mathbb{C})$ which is the solution of the  Stratonovich stochastic differential equation 
\begin{equation}\label{Stratonovich SDE}
    dG_{\lambda, \tau}(t)= G_{\lambda, \tau}(t)\circ  i \, dZ_{\lambda, \tau}(t)  \quad \text{with}\quad G_{\lambda, \tau}(0)=I_N.
\end{equation}
It  is well-defined for all $\lambda\geq 0$ and $\tau\in \mathbb{C}$ such that $|\tau-\lambda|\leq \lambda$. The case $(\lambda,\tau)=(1,0)$ yields the Brownian motion $(U_t)_{t\geq 0}$, while the case $(\lambda,\tau)=(1,1)$ corresponds to the Brownian motion $(G_t)_{t\geq 0}$. In both cases, the convergence in $*$-distribution as $N\to \infty$ has been established~\cite{biane1997free,cebron2013free}. More generally, the convergence in $*$-distribution of $(G_{\lambda, \tau}(t))_{t\geq 0}$  is proved in~\cite{kemp2016large} when $\tau$ is real. The proofs of~\cite{cebron2013free} and ~\cite{kemp2016large} extend  in a straightforward way for $\tau\in \mathbb{C}$, yielding the convergence in $*$-distribution of $(G_{\lambda, \tau}(t))_{t\geq 0}$ towards a  process $(g_{\lambda,\tau}(t))_{t\geq 0}$, known as the \emph{free multiplicative $(\lambda,\tau)$-Brownian motion}, which is defined in Section~\ref{Section: Notation}. For completeness, we will provide a proof of this convergence in the Appendix. 

\paragraph*{Strong convergence.} In 2005, Haagerup and Thorbj{\o}rnsen~\cite{haagerup2005new} solved a longstanding problem on $C^*$-algebras, by strengthening the asymptotic freeness of GUE matrices to a stronger form. They showed that, almost surely, as $N\to\infty$, for
any noncommutative polynomial $P$ in $p$ variables, the operator norm
$\|P(H^{(1)},\ldots,H^{(p)})\|$ converges to  $\|P(s^{(1)},\ldots,s^{(p)})\|$, where   $(s^{(1)},\ldots,s^{(p)})$ are freely independent semicircular elements. This property is referred to as strong convergence, and we say that  $(H^{(1)},\ldots,H^{(p)})$ \emph{strongly converges} to $(s^{(1)},\ldots,s^{(p)})$ as $N\to\infty$. It follows by linearity that a sequence
of tuples of independent random matrices $(Z^{(1)},\ldots,Z^{(p)})$ from the Ginibre ensemble strongly converges to $(z^{(1)},\ldots,z^{(p)})$ as $N\to\infty$, where $(z^{(1)},\ldots,z^{(p)})$ are freely independent circular variables. 
This result was  extended to the case of GOE/GSE matrices \cite{Schultz}, general Wigner matrices \cite{Anderson, capitaine2007strong}, and Haar matrices \cite{CollinsMale}. Furthermore, it remains valid when such families of random matrices are considered alongside a family of strongly converging deterministic matrices \cite{Male, CollinsMale, BelinschiCapWigner}. More recently, similar strong convergence results have been established for a broad class of Gaussian random matrix models~\cite{Bandeira}.

It is worth noticing that establishing strong convergence for sequences of matrices plays  a central role in several recent breakthrough results in particular on random graphs \cite{BordenaveCollinsANNofMath,chen2024newapproachstrongconvergenceI} and operator algebras \cite{haagerup2005new,HST, Hayes, BelinschiCap, HayesJekel, BordenaveCollinsPT, DelasalleMagee, ParraudPT,  chen2024newapproachstrongconvergenceII}.
Many results are essentially based on the so-called linearization trick introduced in  \cite{haagerup2005new} and inspired by \cite{Pisier}. Note that 
\cite{CGParraud, ParraudGUE} obtained a new proof of the result of  Haagerup and Thorbj{\o}rnsen~\cite{haagerup2005new} requiring neither the linearization trick nor the Stieltjes transform. The main idea of their proof is to interpolate GUE matrices and a free semi-circular system with the help of a free Ornstein-Uhlenbeck process. Very recently, \cite{chen2024newapproachstrongconvergenceI} found a new  approach to strong
convergence based on the $1/N$-expansion. 
\paragraph*{Results.} Concerning the multiplicative $(\lambda,\tau)$-Brownian motion $(G_{\lambda, \tau}(t))_{t\geq 0}$, the strong convergence has been shown by Collins-Dahlqvist-Kemp in the particular case $(\lambda,\tau)=(1,0)$, which corresponds to the Brownian motion on the unitary group \cite{Collins-Dahlqvist-Kemp}. Their proof uses the convergence of the support of the empirical law of eigenvalues together with the fact that $G_{1, 0}(t)$ is a normal matrix. For general parameters $(\lambda,\tau)$, even if we do have a hard edge result for the empirical law of the singular values of $G_{\lambda, \tau}(t)$ (at least when $\lambda=\tau$, see~\cite{ahn2023extremal}), it does not imply strong convergence because $G_{\lambda, \tau}(t)$ is not necessarily normal. The main result of the paper is an inclusion of  spectra involving $(\lambda,\tau)$-Brownian motion on $GL_N(\mathbb{C})$ for any allowed parameters $(\lambda,\tau)$ and deterministic matrices. More precisely, we prove that the spectrum of any self-adjoint polynomial in $(G_{\lambda, \tau}(t))_{t\geq 0}$ and deterministic matrices $A_N$ is asymptotically included in a neighborhood of the spectrum of the same polynomial in $(g_{\lambda, \tau}(t))_{t\geq 0}$ and  $A_N$, where $(g_{\lambda, \tau}(t))_{t\geq 0}$ and  $A_N$ are freely independent. In order to formulate this result, we need to consider a  $C^\ast$-probability space $(\mathcal{A}_N,\varphi_N)$ in which  $(g_{\lambda, \tau}(t))_{t\geq 0}$ and  $A_N$ are freely independent. Such a $C^\ast$-probability space always exists thanks to Theorem \ref{freeproduct}: it suffices to consider the free product of $(M_N(\mathbb{C}), \frac{1}{N}\Tr)$ with the noncommutative probability space where the free multiplicative $(\lambda,\tau)$-Brownian motion $(g_{\lambda, \tau}(t))_{t\geq 0}$ is defined. Note that in this  $C^\ast$-probability space $(\mathcal{A}_N,\varphi_N)$, the restriction of $\varphi_N$ to $M_N(\mathbb{C})$ is $\frac{1}{N}\Tr$.
 
\begin{theorem}\label{maintheo-sp}
Let $\lambda\geq 0$ and $\tau\in \mathbb{C}$ such that $|\tau-\lambda|\leq \lambda$ and let, for all $N \geq 1$,  $(G_t)_{t\geq 0}:=(G_{\lambda, \tau}(t))_{t\geq 0}$ be a multiplicative $(\lambda,\tau)$-Brownian motion on $GL_N(\mathbb{C})$.  Let $A^N:=(A_1^N,  \dots, A_r^N,A_1^{N\ast},  \dots, A_r^{N\ast})$  be a $2r$-tuple of $N \times N$ deterministic matrices that are uniformly bounded over $N$. Let $(g_t)_{t\geq 0}:=(g_{\lambda, \tau}(t))_{t\geq 0}$ be a free multiplicative $(\lambda,\tau)$-Brownian motion in $(\mathcal A, \varphi)$ as defined by \eqref{free Stoch Diff Eq} which is free from $A^N$ in $(\mathcal{A}_N,\varphi_N)$ for all $N\geq 1$.

Then, for any $\delta>0$,  $p\geq 0$,  times $t_1,\ldots,t_p\geq 0$ and   any noncommutative self-adjoint polynomial $Q$ in $4p+2r$ indeterminates,  almost surely, we have for $N$ large  enough
\begin{multline*}
sp\big(Q(G_{t_1},G_{t_1}^*,G_{t_1}^{-1},(G_{t_1}^{-1})^*,\ldots,G_{t_p},G_{t_p}^*,G_{t_p}^{-1},(G_{t_p}^{-1})^*, A^N)\big) \\ \subset sp\big(Q(g_{t_1},g_{t_1}^*,g_{t_1}^{-1},(g_{t_1}^{-1})^*,\ldots,g_{t_p},g_{t_p}^*,g_{t_p}^{-1},(g_{t_p}^{-1})^*,A^N)\big) + (-\delta, \delta).
\end{multline*}
 \end{theorem}

 As explained at the end of Section \ref{Section: idea of proof}, this allows us to deduce in particular the following strong convergence when the sequence of deterministic matrices converges strongly (for the precise definition of the strong convergence, see Section~\ref{Section: Notation}).
 \begin{cor}\label{maintheo}
 Let the setting be like above. Assume, in addition, that $A^N$ converges strongly to $\mathbf{a}$ in $(\mathcal{A},\varphi)$.  Then for any $p\geq 0$, and times $t_1,\ldots,t_p\geq 0$, almost surely for  any noncommutative polynomial $P$ in $4p+2r$ indeterminates,
\begin{multline*}
    \lim_{N\to \infty}\tr_N\left(P(G_{t_1},G_{t_1}^*,G_{t_1}^{-1},(G_{t_1}^{-1})^*,\ldots,G_{t_p},G_{t_p}^*,G_{t_p}^{-1},(G_{t_p}^{-1})^*,A^N)\right)\nonumber\\
=\varphi\left(P(g_{t_1},g_{t_1}^*,g_{t_1}^{-1},(g_{t_1}^{-1})^*,\ldots,g_{t_p},g_{t_p}^*,g_{t_p}^{-1},(g_{t_p}^{-1})^*,\mathbf{a})\right)
\end{multline*}
and
\begin{multline*}
    \lim_{N\to \infty}\left\|P(G_{t_1},G_{t_1}^*,G_{t_1}^{-1},(G_{t_1}^{-1})^*,\ldots,G_{t_p},G_{t_p}^*,G_{t_p}^{-1},(G_{t_p}^{-1})^*,A^N)\right\|\\
    =\|P(g_{t_1},g_{t_1}^*,g_{t_1}^{-1},(g_{t_1}^{-1})^*,\ldots,g_{t_p},g_{t_p}^*,g_{t_p}^{-1},(g_{t_p}^{-1})^*,\mathbf{a})\|
\end{multline*}
with $\mathbf{a}$ being freely independent of $g_t$.
 \end{cor}

To establish strong convergence results, most proofs build upon foundational ideas from~\cite{haagerup2005new}, which analyze the spectrum by obtaining precise estimates of the mean Stieltjes transform of the empirical spectral measure. However, more recent approaches bypass the linearization trick and instead rely on interpolation techniques.  

In this paper, we adopt such an interpolation-based approach, drawing on ideas developed by Collins, Guionnet, and Parraud~\cite{CGParraud} and further extended by Parraud~\cite{ParraudGUE,ParraudHaar,parraud-HaarUnitary,ParraudPT}. Another interpolation technique was introduced in~\cite{chen2024newapproachstrongconvergenceI}, but it does not apply to our setting, see Section \ref{Section: idea of proof} for more details. 
Unlike the previously mentioned interpolation techniques, our method is \emph{multiplicative} rather than additive, marking a key distinction from earlier approaches. Moreover, multiplicative interpolation requires careful choices at each step of the proof to interpolate with either a left- or right-invariant multiplicative Brownian motion, to ensure obtaining symmetric terms that can be combined and effectively controlled.
 
 To obtain almost sure convergence results, we prove an estimate on the variance that is of independent interest as well.
\begin{theorem}\label{theo: variance estimate}
Let $f$ be a  function of $C^3(\mathbb{R})$ with compact support. For any $p\in \mathbb{N}$ and $\ell=1,\dots,p$, let  $(G^{(\ell)}_{\lambda, \tau}(t))_{t\geq 0}$ be independent  $(\lambda,\tau)$-Brownian motions as defined in \eqref{Stoch Diff Eq sigma} and denote by $\G_t=(G^{(1)}_{\lambda, \tau}(t), \dots, G^{(p)}_{\lambda, \tau}(t))$. Let $A^N:=(A_1^N,  \dots, A_r^N,A_1^{N\ast},  \dots, A_r^{N\ast})$  be a $2r$-tuple of $N \times N$ deterministic matrices that are uniformly bounded over $N$. Then, for any noncommutative polynomial $P$ in $4p+2r$ variables, and any $t\geq 0$,
    \begin{equation*} 
\Var{\left[ \tr _N(f(PP^{*}(\G_t,\G_t^\ast,\G_t^{-1}, (\G_t^{-1})^\ast,A^N)))\right]}
=O\Big(\frac{1}{N^2}\Big).
\end{equation*}

\end{theorem}
The rate $O(1/N^2)$ was to be expected, since Kemp \cite[Theorem 1.5.]{kemp2017heat} has shown this kind of variance estimates for $f$ in some Gevrey class,  and $\lambda,\tau\in \mathbb{R}$. Adding the compactness of the support of $f$ allows us to weaken the regularity of $f$ to $C^3(\mathbb{R})$.
 
\paragraph{Organization of the paper.}
The paper is organized as follows:   In Section~\ref{Section: Notation}, we provide the necessary background, precisely define the tools and objects under study, and introduce the notation used throughout the paper. Additionally, we present some basic but important stochastic computations that will be essential for the proof.   In Section~\ref{Section: idea of proof}, we outline the main steps of the proof and highlight the key intermediate results that need to be established. Sections~\ref{Section:Estimating Variance} and~\ref{Section:fundamental result} form the core of the paper, containing the main proofs.  Finally, the Appendix includes the proof of the weak convergence step.  
 
 \section{Set-up, notation and basic properties}\label{Section: Notation}
 In this section, we provide the necessary background for this paper on free probability, multiplicative Brownian motions on general linear groups, and other necessary tools for the proofs.
 
 \paragraph{Noncommutative probability spaces and distributions.} 

 A ${\cal C}^*$-probability space is a pair $\left({\mathcal A}, \varphi\right)$ consisting of a unital $ {\mathcal C}^*$-algebra ${\mathcal A}$ and  
a linear map $\varphi: {\cal A}\rightarrow \mathbb{C}$ such that $\varphi(1_{\cal A})=1$ and $\varphi(aa^*)\geq 0$ for all $a \in {\mathcal A}$. The map $\varphi$ is a trace if it satisfies $\varphi(ab)=\varphi(ba)$ for every $(a,b)\in {\mathcal A}^2$. A trace is said to be faithful if $\varphi(aa^*)>0$ whenever $a\neq 0$. 
An element of ${\mathcal A}$ is called a noncommutative random variable.

 The noncommutative distribution of a family $\mathbf{a}=(a_1,\ldots,a_p)$ of noncommutative random variables in a ${\mathcal C}^*$-probability space $\left({\mathcal A}, \varphi\right)$ is defined as the linear functional $\mu_{\mathbf{a}}:P\mapsto \varphi(P(\mathbf{a},\mathbf{a}^*))$ defined on the set of polynomials in $2p$ noncommutative indeterminates, where $(\mathbf{a},\mathbf{a}^*)$ denotes the $2p$-tuple $(a_1,\ldots,a_p,a_1^*,\ldots,a_p^*)$.

For any self-adjoint element $a$ in a ${\mathcal C}^*$-probability space  $\left({\mathcal A}, \varphi\right)$,  there exists a probability  measure $\nu_{a}$ on $\mathbb{R}$ such that,   for every polynomial P, we have
$$\mu_{a}(P)=\int P(t) \mathrm{d}\nu_{a}(t).$$
We then identify $\mu_{a}$ and $\nu_{a}$. If $\varphi$ is faithful then the support of $\nu_{a}$ is the spectrum of $a$ leading to the norm characterization  $\|a\| = \sup\{|z|, z\in \rm{support} (\nu_{a})\}$. 

 A fundamental distribution in  free probability is the semicircular distribution. A self-adjoint noncommutative random variable $a$ in a ${\mathcal C}^*$-probability space  $\left({\mathcal A}, \varphi\right)$ is said to be  semicircular  with mean $0$ and variance $\sigma^2$ if
  for any $k\in \mathbb{N}$, $$\varphi(x_i^k)= \frac{1}{2\pi} \int t^k 
\sqrt{4\sigma^2-t^2}{\1}_{[-2\sigma;2\sigma]}(t) dt.$$ We say that a semicircular noncommutative random variable is standard if its variance is equal to one. 
\paragraph{Freeness, free product.} A family of noncommutative random variables $(a_i)_{i\in I}$ in a ${\mathcal C}^*$-probability space  $\left({\cal A}, \varphi\right)$ is free if for all $p\in \mathbb{N}$ and all polynomials $P_1,\ldots,P_p$ in two noncommutative indeterminates, one has 
\begin{equation}\label{freeness}
\varphi(P_1(a_{i_1},a_{i_1}^*)\cdots P_k (a_{i_p},a_{i_p}^*))=0
\end{equation}
whenever $i_1\neq i_2, i_2\neq i_3, \ldots, i_{p-1}\neq i_k$ and $\varphi(P_l(a_{i_l},a_{i_l}^*))=0$ for $l=1,\ldots,p$.

For a thorough introduction to free probability theory, we refer to \cite{voiculescu1992free}.
\begin{theorem}\label{freeproduct} (\cite[Theorem 7.9]{NicaSpeicher})
    Let $(\mathcal{A}_i, \varphi_i)_{i\in I}$ be a family of  ${\mathcal C}^*$-probability spaces such that the functionals $\varphi_i: \mathcal{A}_i \rightarrow \mathbb{C}, i\in I$ are faithful traces. Then there exists a ${\mathcal C}^*$-probability space $(\mathcal{A}, \varphi)$ with $\varphi$ a faithful trace and a family of norm-preserving unital $*$-homomorphisms
    $W_i: \mathcal{A}_i\rightarrow \mathcal{A}, i\in I,$ such that 
    \begin{itemize}
        \item $\varphi\circ W_i=\varphi_i, \forall i\in I$,
        \item the unital ${\mathcal C}^*$-subalgebras form a free family in $(\mathcal{A}, \varphi)$.
    \end{itemize}
\end{theorem}
\paragraph{Free Brownian motion.}
Let $\mathcal A$ denote a von Neumann algebra and $\varphi$ be a faithful normal, tracial
state on $\mathcal A$. 
Let  $(\mathcal A_t)_{t\geq 0}$ be a family  of unital weakly closed $\ast$-subalgebras of $\mathcal A$ with
$\mathcal A_u\subset \mathcal A_t$ for $u\leq t$. A $(\mathcal A_t)$-free semicircular Brownian motion $(s_t)_{t\geq 0}$ is a family of 
a self-adjoint elements of  $\mathcal A$ such that $s_t$ has  semicircular distribution 
with mean 0 and variance $t$, $s_t\in \mathcal A_t$ and for $u\leq t$, $s_t-s_u$ is free from $\mathcal A_s$ and has 
 semicircular distribution 
with mean 0 and variance $t-u$.
We refer to \cite{BianeSpeicher} for the construction of a free Brownian motion on the
free Fock space, using creation and annihilation operators.
\paragraph{Strong convergence.}
Let $\mathbf a=(a_1,\ldots,a_p)$ be a  $p$-tuple of noncommutative random variables in a $\mathcal C ^*$-noncommutative probability space  $(\mathcal A,\varphi)$. 
 Let $(\mathbf{a}_N)_{N\geq 1}=(a_1^{(N)}, \ldots, a_p^{(N)})_{N\geq 1}$ be a sequence of  $p$-tuples of noncommutative random variables in  $\mathcal C ^*$-noncommutative probability spaces $(\mathcal A_N,\varphi_N)$.  We say that $(\mathbf{a}_N)_{N\geq 1}$ converges in $*$-distribution  
 towards $\mathbf a$ if for any polynomial $P$ in $2k$ noncommutative indeterminates,
$$\lim_{N\rightarrow +\infty} \varphi_N(P(\mathbf a_N, \mathbf a^*_N))=\varphi(P(\mathbf a, \mathbf a^*)).$$
We say that $(\mathbf{a}_N)_{N\geq 1}$ strongly converges towards $\mathbf a$  if $(a_N)_{N\geq 1}$ converges in $*$-distribution  
and if  for any polynomial $P$ in $2k$  noncommutative indeterminates,
$$\lim_{N\rightarrow +\infty} \|P(\mathbf a_N, \mathbf a^*_N)\|=\|P(\mathbf a, \mathbf a^*)\|,$$
where $\|.\|$ denotes the operator norm. According to Proposition 2.1 in \cite{CollinsMale}, we have the following equivalent formulation.
\begin{proposition}\label{scCollinsMale}
Let $\mathbf{a}_N=(a_1^{(N)}, \ldots, a_p^{(N)})$ and $\mathbf a =(a_1,\ldots,a_p)$ be $p$-tuples of variables in $\mathcal{C}^*$-probability spaces, $(\mathcal A_N, \varphi_N)$ and $(\mathcal A, \varphi)$ with faithful states. Then the following assertions are equivalent.
\begin{enumerate} \item $\mathbf{a}_N$ converges strongly to $\mathbf{a}$.
\item For any self-adjoint variable $h_N=P(\mathbf{a}_N,\mathbf{a}_N^*)$, where $P$ is a fixed polynomial, $\mu_{h_N}$ converges in weak $\ast$-topology (that is relatively to continuous function on $\C$) to $\mu_h$, where we
denote by   $\mu_{h_N}$ the distribution of $h_N$ and by $\mu_h$ the distribution of $h=P(\mathbf a,\mathbf a^*)$.
Moreover, the support of $\mu_{h_N}$ converges in Hausdorff distance to the support of $\mu_h$, that is:
 for any $\epsilon >0$, there exists $N_0$ such that for any $N\geq N_0$, 
$$\mbox{Supp}(\mu_{h_N}) \subset \mbox{Supp}(\mu_{h})+(-\epsilon, \epsilon).$$
The symbol $\mbox{Supp}$ means the support of the measure.
\end{enumerate}
In particular, the strong convergence in distribution of a single self-adjoint variable is its convergence in
distribution together with the Hausdorff convergence of its spectrum.
\end{proposition}

 \paragraph{Multiplicative $(\lambda,\tau)$-Brownian motion.} We say $(H_t)_{t\geq 0}$ is an $N \times N$ Wigner Brownian motion whenever it is an $N \times N$ Hermitian matrix with iid complex Brownian motions above the diagonal and iid real Brownian motions on the diagonal, all of variance $t/N$ and starting at $0$.  Let $\theta\in \mathbb{R}$, $a\geq 0$ and $b\geq 0$, and set 
$\lambda=a^2+b^2$ and $ \tau=\lambda-e^{2i\theta}(a^2-b^2)$. For convenience, we recall the definition of the rotated elliptic Brownian motion  $(Z_{\lambda,\tau}(t))_{t\geq 0}$ on $M_N(\mathbb{C})$ defined above:
\begin{equation*}
    Z_{\lambda,\tau}(t)= e^{i \theta}(aH_t+ib\tilde H_t),
\end{equation*}
where $(H_t)_{t\geq 0}$ and $(\tilde H_t)_{t\geq 0}$ are two independent Wigner Brownian motions. Note that its law is uniquely defined by the parameters $\lambda\geq 0$ and $\tau\in \mathbb{C}$. These two parameters $\lambda$ and $\tau$ correspond respectively to the positive variance and the complex covariance:
$$\lambda=\mathbb{E}\left[\frac{1}{N}\Tr\Big(Z_{\lambda,\tau}(1)^* Z_{\lambda,\tau}(1)\Big)\right],\ \ \tau=\mathbb{E}\left[\frac{1}{N}\Tr\Big(Z_{\lambda,\tau}(1)^* Z_{\lambda,\tau}(1)\Big)\right]-\mathbb{E}\left[\frac{1}{N}\Tr\Big(Z_{\lambda,\tau}(1) ^2\Big)\right] .$$
The multiplicative $(\lambda,\tau)$-Brownian motion $(G_{\lambda, \tau}(t))_{t\geq 0}$ satisfies the Stratonovich SDE in \eqref{Stratonovich SDE}, which can be written in It\^o form as follows: for any $t\geq 0$,
\begin{equation}\label{Stoch Diff Eq}
    dG_{\lambda, \tau}(t)= \, G_{\lambda, \tau}(t)\big( i \, dZ_{\lambda, \tau}(t) -\frac{1}{2}(\lambda-\tau)\, dt \big) \quad \text{with}\quad G_{\lambda, \tau}(0)=I_N.
\end{equation}
Because of the global Lipschitzness of the coefficients, the (strong) existence and the uniqueness of the solution follow from a standard Picard iteration argument.
 \begin{proposition}\label{SDEforinverse}
    For any $\lambda\geq 0$ and $\tau\in \mathbb{C}$ such that $|\tau-\lambda|\leq \lambda$, $(G_{\lambda, \tau}(t))_{t\geq 0}$ is invertible for all $t\geq 0$, and has  independent stationary multiplicative increments.  
\end{proposition}
\begin{proof}
    Define $(K(t))_{t\geq 0}$ to be the solution to
    $$ dK(t)= \, \big( -i \, dZ_{\lambda, \tau}(t) -\frac{1}{2}(\lambda-\tau)\, dt \big) K(t)\quad \text{with}\quad K(0)=I_N.$$
    Using stochastic calculus as described below,  It\^{o}'s formula and the expression \eqref{crochetZ} of the
quadratic variation of $Z_{\lambda,\tau}$ imply that $d(G_{\lambda, \tau}(t)K(t))=0$. This proves that $G_{\lambda, \tau}(t)K(t)=I_N$ for all $t\geq 0$.  The proof of the independence and stationary of the multiplicative increments is a straightforward adaptation of the proof for the two-parameter family studied by Kemp, see Section 5.1 of \cite{kemp2016large}.
\end{proof}
 
 \paragraph{Free multiplicative $(\lambda,\tau)$-Brownian motion.} Let $(\A, \varphi)$ be a $W^*$-noncommutative probability space, that is $\A$ is a von Neumann algebra and $\varphi$ is a faithful, normal, tracial state.  Consider the rotated elliptic element, which is of the form
\begin{equation}\label{RE element z}
    z=e^{i \theta}(ax+iby),
\end{equation}
where $x$ and $y$ are freely  independent standard semicircular elements, $\theta$ a real number and $a,b$ are real numbers that are not both zero.  We consider the positive variance parameter $\lambda$ and the complex covariance parameter $\tau$ defined by 
\[
\lambda=\varphi[z^\ast z]=a^2 +b^2 \quad \text{and} \quad \tau= \varphi[z^\ast z] - \varphi[z^2]= \lambda - e^{2i\theta}(a^2-b^2).
\]
For any $t\geq 0$, we construct the rotated elliptic Brownian motion $z_{\lambda,\theta}(t)$ as
\begin{equation*}
    z_{\lambda,\tau}(t)= e^{i \theta}(ax_t+iby_t), 
\end{equation*}
where $a, b$ and $\theta$ as chosen above and where $x_t$ and $y_t$ are freely independent semicircular Brownian motions. Finally, we consider the free multiplicative $(\lambda,\tau)$-Brownian motion $g_{\lambda, \tau}$ defined for any $t\geq 0$ as the solution of the free stochastic differential equation 
\begin{equation}\label{free Stoch Diff Eq}
    dg_{\lambda, \tau}(t)= g_{\lambda, \tau}(t)\big( i \, dz_{\lambda, \tau}(t) -\frac{1}{2}(\lambda-\tau)\, dt \big) \quad \text{with}\quad g_{\lambda, \tau}(0)=I_\A.
\end{equation}
The existence and uniqueness of the solution in $\A$ follow from a standard Picard iteration argument. We refer to \cite{Hall-Ho-23} for more details on rotated elliptic Brownian motions. We show in the following proposition that $(g_{\lambda, \tau}(t))_{t\geq 0}$ inherits the properties of $(G_{\lambda, \tau}(t))_{t\geq 0}$.
\begin{proposition}\label{Prop:stationarity}
    For any $\lambda\geq 0$ and $\tau\in \mathbb{C}$ such that $|\tau-\lambda|\leq \lambda$, $(g_{\lambda, \tau}(t))_{t\geq 0}$ is invertible for all $t\geq 0$, has freely independent stationary multiplicative increments.  
\end{proposition}
\begin{proof}The proof is again a straightforward adaptation of the two-parameter setting, see Section 5.2 of \cite{kemp2016large} for the details. In Section~\ref{proofofweakconv}, we will give more details about the inverse of $(g_{\lambda, \tau}(t))_{t\geq 0}$, which is also a solution of some free stochastic differential equation.
\end{proof}

 \paragraph{General stochastic calculus.}
We start by setting some notation and defining some maps that we use frequently in the sequel.
If $A$, $B$ and $C$ are three elements of an algebra, we use the following notation for the insertion map $
(A \otimes B) \sharp C = ACB $,
where $\otimes$ denotes the algebraic tensor product.

For any $t\geq 0$ and any two semi-martingales $N^{(1)}_t, N^{(2)}_t$ in $M_N(\mathbb{C})$, we define their quadratic covariation  by 
\begin{equation}\label{def: bracket}
    \langle N^{(1)}, N^{(2)}\rangle_t=\sum_{i,j=1 }^N \sum_{k,l =1 }^N\langle  N^{(1)}_{i,j},N^{(2)}_{kl} \,  \rangle_t E_{ij}\otimes E_{kl},\end{equation}
    where $(E_{ij})_{1\leq i,j\leq N}$ stands for the canonical basis of $M_N(\mathbb{C}).$
    This readily yields   that, for any adapted processes $A, B, C $ and $D$  in $M_N(\mathbb{C})$ and any process $V$ in $M_N(\mathbb{C})$,
    \begin{equation}\label{compute brackets} d\langle \int_0^{\cdot}A_udN^{(1)}_u B_u, \int_0^{\cdot}C_u dN^{(2)}_u D_u\rangle_t  \sharp V_t= A_t \left[d\langle N^{(1)}, N^{(2)}\rangle_t \sharp (B_tV_tC_t) \right] D_t.\end{equation}
   Now, by It\^o's formula, we have 
    for  semi-martingales $N^{(1)},\ldots, N^{(n)}$,  
   \begin{eqnarray}\label{Itoproduit}
   {d\left[N^{(1)}_t\cdots N^{(n)}_t\right]}\!\! &= \! \!&\sum_{i=1}^n N_t^{(1)}\cdots N_t^{(i-1)}dN^{(i)}_t N_t^{(i+1)}\cdots N_t^{(n)}\\&+& \sum_{i<j}N_t^{(1)}\cdots N_t^{(i-1)}\left(d<N^{(i)}, N^{(j)}>_t \sharp [N_t^{(i+1)}\cdots N_t^{(j-1)}]\right)N_t^{(j+1)}\cdots N_t^{(n)},\nonumber\end{eqnarray}
  and for any $z \in \mathbb{C}_+$, the resolvent $R_{N_t}(z)=(z-N_t)^{-1}$ of a semi-martingale $N_t$ in the set of Hermitian matrices  at point $z$ satisfies
\begin{equation}\label{Itoresolvante}
dR_{N_t}(z)= R_{N_t}(z) dN_t R_{N_t}(z) + R_{N_t}(z)\big(d<N, N>_t \sharp R_{N_t}(z)\big)R_{N_t}(z).     
\end{equation} 
 Let $k$ be a positive integer number.
In the paper, we will need to imbed sub-martingales $N_t$ from
$M_N(\mathbb{C})$ into $M_{Nk}(\mathbb{C})$:
$$N_t \rightarrow I_k \otimes  N_t.$$
 By simple computations, denoting by $id_k$ the identity map on $M_k(\C)$, it is easy to see that for any $Nk \times Nk$ matrix $V$,
\begin{equation}\label{tensor}\langle (I_k\otimes  N^{(1)}), (I_k \otimes  N^{(2)})\rangle_t\sharp V =\left[id_k \otimes 
    \left(\langle N^{(1)}, N^{(2)}\rangle_t\sharp\right)\right] V,\end{equation} the right hand side of the previous equality  meaning that, if $V=\sum_{i,j=1}^N V^{(i,j)}\otimes E_{ij}$, then  
    $$\left[id_k \otimes 
    \left(\langle N^{(1)}, N^{(2)}\rangle_t\sharp\right)\right] V=\sum_{i,j=1}^N V^{(i,j)}\otimes \left[\langle N^{(1)}, N^{(2)}\rangle_t\sharp E_{ij}\right]. $$
     \paragraph{Stochastic calculus on $GL_N(\mathbb{C})$.} 
  Let $H$\label{calculsto} be the $N \times N$ Wigner   Brownian motion as defined above then for any  $N \times N$ matricial process  $V_t$, we have $$d\langle H ,  H\rangle_t\sharp V_t= \tr_N (V_t)I_Ndt.$$
Moreover, considering the rotated elliptic Brownian motion $(Z_{\lambda,\tau}(t))_{t\geq 0}$  defined in \eqref{R-E-BM}, we have  
\begin{align}
    d\langle Z_{\lambda,\tau} ,  Z_{\lambda,\tau}\rangle_t\sharp V_t &=(\lambda-\tau)\tr_N (V_t)I_Ndt, \label{crochetZ}\\
    d\langle Z_{\lambda,\tau}^* ,  Z_{\lambda,\tau}^*\rangle_t\sharp V_t &=(\lambda-\bar{\tau}) \tr_N (V_t)I_Ndt, \label{crochetZstar}\\
    d\langle Z_{\lambda,\tau} ,  Z_{\lambda,\tau}^*\rangle_t\sharp V_t &=d\langle Z_{\lambda,\tau}^*,Z_{\lambda,\tau}\rangle_t\sharp V_t=
    \lambda \tr_N( V_t)I_Ndt.\label{crochetZZstar}
\end{align}
As explained in Section \ref{Section: idea of proof}, we shall study a multiplicative $(\lambda,\tau)$-Brownian motion $G_{\lambda, \tau}$ satisfying the multiplicative SDE
\begin{equation*}
    dG_{\lambda, \tau}(t)=\sigma \, G_{\lambda, \tau}(t)\big( i \, dZ_{\lambda, \tau}(t) -\frac{\sigma}{2}(\lambda-\tau)\, dt \big) \quad \text{with}\quad G_{\lambda, \tau}(0)=I_N,
\end{equation*}
for some $\sigma >0$. Note that $(G_{\lambda, \tau}^\ast(t))_{t\geq 0}$, $(G_{\lambda, \tau}^{-1}(t))_{t\geq 0}$ and $((G_{\lambda, \tau}^{-1})^\ast(t))_{t\geq 0}$ satisfy respectively the  multiplicative SDEs (see the proof of Proposition~\ref{SDEforinverse} for the SDE of $G_{\lambda, \tau}^{-1}$):
\begin{align*}
  dG_{\lambda, \tau}^\ast(t) &=-i {\sigma}    dZ_{\lambda, {\tau}}^{\ast}(t) G_{\lambda, \tau}^\ast (t) -\frac{\sigma^2}{2}(\lambda-\bar{\tau})G_{\lambda, \tau}^\ast (t) \, dt,
\\ 
  dG_{\lambda, \tau}^{-1}(t) &=-i \sigma  \, dZ_{\lambda, \tau}(t) G_{\lambda, \tau}^{-1}(t) -\frac{\sigma^2}{2}(\lambda-\tau)G_{\lambda, \tau}^{-1}(t) dt ,
  \\
  d(G_{\lambda, \tau}^{-1})^\ast(t) &=i \sigma \, (G_{\lambda, \tau}^{-1})^\ast(t) dZ_{\lambda, {\tau}}^\ast(t)  -\frac{\sigma^2}{2}(\lambda-\bar{\tau})(G_{\lambda, \tau}^{-1})^\ast(t) dt ,
\end{align*}
all initiated at the identity matrix. 

%%%%%%%%%%%%%%%%%%%%
We simplify the notation  and  denote  along the remainder of the paper by  $G_t:=G_{\lambda, \tau}(t), G_t^\ast:=G_{\lambda, \tau}^\ast(t), G_t^{-1}:=G_{\lambda, \tau}^{-1}(t)$ and $G_t^{-1,\ast}:=(G_{\lambda, \tau}^{-1})^\ast(t)$. Using \eqref{compute brackets}, we compute the quadratic covariation terms involving $G_t, G_t^\ast, G_t^{-1}$ and $G_t^{-1,\ast}$  for future reference, as they form the basis for our subsequent computations. 
\begin{lemma}\label{Lemma:Calcul QuadraticCov}
Let $(G_t)_{t\geq 0}$ be the multiplicative $(\lambda,\tau)$-Brownian motion defined above. Then, for any $N \times N$ matricial  process $V_t$, we have
{\small \begin{center}
\begin{tabular}{l|l}
$d\langle G  ,  G \rangle_t \sharp V_t = \sigma^2 (\tau-\lambda) G_t \tr_N (G_t V_t)dt$ & $d\langle G  ,  G^{-1} \rangle_t \sharp V_t = -\sigma^2 (\tau-\lambda)  \tr_N ( V_t)dt$
\\ 
$d\langle G  ,  G^\ast \rangle_t \sharp V_t = \sigma^2 \lambda G_t G_t^\ast\tr_N ( V_t)dt$ & $d\langle G  ,  G^{-1,\ast} \rangle_t \sharp V_t = -\sigma^2 \lambda G_t \tr_N (G_t^{-1,\ast} V_t)dt$\\
$d\langle G^{-1}  ,  G \rangle_t \sharp V_t = -\sigma^2 (\tau-\lambda)  \tr_N ( V_t)dt$ & $d\langle G^{-1}  ,  G^{-1} \rangle_t \sharp V_t = \sigma^2 (\tau-\lambda) G_t^{-1} \tr_N (G_t^{-1} V_t)dt$
\\ 
$d\langle G^{-1}  ,  G^\ast \rangle_t \sharp V_t = -\sigma^2 \lambda  G_t^\ast\tr_N (G_t^{-1} V_t)dt$ & $d\langle G^{-1}  ,  G^{-1,\ast} \rangle_t \sharp V_t = \sigma^2 \lambda  \tr_N (G_t^{-1} V_tG_t^{-1,\ast})dt$\\
$d\langle G^\ast  ,  G \rangle_t \sharp V_t = \sigma^2 \lambda  \tr_N (G_t^\ast V_tG_t)I_Ndt$ & $d\langle G^\ast  ,  G^{-1} \rangle_t \sharp V_t = -\sigma^2 \lambda G_t^{-1} \tr_N ( G_t^\ast V_t)dt$
\\ 
$d\langle G^\ast  ,  G^\ast \rangle_t \sharp V_t = {\sigma}^2 (\bar{\tau}-\lambda)  G_t^\ast\tr_N (G_t^\ast V_t)dt$ & $d\langle G^\ast  ,  G^{-1,\ast} \rangle_t \sharp V_t = -{\sigma}^2 (\bar{\tau}-\lambda)  \tr_N ( G_t^\ast V_t G_t^{-1,\ast})dt$\\
{\small $d\langle G^{-1,\ast}  ,  G \rangle_t \sharp V_t =  -\sigma^2 \lambda G_t^{-1,\ast} \tr_N (G_t V_t)dt$} & {\small
$d\langle G^{-1,\ast}  ,  G^{-1} \rangle_t \sharp V_t = \sigma^2 \lambda G_t^{-1,\ast} G_t^{-1} \tr_N ( V_t)dt$}
\\ 
{\small $d\langle G^{-1,\ast}   ,   G^\ast \rangle_t \sharp V_t = -{\sigma}^2 (\bar{\tau}-\lambda)  \tr_N ( V_t)dt$ } & {\small $d\langle G^{-1,\ast}  ,    G^{-1,\ast} \rangle_t  \sharp V_t = {\sigma}^2 (\bar{\tau}-\lambda)  G_t^{-1,\ast} \tr_N \! ( G_t^{-1,\ast}V_t) dt$}
\end{tabular}

\end{center} }
\end{lemma}

%%%%%%%%%%%%%%%%%%%%%%%%%%%%
 Let $k$ be any positive integer. For the proof, we also need to compute the quadratic covariation terms involving $I_k \otimes G_t, I_k \otimes  G_t^\ast, I_k \otimes  G_t^{-1}$ and $I_k \otimes  G_t^{-1,\ast}$. Using \eqref{tensor}, such quantities can be anticipated based on the computations from Lemma~\ref{Lemma:Calcul QuadraticCov}. To illustrate this clearly, we provide two examples. Recalling the definition of the quadratic covariation in \eqref{def: bracket}, then by simple computations, it is easy to see that for any $Nk \times Nk$ matricial process $V_t$, we have 
\begin{align*}
d\langle (I_k \otimes G)  ,  (I_k \otimes G) \rangle_t \sharp V_t &= \sigma^2 (\tau-\lambda) \Big[(id_k \otimes \tr_N)\big((I_k \otimes G_t)V_t\big) \Big] \otimes G_t dt,\\
d\langle (I_k \otimes G)  ,  (I_k \otimes G^\ast) \rangle_t \sharp V_t &= \sigma^2 \lambda \Big[(id_k \otimes \tr_N)\big(V_t\big) \Big] \otimes G_tG_t^\ast dt.
\end{align*}

\paragraph{Helffer-Sj\"ostrand's representation formula.}
\label{HS}
 Let $f$ be a function of class $C^{k+1}(\mathbb R)$ with compact support and denote by $F_k(f): \mathbb{C}^+ \rightarrow \mathbb{C}$ its so-called almost analytic extension defined as 
 \begin{equation*} 
F_k(f)(x+iy) = \sum_{l=0}^k \frac{(iy)^l}{l!} f^{(l)}(x) \chi(y),
\end{equation*}
where $\chi : \mathbb R \to \mathbb R^+ $ is a smooth compactly supported function such that $\chi \equiv 1$ in a neighborhood of 0. The Helffer-Sj\"ostrand's representation formula yields that
\begin{equation} \label{HS}
 f(u) = \frac{1}{\pi} \int_{\mathbb C} \bar{\partial} F_k(f)(z) (u-z)^{-1} d^2z,
 \end{equation}
where $\bar{\partial} = \frac{1}{2} (\partial_x +i \partial_y)$ and $d^2 z$ is the Lebesgue measure on $\mathbb C$. It is easy to see that the function $F_k(f)$ coincides with $f$ on the real axis, extends it to the complex plane, and satisfies 
\begin{equation*} 
\bar{\partial} F_k(f)(x+iy)=\frac{1}{2}  \frac{(iy)^k}{k!} f^{(k+1)}(x) \chi(y)+ \frac{i}{2} \sum_{l=0}^k \frac{(iy)^l}{l!} f^{(l)}(x) \chi'(y).\end{equation*}
Thus, since $\chi \equiv 1$ in a neighborhood of 0,  we have that, in a neighborhood of the  real axis, 
\begin{equation*}
 \bar{\partial} F_k(f)(x+iy) =\frac{1}{2} \frac{(iy)^k}{k!} f^{(k+1)}(x)  = O(|y|^k) \mbox { as } y \rightarrow 0. \end{equation*} 
\begin{lemma}\label{applicationHS}
Let   $h$ be  a measurable function on $\C\setminus \R$  which satisfies
\begin{equation}\label{nestimgdif}
\vert h(z)\vert \leq Q(\vert \Im  z\vert ^{-1})
\end{equation} 
\noindent for some polynomial $Q$ with nonnegative coefficients and degree $q$. Let  $(f_n)_{n\geq 0}$ be a  sequence of functions    of class $C^{q+1}(\mathbb R)$ with compact supports all included in $]-K,K[$, for some $K>0,$ and such that 
$$\sup_n \sup_{x\in[-K;K]} |f_n^{(p)}(x)| =C_p<+\infty, \text{~for any ~} p\leq q+1.$$
Then, 
there exists a constant $C$ which depends on $h$ only through  the polynomial $Q$ above and on $(f_n)_{n\geq 0}$ through the constant numbers $C_i$'s, such that for all $n$,
$$\left|\int_{\mathbb C} \bar{\partial} F_q(f_n)(z) h(z) d^2z\right|\leq C.$$
\end{lemma}
\begin{proof}
    The previous integral on $\mathbb{C}$ is actually integral on a compact set included in $[-K,K]\times \mathcal{C}_\chi$ where $\mathcal{C}_\chi$ denotes the support of $\chi$.
Moreover, since $\chi\equiv 1$ around zero, there exists $0<L<1$ such that ${\chi}^{'}\equiv 0$ on $[-L;L]\subset \mathcal{C}_\chi$ and then $\bar{\partial} F_q(f_n)(x+iy)=\frac{1}{2} \frac{(iy)^q}{q!} f_n^{(q+1)}(x)$ on $[-L;L]$. 
Thus,
\begin{eqnarray*}\int_{\mathbb C} \bar{\partial} F_q(f_n)(z) h(z) d^2z
&=&\int_{[-K,K]}\int_{\mathcal{C}_\chi\setminus[-L;L]}\bar{\partial} F_q(f_n)(z) h(z)d^2z\\&&+
\int_{[-K,K]}\int_{[-L;L]}\frac{1}{2}h(x+iy) \frac{(iy)^q}{q!} f_n^{(q+1)}(x)dx dy,\end{eqnarray*}
with $|h(z)|\leq Q(1/L) $ for all $z \in [-K,K]\times\mathcal{C}_\chi\setminus[-L;L]$ and $|h(x+iy)|\leq  c_Q \frac{1}{|y|^q}
$ for all $z=x+iy \in [-K,K]\times([-L;L]\setminus\{0\})$, where $c_Q$ is a constant only  depending on $Q$. The result readily follows.
\end{proof}

 \section{Idea of the proof}\label{Section: idea of proof}
 In this section, we outline the key points of the proof.
 Let us first explain how, by  using the fact that the processes have independent and stationary multiplicative increments,  establishing the multi-time statement in Theorem \ref{maintheo}, is reduced to   proving the statements for a fixed time $t \geq 0$. With this aim, we let, for any $\ell=1,\dots, p$,  $(Z_t^{(\ell)})_{t\geq0}:=(Z_{\lambda, \tau}^{(\ell)}(t))_{t\geq0}$ be a family of independent rotated elliptic Brownian motions and consider the  multiplicative $(\lambda,\tau)$-Brownian motions $(G_t^{(\ell)})_{t\geq0}:=(G_{\lambda, \tau}^{(\ell)}(t))_{t\geq0}$ defined for any $t\geq 0$ as the solution of the  stochastic differential equation 
\begin{equation}\label{Stoch Diff Eq sigma}
    dG_t^{(\ell)}=\sigma_\ell \, G_t^{(\ell)} \big( i \, dZ_t^{(\ell)}  -\frac{\sigma_\ell}{2}(\lambda-\tau)\, dt \big) \quad \text{with}\quad G_0^{(\ell)}=I_N,
\end{equation}
for some positive real number $\sigma_\ell$. Similarly,  we let $(z_t^{(\ell)})_{t\geq0}:=(z_{\lambda, \tau}^{(\ell)}(t))_{t\geq0}$ be a family of freely independent free rotated elliptic Brownian motions and consider the free multiplicative $(\lambda,\tau)$-Brownian motions $(g_t^{(\ell)})_{t\geq0}:=(g_{\lambda, \tau}^{(\ell)}(t))_{t\geq0}$ defined for any $t\geq 0$ as the solution of the free stochastic differential equation 
\begin{equation}\label{free Stoch Diff Eq with sigma}
    dg_t^{(\ell)} =\sigma_\ell \, g_t^{(\ell)} \big( i \, dz_t^{(\ell)}  -\frac{\sigma_\ell}{2}(\lambda-\tau)\, dt \big) \quad \text{with}\quad g_t^{(\ell)} =I_\A,
\end{equation}
for the same choice of $\sigma_\ell$ as in \eqref{Stoch Diff Eq sigma}.
Finally, for the simplicity of notation, we denote by $\G_t$ and $\g_t$ the $p$-tuples defined for any $t\geq 0$ by
\begin{equation*}
    \G_t:= \big(G_t^{(1)} , \dots, G_t^{(p)} \big)
    \quad \text{and} \quad
    \g_t:= \big(g_t^{(1)} , \dots, g_t^{(p)} \big).
\end{equation*}

 Without loss of generality, we can assume that the times in Theorem \ref{maintheo} are ordered with $t_0=0\leq t_1 \leq t_2 \leq \ldots \leq t_p$. Now, let us set the time-increments $\sigma_\ell^2=t_{\ell}-t_{\ell-1}$ for any $1\leq \ell\leq p$. By a simple change of time in the stochastic equation, we know that $G^{(\ell)}_1$ has the same distribution as $G_{ t_{\ell}-t_{\ell-1}}$. Because the multiplicative increments of the process $(G_t)_{t\geq 0}$ are stationary and independent, it yields that
$(G^{(\ell)}_1)_{1\leq \ell \leq p}$ has the same distribution as $((G_{ t_{\ell-1}})^{-1}G_{ t_{\ell}})_{1\leq \ell \leq p}$.
The same properties hold dealing with the free multiplicative $(\lambda,\tau)$-Brownian motion provided we replace independence by freeness.
Thus, this reduces the proof of Theorem \ref{maintheo-sp} to establishing  the following result for fixed $t\geq 0$.
 \begin{theorem}\label{theo: Strong cov at time t}
 Let $A^N:=(A_1^N,  \dots, A_r^N,A_1^{N\ast},  \dots, A_r^{N\ast})$  be an $r$-tuple of $N \times N$ deterministic matrices that are uniformly bounded over $N$. Let $t\geq 0$. Then, for any $\delta>0$ and   any noncommutative self-adjoint polynomial $Q$ in $4p+2r$ indeterminates, almost surely,  we have for $N$ large  enough
$$sp\big(Q(\G_t,\G_t^\ast,\G_t^{-1}, \G_t^{-1,\ast},A^N)\big)
\subset sp\big(Q(\g_t,\g_t^\ast,\g_t^{-1}, \g_t^{-1,\ast},A^N)\big) +(-\delta, \delta).$$
 \end{theorem}

To proof Theorem \ref{theo: Strong cov at time t},  we use the   idea 
also present in the approach of \cite{haagerup2005new}, which consists in understanding the spectrum  by  obtaining a sharp estimate  of the mean Stieltjes transform of the empirical spectral measure as we explain now.
 We let $\rho\in{\cal C}^\infty(\mathbb R,\mathbb R)$ be such that $\rho\geq 0$, its support is included in $[-1,1]$ and $\int_\mathbb R \rho (x)\,{\rm d}x=1$.  Now,  we let $f_{N,\delta}$ be the function defined for any $x\in\mathbb{R}$ by
\begin{equation}\label{def:f delta}
  f_{N,\delta} (x)=\int _\mathbb{R} \1 _{\K_N(\delta)}(y)\rho _{\frac{\delta }{2}}(x-y)\,{\rm d}y,  
\end{equation}
where
%%%%%%%%%%%%%%%%
\begin{equation}\label{def:K(delta)}
\K_N(\delta )=\{x \in \R \, | \, {\rm dist}\big(x,{\rm sp}(PP^*(\g_t,\g_t^\ast,\g_t^{-1}, \g_t^{-1,\ast}, A^N))\big)\leq\delta\}. \end{equation}
and
$\rho _{\frac{\delta }{2}}(x)=\frac{2}{\delta}\rho\left(\frac{2x}{\delta }\right)
$. Note that the function $f_{N,\delta}$ is in ${\cal C}^\infty (\mathbb{R}, \mathbb{R})$, $f_{N,\delta}\equiv 1$ on $\K_N(\frac{\delta }{2})$ and its support is included in $\K_N(2\delta )$. Moreover, since there exists \( C > 0 \) such that the spectrum of $PP^*(\g_t,\g_t^\ast,\g_t^{-1}, \g_t^{-1,\ast},A^N)$ is contained in \([ -C, C ]\), the support of \( f_{N,\delta} \) is contained in \([ -C-2, C+2 ]\).
Moreover, for any $p>0$, 
$$\sup_{x\in [-C-2,C+2]} |f_{N,\delta}^{(p)}(x)| \leq 
\sup_{x\in [-C-2,C+2]} \int_{-C-1}^{C+1}|
\rho^{(p)}_{\frac{\delta }{2}}(x-y)| dy\leq C_p(\delta).$$

A key point at this stage is the following estimate which is a corollary of Theorem \ref{theo:fundamental} which we prove in Section \ref{Section:fundamental result}: for any $z \in \mathbb{C}\setminus \mathbb{R}$,
\begin{equation}\label{singleresolvent}\E \tr_{N} \big[\big(z-PP^*(\G_t,\G_t^\ast,\G_t^{-1}, \G_t^{-1,\ast},A^N)\big)^{-1}\big]=\varphi_N \big[\big(z-PP^*(\g_t,\g_t^\ast,\g_t^{-1}, \g_t^{-1,\ast},A^N)\big)^{-1}\big] +O\left(\frac{1}{N^2}\right),\end{equation} with $O(1/N^2) \leq \frac{1}{N^2}{T(|\Im z|^{-1})}$ for some polynomial $T$ with nonnegative coefficients. By the above bound and Lemma \ref{applicationHS} around the Helffer-Sj\"ostrand's representation formula, we infer that 
\begin{equation}\label{est: f delta G - g}
\E {\rm \tr}_N \big[f_{N,\delta}(PP^*(\G_t,\G_t^\ast,\G_t^{-1}, \G_t^{-1,\ast},A^N))\big]-\varphi_N\big[f_{N,\delta}(PP^*(\g_t,\g_t^\ast,\g_t^{-1}, \g_t^{-1,\ast}, A^N))\big]
=O_{\delta }\left(\frac{1}{N^2}\right).
\end{equation}
Since the function $\psi_{N,\delta}\equiv 1-f_{N,\delta}$ also satisfies \eqref{est: f delta G - g} and the fact that $\psi_{N,\delta}\equiv 0$ on the spectrum of ${PP^*(\g_t,\g_t^\ast,\g_t^{-1}, \g_t^{-1,\ast},A^N)}$, we deduce that 
\begin{equation}\label{psi2} 
\E   \tr _N \big[\psi_{N,\delta}(PP^*(\G_t,\G_t^\ast,\G_t^{-1}, \G_t^{-1,\ast},A^N))\big]=O_{\delta }\left(\frac{1}{N^2}\right).
\end{equation} 
We  further control the variance and obtain that
\begin{equation}\label{variancefdelta}
\Var{\left[\tr _N(h_N(PP^{*}(\G_t,\G_t^\ast,\G_t^{-1}, \G_t^{-1,\ast}, A^N)))\right]}%=\mathbf{V}
=O\left(\frac{1}{N^4}\right),
\end{equation}
for $h_N=\psi_{N,\delta}$ and $h_N=f_{N,\delta}$ for any $0<\delta<1$. The proof of \eqref{variancefdelta} also relies on Theorem \ref{theo:fundamental} and Helffer-Sj\"ostrand's representation formula. Section \ref{Section:Estimating Variance} is dedicated to its proof, where the estimates are explicitly computed and illustrated. Note that controlling the variance is the main motivation for establishing a more general estimate than \eqref{singleresolvent} in Theorem \ref{theo:fundamental}. Now, denoting by
\[\Psi_{N, \delta }:=\tr _N (\psi_{N,\delta} (PP^*(\G_t,\G_t^\ast,\G_t^{-1}, \G_t^{-1,\ast}, A^N)))
\quad \text{and} \quad
\Omega _{N, \delta }=\{ \left| \Psi_{N, \delta}-\E \left(\Psi_{N, \delta }\right)\right| > N^{-4/3}\},
\]
then by Chebyshev's inequality and the above variance estimate, we obtain that $$
\mathbb{P}(\Omega _{N, \delta })
=O_{\delta }\left(N^{-4/3}\right)
.$$
By Borel-Cantelli lemma, we deduce that, almost surely for all large $N$, 
\begin{equation}\label{concentrezn} 
\left|\Psi_{N, \delta }-\E \left(\Psi_{N, \delta }\right)\right| \leq N^{-4/3}.
\end{equation}
From \eqref{psi2} and \eqref{concentrezn}, we deduce that there exists some constant $C_\delta$ such that, almost surely for all large $N$,
$$
\left| \Psi_{N, \delta }\right|\leq C_\delta N^{-2}+ N^{-4/3}=N^{-1}( C_\delta N^{-1} +N^{-1/3}).
$$
 Since \( \psi_{N, \delta} \geq \mathbf{1}_{\mathbb{R} \setminus \K_N(2\delta)} \), it follows that, almost surely, for all large \( N \), the number of eigenvalues of \( PP^*(\G_t,\G_t^\ast,\G_t^{-1}, \G_t^{-1,\ast},A^N) \) that lie in \( \mathbb{R} \setminus \K_N(2\delta) \) is less than \( \left( N^{-1/3} + C_\delta N^{-1} \right) \). Thus, almost surely, for all large \( N \), the number of eigenvalues of \( PP^*(\G_t,\G_t^\ast,\G_t^{-1}, \G_t^{-1,\ast},A^N) \) in \( \mathbb{R} \setminus \K_N(2\delta) \) must be zero. Therefore, almost surely, for all large \( N \), the spectrum of \( PP^{*}(\G_t,\G_t^\ast,\G_t^{-1}, \G_t^{-1,\ast},A^N) \) is contained in \( \K_N(2\delta)\).\\
 %%%%%%%%%%%%%%%%%
 
While the primary challenge and contribution of this work lie in proving strong convergence, the following weak convergence must first be established as a crucial step toward this goal.
\begin{theorem}\label{th:weakconv}Let $P_1,\ldots,P_n$ be $n$ noncommutative polynomials in $4p$ noncommutative indeterminates, and let $t\geq 0$. There exists a constant $C>0$ such that, for all $0\leq s \leq t$, we have
$$\left|\mathbb{E}\left[\prod_{i=1}^n\tr_N(P_i(\G_s,\G_s^\ast,\G_s^{-1}, \G_s^{-1,\ast}))\right]-\prod_{i=1}^n\varphi\left[P_i(\g_s,\g_s^\ast,\g_s^{-1}, \g_s^{-1,\ast})\right]\right|\leq \frac{C}{N^2}.$$
In particular, by Borel-Cantelli, for any noncommutative polynomial $P$ in $4p$ noncommutative indeterminates, and any time $t\geq 0$, we have almost surely the convergence
$$\lim_{N\to \infty}\tr_N(P(\G_s,\G_s^\ast,\G_s^{-1}, \G_s^{-1,\ast}))=\varphi\left[P(\g_s,\g_s^\ast,\g_s^{-1}, \g_s^{-1,\ast})\right].$$
\end{theorem}
\begin{proof}
    The proof follows \cite{cebron2013free}, and is postponed to Section~\ref{proofofweakconv}.
\end{proof}

 %%%%%%%%%%%%%%%
  
 The core of this paper lies in the proofs of Theorems \ref{theo: variance estimate} and \ref{theo:fundamental}. In particular, Theorem \ref{theo:fundamental} serves as a key point, where we establish a sharp estimate for the mean trace of an arbitrary product of resolvents and monomials involving independent tuples of multiplicative Brownian motions and a bounded sequence of deterministic matrices. The proof of Theorem \ref{theo:fundamental} builds on ideas from the interpolation technique developed by Collins, Guionnet, and Parraud \cite{CGParraud}, later extended by Parraud \cite{ParraudHaar, ParraudGUE, parraud-HaarTensor}, to derive asymptotic expansions for smooth functions of polynomials in deterministic matrices and some iid matrices. In particular, this framework enables the analysis of the operator norm. 
  In the approach we follow, we notice first that, for any $k$, the distribution of  $I_k\otimes  \G_t$ in $\left(M_{Nk}(\mathbb{C}), \tr_{Nk}\right)$ coincides with  the  distribution of  $  \G_t$ in $\left(M_{N}(\mathbb{C}), \tr_{N}\right)$ and we view $\g_t$  as the asymptotic distribution of  an independent tuple of Brownian motions $K_t$ in $GL_{Nk}(\mathbb{C})$ as $k\rightarrow \infty$. The idea is then to interpolate between $I_k\otimes  \G_t$ and $K_t$ and let $k$ grow to infinity. 
  
  Note that another interpolation technique was introduced in~\cite{chen2024newapproachstrongconvergenceI}, but it does not apply to our setting since the expected trace of a polynomial in the multiplicative Brownian motion is not a rational function of $1/N$. For example, in the case $(\lambda, \tau) = (1,0)$, it follows from~\cite[Example 3.5]{driver2013large} that  
\[
\mathbb{E}[\tr_N(G_t^2)] = e^{-t} \cosh(t/N) - e^{-t}N \sinh(t/N).
\]  

 We end this section by explaining how one can deduce Corollary \ref{maintheo} from Theorem \ref{maintheo-sp}. First of all, 
it is sufficient to prove that for any polynomial the convergence of the norm  holds almost
surely (instead of almost surely the convergence holds for all polynomials) by a usual approximation argument using polynomials with rational coefficients.
 
 In the absence of deterministic matrices $A^N$, using the same argument explained at the very beginning of this section, the first multi-time statement of  Corollary \ref{maintheo} readily follows from  Theorem~\ref{th:weakconv}. Now, using the invariance of $G$ by unitary conjugation and applying the asymptotic freeness of unitarily invariant matrices from deterministic matrices (see for example \cite[Theorem 3.5]{cebron2022freeness}) to the imaginary parts and real parts of  the $G_{t_i}$'s and $G_{t_i}^{-1}$'s, we can deduce the almost sure asymptotic freeness of $(G_{t_i}, G_{t_i}^{-1}; i=1,\ldots,p)$ and $A^N$.
By following the proof of Lemma 7.2 in \cite{haagerup2005new}, one can then prove  that  almost surely 
\begin{eqnarray*}\lefteqn{\liminf_{N\rightarrow +\infty} \|P(G_{t_1},G_{t_1}^*, G_{t_1}^{-1}, G_{t_1}^{-1,\ast}\ldots,G_{t_p},G_{t_p}^*, G_{t_p}^{-1}, G_{t_p}^{-1,\ast}, A^N)\|}\\&\geq& \| P(g_{t_1},g_{t_1}^*,  g_{t_1}^{-1}, g_{t_1}^{-1,\ast},\ldots,g_{t_p},g_{t_p}^*,  g_{t_p}^{-1}, g_{t_p}^{-1,\ast}, \bf a)\|.\end{eqnarray*}
 According to Pisier \cite{Pisier-StrongCon}, since  $A_N$ converges strongly towards $\bf a$ which is free from  $(g_{t_i}, g_{t_i}^*, g_{t_i}^{-1}, g_{t_i}^{-1,\ast}; i=1,\ldots,p)$, then we can deduce that $(g_{t_1},g_{t_1}^*, g_{t_1}^{-1}, g_{t_1}^{-1,\ast},\ldots,g_{t_p},g_{t_p}^*, g_{t_p}^{-1}, g_{t_p}^{-1,\ast}, A^N)$ converges strongly towards $(g_{t_1},g_{t_1}^*, g_{t_1}^{-1}, g_{t_1}^{-1,\ast},\ldots,g_{t_p},g_{t_p}^*, g_{t_p}^{-1}, g_{t_p}^{-1,\ast}, \bf a)$. Therefore, 
according to Proposition \ref{scCollinsMale},  for large N, $$sp\big(PP^*(g_{t_1},g_{t_1}^*, g_{t_1}^{-1}, g_{t_1}^{-1,\ast},\ldots,g_{t_p},g_{t_p}^*, g_{t_p}^{-1}, g_{t_p}^{-1,\ast}, A^N)\big) \subset sp\big(PP^*(g_{t_1},g_{t_1}^*,\ldots,g_{t_p},g_{t_p}^*,\bf a)\big) + (-\delta, \delta).$$
Finally, by Theorem \ref{maintheo-sp},
almost surely, for $N$ large  enough
\begin{eqnarray*} \lefteqn{
sp\big(PP^*(G_{t_1},G_{t_1}^*, G_{t_1}^{-1}, G_{t_1}^{-1,\ast},\ldots,G_{t_p},G_{t_p}^*, 
G_{t_p}^{-1}, G_{t_p}^{-1,\ast}, A^N)\big) }\\& \subset &sp\big(PP^*(g_{t_1},g_{t_1}^*, g_{t_1}^{-1}, g_{t_1}^{-1,\ast},\ldots,g_{t_p},g_{t_p}^*, g_{t_p}^{-1}, g_{t_p}^{-1,\ast}, {\bf a})\big) + (-2\delta, 2\delta),\end{eqnarray*}
which easily implies 
that \begin{eqnarray*}\lefteqn{\limsup_{N\rightarrow +\infty} \|P(G_{t_1},G_{t_1}^*, G_{t_1}^{-1}, G_{t_1}^{-1,\ast}\ldots,G_{t_p},G_{t_p}^*, G_{t_p}^{-1}, G_{t_p}^{-1,\ast}, A^N)\|}\\&\leq& \| P(g_{t_1},g_{t_1}^*,  g_{t_1}^{-1}, g_{t_1}^{-1,\ast},\ldots,g_{t_p},g_{t_p}^*,  g_{t_p}^{-1}, g_{t_p}^{-1,\ast}, \bf a)\| ~~\text{a.s.}\end{eqnarray*}
 
\section{Estimating the variance}\label{Section:Estimating Variance}
This section is dedicated to the proof of Theorem \ref{theo: variance estimate} in which we obtain a  quantitative estimate on general variances and the proof of the particular case \eqref{variancefdelta}. The statement of Theorem \ref{theo: variance estimate} involves uniformly bounded  deterministic matrices $A^N$. However, we will illustrate the proof for functions in $(\G_t,\G_t^\ast,\G_t^{-1}, \G_t^{-1,\ast})$ only as  $A^N$ does not affect any argument in the proof.      
\paragraph{General case.}
Let $P$ be any noncommutative polynomial in $4p$ noncommutative indeterminates.
Set $\mathcal{G}_t:=(\G_t,\G_t^\ast,\G_t^{-1}, \G_t^{-1,\ast})$.
Then by  Helffer-Sj\"ostrand's representation formula \eqref{HS}, we write for some $k$ that we choose later on:
\begin{align}
&\Var{\left[\tr _N(f(PP^{*}(\mathcal{G}_t)))\right]}\nonumber
\\ & =\frac{1}{\pi^2} \E \left[\int_{\mathbb C} \bar{\partial} F_k(f)(z_1) \tr_N(PP^*(\mathcal{G}_t)-z_1)^{-1} d^2z_1\int_{\mathbb C} \bar{\partial} F_k(f)(z_2)\tr_N (PP^*(\mathcal{G}_t)-z_2)^{-1} d^2z_2\right]\nonumber \\
&-\frac{1}{\pi^2} \E \left[\int_{\mathbb C} \bar{\partial} F_k(f)(z_1) \tr_N(PP^*(\mathcal{G}_t)-z_1)^{-1} d^2z_1\right]\E \left[\int_{\mathbb C} \bar{\partial} F_k(f)(z_2)\tr_N (PP^*(\mathcal{G}_t)-z_2)^{-1} d^2z_2\right]\nonumber
\\& =\frac{1}{\pi^2} \int_{\mathbb C} \int_{\mathbb C}\bar{\partial} F_k(f)(z_1)\bar{\partial} F_k(f)(z_2)\Cov\big[\tr_N(z_1-PP^*(\mathcal{G}_t))^{-1}, \tr_N(z_2-PP^*(\mathcal{G}_t))^{-1}\big]d^2z_1d^2z_2.\label{methodeHS}
\end{align}
This reduces the problem to estimating the covariance of the normalized traces of the resolvent of \(PP^\ast(\mathcal{G}_t)\) at points \(z_1\) and \(z_2\) in \(\mathbb{C}\). To achieve this, we first recall the multiplicative stochastic differential equations in \eqref{Stoch Diff Eq sigma} relative to $\cG_t$ and proceed by interpolation. Specifically, for \(\ell = 1, \dots, p\), we consider independent copies \(\widetilde{Z}_{\lambda,\tau}^{(\ell)}(t)\) and $\hat{Z}_{\lambda,\tau}^{(\ell)}(t)$ of \(Z_{\lambda,\tau}^{(\ell)}(t)\), as defined in \eqref{R-E-BM} and let \(\widetilde{G}_{\lambda,\tau}^{(\ell)}\) and \(\hat{G}_{\lambda,\tau}^{(\ell)}\) be the independent multiplicative \((\lambda,\tau)\)-Brownian motions, defined for any \(t \geq 0\) as the solution of the multiplicative stochastic differential equation 
\begin{equation}\label{stoch diff eq hat and tilde}
  \begin{array}{cc}
  d\widetilde{G}_{\lambda, \tau}^{(\ell)}(t)= \, \big( i\sigma_\ell \, d\widetilde{Z}_{\lambda, \tau}^{(\ell)}(t) -\frac{\sigma_\ell^2}{2}(\lambda-\tau)\, dt \big)\widetilde{G}_{\lambda, \tau}^{(\ell)}(t) \quad &\text{with}\quad \widetilde{G}_{\lambda, \tau}^{(\ell)}(0)=I_N,
    \\ d\hat{G}_{\lambda, \tau}^{(\ell)}(t)= \, \big( i \sigma_\ell\, d\hat{Z}_{\lambda, \tau}^{(\ell)}(t) -\frac{\sigma_\ell^2}{2}(\lambda-\tau)\, dt \big)\hat{G}_{\lambda, \tau}^{(\ell)}(t) \quad &\text{with}\quad G_{\lambda, \tau}^{(\ell)}(0)=I_N. 
  \end{array}
    \end{equation}
Note that the difference between the above stochastic differential equations and the one satisfied by ${G}_{\lambda, \tau}^{(\ell)}(t)$ in \eqref{Stoch Diff Eq sigma} lies in the sense of multiplication, a choice that will become clear later in the proof. Finally, we denote by 
\begin{align*}
 \wtcG_t:=(\wtbfG_t,\wtbfG_t^\ast,\wtbfG_t^{-1}, \wtbfG_t^{-1,\ast})
\quad &\text{with} \quad
\wtbfG_t=(\wtG_{\lambda,\tau}^{(1)},\dots , \wtG_{\lambda,\tau}^{(p)}), 
\\  \text{and} \quad\hat{\mathcal{G}}_t:=(\hbfG_t,\hbfG_t^\ast,\hbfG_t^{-1}, \hbfG_t^{-1,\ast}) 
\quad &\text{with} \quad
\hbfG_t=(\hG_{\lambda,\tau}^{(1)},\dots , \hG_{\lambda,\tau}^{(p)}).
\end{align*}
For any $m\geq 1$ and any $m$-tuples of noncommutative elements $a=(a_1, \dots, a_{m})$ and $b=(b_1, \dots, b_{m})$, we define the product $ab$ to be the $m$-tuple $(a_1 b_1, \dots, a_{m}b_{m})$.

\begin{lemma}\label{cov}
Let $P$ be any noncommutative polynomial in $4p$ noncommutative indeterminates and let $\cG_t, \wtcG_t$ and $\hcG_t$ be defined as above. Then, for any $z_1$, $z_2 \in \mathbb{C}\setminus \mathbb{R}$, the covariance  \[\Cov\big[\tr_N(z_1-PP^*(\mathcal{G}_t))^{-1}, \tr_N(z_2-PP^*(\mathcal{G}_t))^{-1}\big]\]  is a finite sum  of terms of the form 
 \[
 \frac{1}{N^2} \int_0^t \E\tr_N\big[( \widetilde{R}_{t,u}^{z_1})^2 Q_1 (\hat{R}_{t,u}^{z_2})^2 Q_2\big]du,
 \]
 where $Q_1$ and $Q_2$ are monomials in $\cG_u$,
 $\wtcG_{t-u}$ and $\hcG_{t-u}$  while $\widetilde{R}_{t,u}^{z}$ and $\hat{R}_{t,u}^{z}$ are the resolvents defined by 
 \begin{align*}
    \widetilde{R}_{t,u}^{z}:= (z-PP^\ast(\cG_u\wtcG_{t-u} ))^{-1}&  \quad \text{and} \quad \hat{R}_{t,u}^z:=(z- PP^\ast( \cG_u \hcG_{t-u}))^{-1}.
\end{align*}
\end{lemma}

\begin{proof}[Proof of Lemma \ref{cov}]
We start by setting some notation. Fix $t\geq 0$, then for any monomial $Q$ in noncommuting variables and any $u,v\in [0,t]$, we denote by 
\begin{align*}
    \widetilde{Q}_{t,u}:= Q(\cG_u\wtcG_{t-u} )&, \quad \hat{Q}_{t,u}:= Q(\cG_u\hcG_{t-u} ),
    \\ \widetilde{Q}_{t,u,v}:= Q(\cG_v\wtcG_{t-u} )&, \quad \hat{Q}_{t,u,v}:= Q(\cG_v\hcG_{t-u} ),
\end{align*}
  and  define for any $z \in \mathbb{C}^+$,  the associated resolvents 
\begin{align*}
    \widetilde{R}_{t,u}^{z}:= (z-PP^\ast(\cG_u\wtcG_{t-u} ))^{-1}:=R^{z}(\cG_u\wtcG_{t-u} )&, \ \hat{R}_{t,u}^z:=(z- PP^\ast(\cG_u\hcG_{t-u} ))^{-1}:=R^{z}(\cG_u\hcG_{t-u} ),
    \\ \widetilde{R}_{t,u,v}^z:=(z- PP^\ast(\cG_v\wtcG_{t-u} ))^{-1}:=R^{z}(\cG_v\wtcG_{t-u} )&, \ \hat{R}_{t,u,v}^z:=(z- PP^\ast(\cG_v\hcG_{t-u} ))^{-1}:=R^{z}(\cG_v\hcG_{t-u} ).
\end{align*}
We note that $\widetilde{Q}_{t,t}=\hat{Q}_{t,t}=Q(\cG_t)$, while $\widetilde{Q}_{t,0}$ and $\hat{Q}_{t,0}$ are two independent copies of $Q(\cG_t)$ thanks to Proposition~\ref{samelaw}. With this at hand, we write for any $z_1,z_2 \in \mathbb{C}$,
\begin{equation*}
\Cov\big[\tr_N(z_1-PP^*(\mathcal{G}_t))^{-1}, \tr_N(z_2-PP^*(\mathcal{G}_t))^{-1}\big] 
= \int_0^t h'(u) du,
\end{equation*}
where \[h(u)=\E \big[\tr_N\big(R^{z_1}( \cG_u \wtcG_{t-u})\big)\tr_N\big(R^{z_2}(\cG_u\hcG_{t-u} )\big)\big].\]
Computing the derivative, we obtain  
\begin{equation}\label{h' in Covariance term}
    \begin{array}{cc}
   h'(u)& = \frac{d}{dv}_{|_{v=u}}\E \big[ \tr_N\big(R^{z_1}(\cG_v\wtcG_{t-u} )\big)\tr_N\big(R^{z_2}(\cG_v\hcG_{t-u} )\big)\big]
  \\ &\quad -\frac{d}{dv}_{|_{v=t-u}}\E \big[ \tr_N\big(R^{z_1}(\cG_u\wtcG_{v} )\big)\tr_N\big(R^{z_2}(\cG_u\hcG_{v} )\big)\big] ,
    \end{array}
\end{equation}
 and note that for any $s_1,s_2$, 
\[
\E \big[ \tr_N\big(R^{z_1}(\cG_{s_1} \wtcG_{s_2})\big)\tr_N\big(R^{z_2}(\cG_{s_1} \hcG_{s_2})\big)\big]=\frac{1}{N^2}\sum_{i,j=1}^N \E \big[ \Tr_N \big(E_{ji} R^{z_1}(\cG_{s_1} \wtcG_{s_2}) E_{ij} R^{z_2}(\cG_{s_1} \hcG_{s_2}) \big)\big],
\]
where we recall that the $E_{ij}$'s are the unit $N\times N $ matrices. Fixing $t$ and $u$, we note that by \eqref{Itoproduit}, we have 
\begin{multline}\label{h' martinagle bracket term}
    d\big\{ R^{z_1}(\cG_v\wtcG_{t-u} ) E_{ij} R^{z_2}(\cG_v\hcG_{t-u} ) \big\} 
    \\ = d\big[ R^{z_1}(\cG_v\wtcG_{t-u} )\big] E_{ij} R^{z_2}(\cG_v\hcG_{t-u} ) +R^{z_1}(\cG_v\wtcG_{t-u} ) E_{ij} d\big[R^{z_2}(\cG_v\hcG_{t-u} )\big]
    \\ + d\big\langle  R^{z_1}(\cG_v\wtcG_{t-u} ) ,  R^{z_2}(\cG_v \hcG_{t-u} ) \big\rangle_v \sharp E_{ij},
    \end{multline}
    and 
    %%%%
\begin{multline}\label{h' martinagle second bracket term}
    d\big\{ R^{z_1}(\cG_u\wtcG_{v} ) E_{ij} R^{z_2}(\cG_u\hcG_{v} ) \big\} 
    \\ = d\big[ R^{z_1}(\cG_u\wtcG_{v} )\big] E_{ij} R^{z_2}(\cG_u\hcG_{v} ) +R^{z_1}(\cG_u\wtcG_{v} ) E_{ij} d\big[R^{z_2}(\cG_u\hcG_{v} )\big]
    \\ + d\big\langle  R^{z_1}(\cG_u\wtcG_{v} ),  R^{z_2}(\cG_u\hcG_{v} ) \big\rangle_v \sharp E_{ij}.
    \end{multline}

    %%%%
    \noindent Note that \[
d \big\langle  R^{z_1}(\cG_u\wtcG_{v} ) ,  R^{z_2}(\cG_u\hcG_{v} ) \big\rangle_{v} \sharp E_{ij}=0
\] since the martingale parts of the two sides of the tensor product are stochastic integrals with respect to independent 
elliptic Brownian motions.
Now, recalling \eqref{Itoresolvante},
\begin{multline}\label{finitevaru}d\big[ R^{z_1}(\cG_v\wtcG_{t-u} )\big]=\text{a martingale part}\\ +  R^{z_1}(\cG_v\wtcG_{t-u} )\left(d \big\langle PP^* (\cG_v\wtcG_{t-u} ),  PP^* (\cG_v \wtcG_{t-u} ) \big\rangle_{v} \sharp 
  R^{z_1}(\cG_v\wtcG_{t-u} )\right) R^{z_1}(\cG_v\wtcG_{t-u} )\end{multline}
and 
\begin{multline}\label{finitevartmoinsu}d\big[ R^{z_1}(\cG_u\wtcG_{v} )\big]=\text{a martingale part}\\ + R^{z_1}(\cG_u\wtcG_{v} )\left( d \big\langle PP^* (\cG_u\wtcG_{v} ),   PP^* (\cG_u\wtcG_{v} ) \big\rangle_{v} \sharp 
  R^{z_1}(\cG_u\wtcG_{v} )\right) R^{z_1}(\cG_u\wtcG_{v} )\end{multline}
with similar equations  for  $d\big[ R^{z_2}(\cG_v\hcG_{t-u} )\big]$ and  $d\big[ R^{z_2}(\cG_u\hcG_{v} )\big]$.
When taking the expectation, the martingale terms  
are not involved.
Now, recall that the processes $\wtcG_t$ and $\hcG_t$ are chosen to satisfy the stochastic differential equations in \eqref{stoch diff eq hat and tilde}, which have a sense of multiplication opposite to that in \eqref{Stoch Diff Eq sigma}, satisfied by $\cG_t$. This choice ensures that the finite variation parts of the  first two terms in \eqref{h' martinagle   bracket term} evaluated at $v=u$
and of the first two terms in \eqref{h' martinagle second bracket term} evaluated at $v=t-u$ cancel out when computing $h'(v)$ in \eqref{h' in Covariance term}. 
Indeed, this readily follows from \eqref{finitevaru} and \eqref{finitevartmoinsu} and the fact that  for any monomials $Q_1$ and $Q_2$ in $4p$ noncommutative indeterminates and for any finite products $P_1$, $P_2$ and $V$ of monomials or resolvents of nonnegative polynomials,
\begin{multline}
\frac{d}{dv}_{|_{v=u}} \int_0^v P_1(\cG_s\wtcG_{t-u} ) \left[d\big\langle  Q_1(\cG\wtcG_{t-u} ) ,  Q_2(\cG\wtcG_{t-u} ) \big\rangle_s   \sharp V (\cG_s\wtcG_{t-u} )\right] P_2(\cG_s\wtcG_{t-u} )
\\=
\frac{d}{dv}_{|_{v=t-u}} \int_0^vP_1(\cG_u\wtcG_{s} )\left[d\big\langle Q_1(\cG_u\wtcG ) ,    Q_2(\cG_u\wtcG ) \big\rangle_s  \sharp V(\cG_u\wtcG_{s})\right] P_2(\cG_u\wtcG_{s} ).\label{equalitybrackets}
\end{multline}
To prove  \eqref{equalitybrackets}, we first introduce the notation: for any $\epsilon=(\epsilon_1,\epsilon_2)\in \{\cdot,\ast,-1\}^2$ and matrix $A$,  $A^{\epsilon}:=(A^{\epsilon_1})^{\epsilon_2},$
with $A^{\cdot}=A$.  It is easy to see that \eqref{equalitybrackets} follows from the identities:
  $\forall \epsilon=(\epsilon_1,\epsilon_2), \hat \epsilon=(\hat \epsilon_1,\hat \epsilon_2)$ such that $\epsilon_1,\epsilon_2, \hat \epsilon_1,\hat \epsilon_2\in \{\cdot, -1,\ast\}$, 
\begin{align*}
    \frac{d}{dv}_{|_{v=u}}\int_0^v d \big  \langle (\cG \wtcG_{t-u})^\epsilon 
    \,  
   , (\cG\wtcG_{t-u})^{\hat \epsilon} \,  \big\rangle_s \sharp V (\cG_s\wtcG_{t-u})
    \\
  = \frac{d}{dv}_{|_{v=t-u}}\int_0^v d\big\langle (\cG_u \wtcG)^\epsilon \,  , (\cG_u \wtcG)^{\hat \epsilon} \,  \big\rangle_s \sharp V(\cG_u \wtcG_{s}).
   \end{align*}

\noindent Note that the quadratic variation terms involving $\tilde G$ can be deduced from 
the ones involving $G$ listed in Lemma \ref{Lemma:Calcul QuadraticCov} by replacing $G$ by $\tilde G^{-1}$, $G^{-1}$ by $\tilde G$, $G^\ast$ by $\tilde G^{-1,\ast}$ and $G^{-1,\ast}$ by $\tilde G^\ast$. 
Consequently, when taking the expectation, the only contributing term is the quadratic covariation term in \eqref{h' martinagle bracket term} which yields that
\begin{align}\label{expressionhprime}
    h'(v)&= \frac{1}{N^2}\sum_{i,j=1}^N \E \Tr_N \Big[ E_{ji} \frac{d}{dv}_{|_{v=u}}\int_0^v d\big\langle  R^{z_1}(\cG\wtcG_{t-u} ) ,  R^{z_2}(\cG\hcG_{t-u} ) \big\rangle_s \sharp E_{ij}\Big]
    . 
\end{align}
We proceed by noting that $PP^\ast$ can be written as a sum of monomials: $
PP^\ast=\sum_{\alpha=1}^m W_\alpha$.
To compute the contributing quadratic covariation term above, we write for any process $V$ in $M_N(\mathbb{C})$,
\begin{multline}
  d \big\langle  R^{z_1}(\cG\wtcG_{t-u} ) ,  R^{z_2}(\cG\hcG_{t-u} ) \big\rangle_v \sharp V_v
 = \sum_{\widetilde{\alpha},\hat{\alpha}=1}^m \sum_{\ell=1}^p \sigma_\ell^2 \\ \sum_{\substack{W_{\widetilde{\alpha}}(\cG_v\wtcG_{t-u} )=A^{\ell ,\widetilde{\alpha}}_{t,u,v} \widetilde{X}^{\ell }_v B^{\ell ,\widetilde{\alpha}}_{t,u,v}\\W_{\hat{\alpha}}(\cG_v\hcG_{t-u} )=C^{\ell ,\hat{\alpha}}_{t,u,v} \hat{X}^{\ell }_v D^{\ell ,\hat{\alpha}}_{t,u,v}}}  \widetilde{R}_{t,u,v}(z_1)  A^{\ell ,\widetilde{\alpha}}  \Big( d\big\langle \widetilde{X}^{\ell } , \hat{X}^{\ell }\big\rangle_{v} \sharp \big(B^{\ell ,\widetilde{\alpha}} \widetilde{R}_{t,u,v}^{z_1} V_v \hat{R}_{t,u,v}^{z_2} C^{\ell ,\hat{\alpha}}\big)\Big) D^{\ell ,\hat{\alpha}}  \hat{R}_{t,u,v}^{z_2} \label{bracketRtildeRhat}
\end{multline}
at time $v=u$ and  for all choices of $\widetilde{X}^{\ell }_v  \in \big\{ G_v^{(\ell)}\widetilde{G}_{t-u}^{(\ell)}, ( G_v^{(\ell)}\widetilde{G}_{t-u}^{(\ell)} )^\ast\!, (G_v^{(\ell)}\widetilde{G}_{t-u}^{(\ell)} )^{-1}\!, ((G_v^{(\ell)}\widetilde{G}_{t-u}^{(\ell)} )^{-1})^\ast\!  \big\}$ and  $\hat{X}^{\ell }_v \in$ $\big\{ G_v^{(\ell)}\hat{G}_{t-u}^{(\ell)} \!\!,$ $ (G_v^{(\ell)}\hat{G}_{t-u}^{(\ell)} )^\ast, (G_v^{(\ell)}\hat{G}_{t-u}^{(\ell)} )^{-1}, ((G_v^{(\ell)}\hat{G}_{t-u}^{(\ell)} )^{-1})^\ast  \big\}$. To simplify the presentation, we have omitted the subscripts from the terms \(A\), \(B\), \(C\), and \(D\), retaining them only for the resolvents, which will reflect these subscripts throughout the computation. 
Now, one can easily check that  for all such choices of $\widetilde{X}^{\ell }_v
$ and  $\hat{X}^{\ell }_v $ as above, we have 
\begin{equation}\label{crochetsgeneraux}d\big\langle \widetilde{X}^{\ell } , \hat{X}^{\ell }\big\rangle_{v} \sharp V_v= Q_1^{(\ell)}\tr_N(Q_2^{(\ell) }V_v)dv,\end{equation}
for some monomials $Q_1^{(\ell)}$ and $Q_2^{(\ell) }$ in $G^{(\ell)}_v, (G_v^{(\ell)})^{-1}, (G_v^{(\ell)})^{\ast}, (G_v^{(\ell)})^{-1,\ast}$, $ $
$\tilde G^{(\ell)}_{t-u}, (\tilde G_{t-u}^{(\ell)})^{-1}, $ $(\tilde G_{t-u}^{(\ell)})^{\ast},$ $ (\tilde G_{t-u}^{(\ell)})^{-1,\ast}$,
 $\hat G^{(\ell)}_{t-u}, (\hat G_{t-u}^{(\ell)})^{-1},$ $ (\hat G_{t-u}^{(\ell)})^{\ast},$ and $ (\hat G_{t-u}^{(\ell)})^{-1,\ast}$.
Hence, we obtain that 
\begin{align}
& \frac{1}{N^2}\sum_{i,j=1}^N  \Tr_N \frac{d}{dv}_{|_{v=u}}\int_0^v \Big[ E_{ji} \widetilde{R}_{t,u,v}^{z_1}  A^{\ell ,\widetilde{\alpha}}  \Big(d \big\langle \widetilde{X}^{\ell } , \hat{X}^{\ell }\big\rangle_v \sharp \big(B^{\ell ,\widetilde{\alpha}} \widetilde{R}_{t,u,v}^{z_1} E_{ij}\hat{R}_{t,u,v}^{z_2} C^{\ell ,\hat{\alpha}}\big)\Big) D^{\ell ,\hat{\alpha}}  \hat{R}_{t,u,v}^{z_2}  \Big]
 \nonumber\\ & =\frac{1}{N^3} \sum_{i,j=1}^N  \Tr_N\big[E_{ji}\widetilde{R}_{t,u}^{z_1}  A^{\ell ,\widetilde{\alpha}} Q_1^{(\ell)}D^{\ell ,\hat{\alpha}}  \hat{R}_{t,u}^{z_2} \big] \Tr_N\big[Q_2^{(\ell)} B^{\ell ,\widetilde{\alpha}} \widetilde{R}_{t,u}^{z_1} E_{ij} \hat{R}_{t,u}^{z_2} C^{\ell ,\hat{\alpha}} \big]
\nonumber\\&=  \frac{1}{N^2} \tr_N \Big[ (\hat{R}_{t,u}^{z_2})^2 C^{\ell ,\hat{\alpha}} Q_2^{(\ell)} B^{\ell ,\widetilde{\alpha}} (\widetilde{R}_{t,u}^{z_1})^2 A^{\ell ,\widetilde{\alpha}} Q_1^{(\ell)} D^{\ell ,\hat{\alpha}}
 \Big].\label{expressioncondense}
 \end{align}
Finally, we note that \eqref{expressionhprime}, \eqref{bracketRtildeRhat} and  
\eqref{expressioncondense} readily yield the statement of the lemma and end the proof.
 \end{proof}
%%%%%%%%%%%%%%%%%
Theorem \ref{theo: variance estimate} readily follows from \eqref{methodeHS}, Lemma \ref{cov}, H\"older inequality, Corollary \ref{NormeLp}, Proposition \ref{samelaw} and 
applying  Lemma \ref{applicationHS} for the integral on $z_1$ and then on $z_2$.
%%%%%%%%%%%%%%%%%%%%%%%
\paragraph{The particular case of $f_{N,\delta}$.}
We let  $0<\delta <1$ and recall the definitions of the set $\K_N(\delta)$ and the function $f_{N,\delta}$ as given in \eqref{def:K(delta)} and \eqref{def:f delta}, respectively. 
As explained in Section \ref{Section: idea of proof},  we need a sharper estimate of $\Var{\left[\tr _N(f_{N,\delta}(PP^{*}(\mathcal{G}_t,A^N)))\right]}$ of order $\frac{1}{N^4}$.

At this point, we can see that we  need a more general estimate than \eqref{singleresolvent}, since Lemma \ref{cov} involves a mixture of resolvents and monomials as well as independent
tuples of multiplicative Brownian motions  that are left- and right-invariant.  Thus, the remaining part relies on Theorem \ref{theo:fundamental}.  
Indeed, assume that Theorem \ref{theo:fundamental}
is established. According to  Lemma \ref{cov}, there exists some finite subset $\Theta$ so that 
  \begin{multline*}
      \Cov\big[\tr_N(z_1-PP^*(\mathcal{G}_t,A^N))^{-1},\; \tr_N(z_2-PP^*(\mathcal{G}_t,A^N))^{-1}\big]\\=\frac{1}{N^2} \sum_{\theta\in \Theta}
  \int_0^t \E\tr_N\big[ (z_1-PP^\ast(\cG_u\wtcG_{t-u},A^N ))^{-2} Q_1^{(\theta)} (\cG_u,\wtcG_{t-u},\hcG_{t-u},A^N) \\(z_2- PP^\ast(\cG_u\hcG_{t-u},A^N ))^{-2}  Q_2^{(\theta)}(\cG_u,\wtcG_{t-u},\hcG_{t-u},A^N)\big]du,
 \end{multline*}
 where $Q_1^{(\theta)}$ and $Q_2^{(\theta)}$ are noncommutative monomials in $12p+2r$ variables.
 Let  $\tilde  g_{\lambda, \tau}^{(\ell)}$, $\hat g_{\lambda, \tau}^{(\ell)}$, $\ell=1,\ldots,p$, be freely independent copies of the free multiplicative $(\lambda,\tau)$-Brownian motions defined for any $t\geq 0$ as the solution of the free stochastic differential equations 
\begin{equation*}
    d\hat g_{\lambda, \tau}^{(\ell)}(t)=\sigma_\ell \, \big( i \, d\hat z_{\lambda, \tau}^{(\ell)}(t) -\frac{\sigma_\ell}{2}(\lambda-\tau)\, dt \big) \hat g_{\lambda, \tau}^{(\ell)}(t)\quad \text{with}\quad \hat g_{\lambda, \tau}^{(\ell)}(0)=I_\A,\end{equation*}
    \begin{equation*}
    d\tilde g_{\lambda, \tau}^{(\ell)}(t)=\sigma_\ell \, \big( i \, d\tilde z_{\lambda, \tau}^{(\ell)}(t) -\frac{\sigma_\ell}{2}(\lambda-\tau)\, dt \big) \tilde g_{\lambda, \tau}^{(\ell)}(t)\quad \text{with}\quad \tilde g_{\lambda, \tau}^{(\ell)}(0)=I_\A,\end{equation*}
    where the $\hat z_{\lambda, \tau}^{(\ell)} $'s, $\tilde z_{\lambda, \tau}^{(\ell)} $'s are free rotated elliptic Brownian motions that are freely independent from each other and from the $z_{\lambda, \tau}^{(\ell)}$'s. Note that the difference between the above free stochastic differential equations and the one satisfied by ${g}_{\lambda, \tau}^{(\ell)}(t)$ in \eqref{free Stoch Diff Eq with sigma} lies in the sense of multiplication.  Finally, we denote by 
     \begin{align*}   \quad{\underline{g}}_t :=(\g_t,\g_t^\ast,\g_t^{-1},\g_t^{-1,\ast}) 
\quad &\text{with} \quad
 \g_t=( g_{\lambda,\tau}^{(1)},\dots , g_{\lambda,\tau}^{(p)}), \\
   \quad\hat{\underline{g}}_t:=(\hat\g_t,\hat\g_t^\ast,\hat \g_t^{-1},\hat \g_t^{-1,\ast}) 
\quad &\text{with} \quad
\hat \g_t=(\hat{g}_{\lambda,\tau}^{(1)},\dots , \hat{g}_{\lambda,\tau}^{(p)}),
\\
   \quad\tilde{\underline{g}}_t:=(\tilde\g_t,\tilde\g_t^\ast,\tilde \g_t^{-1},\tilde \g_t^{-1,\ast}) 
\quad &\text{with} \quad
\tilde \g_t=(\tilde{g}_{\lambda,\tau}^{(1)},\dots , \tilde{g}_{\lambda,\tau}^{(p)}).
\end{align*}
As for $(g_{\lambda,\tau}(t))_{t\geq 0}$, we will work in the  $C^\ast$-probability space $(\mathcal{A}_N,\varphi_N)$ in which  $({\underline{g}}_t)_{t\geq 0}$, $(\hat{\underline{g}}_t)_{t\geq 0}$, and $(\tilde{\underline{g}}_t)_{t\geq 0}$ are freely independent from  $A_N$.

Now, according to Theorem \ref{theo:fundamental}, there exist commutative polynomials with nonnegative coefficients $T_1$ and $T_2$ such that for all $0\leq u\leq t$, 
 \begin{multline*}     
\left|\E\tr_N\big[ (z_1-PP^\ast(\cG_u\wtcG_{t-u} ,A^N))^{-2} Q_1^{(\theta)}(\cG_u,\wtcG_{t-u},\hcG_{t-u},A^N) \right. \\ \cdot
(z_2- PP^\ast(\cG_u\hcG_{t-u} ,A^N))^{-2} Q_2^{(\theta)}(\cG_u,\wtcG_{t-u},\hcG_{t-u},A^N)\big]\\
-\left. \varphi_N\big[ (z_1-PP^\ast({\underline{g}}_u\tilde {\underline{g}}_{t-u} ,A^N))^{-2} Q_1^{(\theta)}({\underline{g}}_u,\tilde {\underline{g}}_{t-u},\hat {\underline{g}}_{t-u},A^N) (z_2-PP^\ast({\underline{g}}_u\hat {\underline{g}}_{t-u},A^N ))^{-2}  Q_2^{(\theta)}({\underline{g}}_u,\tilde {\underline{g}}_{t-u},\hat {\underline{g}}_{t-u},A^N) \big]\right|\\
\leq \frac{1}{N^2} T_1(|\Im z_1|^{-1})T_2(|\Im z_2|^{-1}).
 \end{multline*}
 Choose $k=\max\{k_1,k_2,2\}$ where the $k_i$'s denote the degrees of the $T_i$'s.
 Then, applying successively Lemma \ref{applicationHS} for the integral on $z_1$ and then on $z_2$,  we obtain that \begin{align}
&\Var{\left[\tr _N(f_{N,\delta}(PP^{*}(\mathcal{G}_t,A^N)))\right]}
\\& =\frac{1}{\pi^2} 
\frac{1}{N^2} \sum_{\theta\in \Theta}\int_0^t \int_{\mathbb C} \int_{\mathbb C}\bar{\partial} F_k(f_{N,\delta})(z_1)\bar{\partial} F_k(f_{N,\delta})(z_2)\nonumber
 \varphi_N\big[ (z_1-PP^\ast({\underline{g}}_u\tilde {\underline{g}}_{t-u} ,A^N))^{-2} Q_1^{(\theta)}({\underline{g}}_u,\tilde {\underline{g}}_{t-u},\hat {\underline{g}}_{t-u},A^N) \\& \qquad \cdot (z_2-PP^\ast({\underline{g}}_u\hat {\underline{g}}_{t-u},A^N ))^{-2}  Q_2^{(\theta)}({\underline{g}}_u,\tilde {\underline{g}}_{t-u},\hat {\underline{g}}_{t-u},A^N) \big]
   d^2z_1d^2z_2 \nonumber +O(1/N^4).\label{integralescomplexes}
\end{align}

Recall that $ f_{N,\delta}$  belongs to $ {\cal C}^\infty (\mathbb{R}, \mathbb{R}) $, satisfies $ f_{N,\delta} \equiv 1$ on $ {\K}_N(\frac{\delta}{2})$, and has support contained in $ {\K}_N(2\delta)$.
Since $\chi\equiv 1$ around zero, there exists $L>0$ such that ${\chi}^{\prime}\equiv 0$ on $[-L;L]$ from which we deduce that $\bar{\partial} F_k(f_{N,\delta})(z)\equiv 0$ on $K_N({\delta/2})\times [-L;L]$.
Therefore, the previous integrals on $\mathbb{C}$ are actually integrals on a compact set $\mathcal{C}_N$ in $^c(K_N({\delta/2})\times [-L;L])$.
Note that the spectrum of  $ PP^*({\underline{g}}_u\hat {\underline{g}}_{t-u}, A^N)$ and the spectrum of  $ PP^*({\underline{g}}_u\tilde {\underline{g}}_{t-u} ,A^N)$ coincide with the spectrum ${\bf Sp}_N$ of  $ PP^*({\underline{g}}_t,A^N)$. Indeed, thanks to Proposition~\ref{Prop:verylast}, the noncommutative distribution of $\hat {\underline{g}}_{t-u}$ is the same as the one of $\underline{g}_{t-u}$, so the stationarity and freeness of increments of Proposition~\ref{Prop:stationarity} yield that the noncommmutative distribution of ${\underline{g}}_u\hat {\underline{g}}_{t-u}$  is the one of  ${\underline{g}}_{u+(t-u)}={\underline{g}}_t$. Similarly, the noncommmutative distribution of ${\underline{g}}_u\tilde {\underline{g}}_{t-u}$  also matches that of  ${\underline{g}}_t$. As a result,  $ PP^*({\underline{g}}_u\hat {\underline{g}}_{t-u} ,A^N)$, $ PP^*({\underline{g}}_u\tilde {\underline{g}}_{t-u},A^N )$ and $ PP^*({\underline{g}}_t,A^N)$ all share the same spectrum ${\bf Sp}_N$.

Now, for any $u\in {\bf Sp}_N$ and any $z_2 \in \mathcal{C}_N$, $\left|(z_2-u)^{-2}\right|\leq (\min(\delta/2,L))^{-2}.$
Thus, we can differentiate under the integral and obtain that 
$$f_{N,\delta}'(u)=-\frac{1}{\pi}\int_{\mathbb{C}}\bar{\partial} F_k(f_{N,\delta})(z_2){(z_2-u)}^{-2}dz_2. $$
Therefore, 
$$\int_{\mathbb{C}}\bar{\partial} F_k(f_{N,\delta})(z_2)(z_2-PP^*( {\underline{g}}_u\hat {\underline{g}}_{t-u} ,A^N))^{-2}dz_2=-\pi f_{N,\delta}'(PP^*({\underline{g}}_u\hat {\underline{g}}_{t-u} ,A^N))\\=0,$$
so that  we obtain the following estimate on the variance
\begin{equation*}
\Var{\left[\tr _N(f_{N,\delta}(PP^{*}(\mathcal{G}_t, A^N)))\right]}:=\Var{\left[\tr _N(f_{N,\delta}(PP^{*}(\G_t,\G_t^\ast,\G_t^{-1}, \G_t^{-1,\ast},A^N)))\right]}=O\Big(\frac{1}{N^4}\Big).
\end{equation*}

{

\section{Fundamental Result}\label{Section:fundamental result}
This section is dedicated to the proof of the fundamental result of the paper. Let $P^{(1)}, \ldots, P^{(n)},$ $ Q^{(1)}, \ldots, Q^{(n)}$  be  noncommutative polynomials in $X_\ell,$
$ X_\ell^*,$
$
X_\ell^{-1},$
$
X_\ell^{-1,*},$
$ \hat X_\ell,$
$ \hat X_\ell^*,$
$ \hat X_\ell^{-1},$
$
\hat X_\ell^{-1,*},$
$ \ell=1,\ldots,p$ and $Y_i, Y^*_i, i=1,\ldots,r$. 
Define 
for $z_1,\ldots, z_n\in \mathbb{C}\setminus \mathbb{R}$, $$R^{(i)}=
(z_i-P_iP_i^*)^{-1}.$$
Let  $\hat g_{\lambda, \tau}^{(\ell)}$, $\ell=1,\ldots,p$, be freely independent copies of the free multiplicative $(\lambda,\tau)$-Brownian motions defined for any $t\geq 0$ as the solution of the free stochastic differential equation 
\begin{equation*}
    d\hat g_{\lambda, \tau}^{(\ell)}(t)=\sigma_\ell \, \big( i \, d\hat z_{\lambda, \tau}^{(\ell)}(t) -\frac{\sigma_\ell}{2}(\lambda-\tau)\, dt \big) \hat g_{\lambda, \tau}^{(\ell)}(t)\quad \text{with}\quad \hat g_{\lambda, \tau}^{(\ell)}(0)=I_\A,\end{equation*}
    where the $\hat z_{\lambda, \tau}^{(\ell)} $'s are independent rotated elliptic Brownian motions independent from the $z_{\lambda, \tau}^{(\ell)}$'s.
Note that the difference between the above free stochastic differential equations and the one satisfied by ${g}_{\lambda, \tau}^{(\ell)}(t)$ in \eqref{free Stoch Diff Eq with sigma} lies in the sense of multiplication.  Finally, we denote by 
     \begin{align*}   \quad{\underline{g}}_t :=(\g_t,\g_t^\ast,\g_t^{-1},\g_t^{-1,\ast}) 
\quad &\text{with} \quad
 \g_t=( g_{\lambda,\tau}^{(1)},\dots , g_{\lambda,\tau}^{(p)}), \\
   \quad\hat{\underline{g}}_t:=(\hat\g_t,\hat\g_t^\ast,\hat \g_t^{-1},\hat \g_t^{-1,\ast}) 
\quad &\text{with} \quad
\hat \g_t=(\hat{g}_{\lambda,\tau}^{(1)},\dots , \hat{g}_{\lambda,\tau}^{(p)}).
\end{align*}
As previously, we will work in the  $C^\ast$-probability space $(\mathcal{A}_N,\varphi_N)$ in which  $({\underline{g}}_t)_{t\geq 0}$ and $(\hat{\underline{g}}_t)_{t\geq 0}$ are freely independent from  $A_N$.
\begin{theorem}\label{theo:fundamental}
For $i=1,\dots , n$, let $Q^{(i)}$ and $R^{(i)}$  be respectively any noncommutative polynomials and resolvents in $8p+2r$ noncommutative indeterminates, as defined above. Let $\cG$ and $\hcG$ be  independent multiplicative $(\lambda,\tau)$-Brownian motions satisfying \eqref{Stoch Diff Eq sigma} and \eqref{stoch diff eq hat and tilde}, respectively. Then, for any $t>0$, there exists a
commutative polynomial $T_t$ with nonnegative coefficients
such that for all $0\leq u\leq t$, $${\mathbb{E}\left[\tr_{N}\left(R^{(1)}Q^{(1)}\cdots R^{(n)}Q^{(n)} ({\mathcal{G}}_u,\hat{\mathcal{G}}_u, A^N)\right)\right]  }=  \varphi_N\left(R^{(1)}Q^{(1)}\cdots R^{(n)}Q^{(n)} (\underline{g}_u, \hat{\underline{g}}_u,A^N)\right)+\Delta_u$$
  where $$|\Delta_u|\leq \frac{ T_t(|\Im z_i|^{-1}, i=1, \ldots, n)}{N^2}.$$
\end{theorem}

\paragraph*{Step 1.} The first step is inspired by the work of Collins-Guionnet-Parraud
\cite{CGParraud}, later extended by Parraud \cite{ParraudHaar, ParraudGUE, parraud-HaarTensor}, making use of the fact that for any $k$, the distribution of  $I_k\otimes  \G_t, $ in $\left(M_{Nk}(\mathbb{C}), \tr_{Nk}\right)$ coincides with  the  distribution of  $  \G_t$ in $\left(M_{N}(\mathbb{C}), \tr_{N}\right)$, and viewing $\g_t$
 as the asymptotic distribution of 
 an independent tuple of $(\lambda,\tau)$-Brownian motions $\bfK_t$ in $GL_{Nk}(\mathbb{C})$ as $k \rightarrow \infty$. 
 
 With this aim, we let $k\geq 1$ and $Y_{\lambda,\tau}^{(\ell)}$'s, $\hat Y_{\lambda,\tau}^{(\ell)}$ be standard independent Brownian motions on  $M_{Nk}(\mathbb{C})$, rescaled by $\frac{1}{\sqrt{Nk}}$ and  independent from the $Z_{\lambda,\tau}^{(\ell)}$'s and  $\hat Z_{\lambda,\tau}^{(\ell)}$'s. We introduce the multiplicative $(\lambda,\tau)$-Brownian motions on $M_{Nk}(\mathbb{C})$ defined as a solution of the stochastic equations
\begin{align}
 &d  K_{\lambda,\tau}^{(\ell)}(t)=i\sigma_\ell d Y_{\lambda,\tau}^{(\ell)}(t) K^{(\ell)}_{\lambda,\tau}(t) 
-\frac{1}{2} \sigma^2_\ell (\lambda-\tau)  K_{\lambda,\tau}^{(\ell)}(t)dt\in M_{Nk}(\mathbb{C}), \quad  K_{\lambda,\tau}^{(\ell)}(0)=I_{Nk}, \label{K Stocg diff eq}
\\&d\hat K_{\lambda,\tau}^{(\ell)}(t)=i\sigma_\ell \hat K_{\lambda,\tau}^{(\ell)}(t) d\hat Y_{\lambda,\tau}^{(\ell)}(t)-\frac{1}{2} \sigma^2_\ell   (\lambda-\tau) \hat K_{\lambda,\tau}^{(\ell)}(t) dt \in M_{Nk}(\mathbb{C}), \quad\hat  K_{\lambda,\tau}^{(\ell)}(0)=I_{Nk}. \label{K hat Stocg diff eq}
\end{align}
We maintain consistent notation as before and denote by
\begin{align*}\mathcal {K}_t:=(\bfK_t,\bfK_t^\ast,\bfK_t^{-1}, \bfK_t^{-1,\ast}) \quad &\text{where} \quad \bfK_t:=(K_{\lambda,\tau}^{(1)},\ldots,K_{\lambda,\tau}^{(p)}),
\\ \hat{\mathcal {K}}_t:=(\hat{ \bfK}_t,\hat{ \bfK}_t^\ast, \hat{\bfK}_t^{-1}, \hat{\bfK}_t^{-1,\ast}) \quad&\text{where} \quad\hat{\bfK}_t:=(K_{\lambda,\tau}^{(1)},\ldots,K_{\lambda,\tau}^{(p)}).\end{align*}
The idea is to let  $k$ grow to infinity while keeping $N$ fixed so that $({\mathcal {K}},\hat{\mathcal {K}})$ converges to  $ (\underline{ g},\underline{\hat g})$.
 Finally, we denote by 
\begin{align*}
I_k\otimes  \cG_t:=(I_k\otimes\G_t,I_k\otimes\G_t^\ast, I_k\otimes\G_t^{-1},I_k\otimes \G_t^{-1,\ast}),
\quad \\  \text{and} \quad I_k\otimes\hat{\mathcal{G}}_t:=(I_k\otimes\hbfG_t,I_k\otimes\hbfG_t^\ast, I_k\otimes\hbfG_t^{-1},I_k\otimes\hbfG_t^{-1,\ast}) ,
\quad
\end{align*}
with $I_k\otimes\G_t=(I_k\otimes G_{\lambda,\tau}^{(1)},\dots , I_k\otimes G_{\lambda,\tau}^{(p)})$ and $
I_k\otimes\hbfG_t=(I_k\otimes\hG_{\lambda,\tau}^{(1)},\dots , I_k\otimes\hG_{\lambda,\tau}^{(p)}).$
Note that, for any $k$, the distribution of  $(I_k\otimes  \cG_t, I_k\otimes\hat{\mathcal{G}}_t)$ in $(M_{Nk}(\mathbb{C}), \tr_{Nk})$ coincides with  the  distribution of  $(  \cG_t, \hat{\mathcal{G}}_t)$ in $\left(M_{N}(\mathbb{C}), \tr_{N}\right)$. With all the necessary elements in place, we are now ready to illustrate the first step in the proof of Theorem \ref{theo:fundamental}, which consists of the following proposition.
\begin{proposition}\label{step1}
 Let $\mathcal{D}$ be a tuple of deterministic matrices. Then 
\begin{multline*}
\mathbb{E}\left[\tr_{Nk}\big(R^{(1)}Q^{(1)}\cdots R^{(n)}Q^{(n)} (\cK_t, \hcK_t, \mathcal{D})\big)\right] 
  \\ =  \mathbb{E}\left[\tr_{Nk}\big(R^{(1)}Q^{(1)}\cdots R^{(n)}Q^{(n)} (I_k\otimes  \cG_t, I_k\otimes\hat{\mathcal{G}}_t, \mathcal{D})\big)\right]  +  \int_{0}^t \mathbf{M}_u du ,
\end{multline*}
with $\mathbf{M}_u$ being a finite sum of terms of the form 
\begin{equation}\label{Mu main prop}
   \E\big[ \tr_{Nk}(M) \tr_{Nk}(\tilde M )\big]-\mathbb{E}\big[ \tr_{Nk}\big(\tilde M [(id_k\otimes \tr_{N}) (M)\otimes I_N] \big)\big], 
\end{equation}
and $M$ and $\tilde M$ are finite products of monomials or resolvents of nonnegative polynomials in the components of $I_k\otimes  \cG_{t-u}, I_k\otimes\hat{\mathcal{G}}_{t-u}, \mathcal{K}_u, \hat{\mathcal {K}}_u$ and $\mathcal{D}$.
   
\end{proposition}
\begin{proof}
We start by setting some notation. For any polynomial or resolvent $W$ in noncommuting variables, we denote by 
\begin{align*}
 W_{t,u}&:=W((I_k\otimes \cG_{t-u})\cK_u,  \hcK_u (I_k\otimes \hcG_{t-u}),\mathcal{D})
.
\end{align*}
Finally, we set  $$\Pi_{t,u}=R^{(1)}_{t,u}Q^{(1)}_{t,u}\cdots R^{(n)}_{t,u}Q^{(n)}_{t,u},$$ 
so that 
\begin{multline*}
\mathbb{E}\Big[\tr_{Nk}\Big(R^{(1)}Q^{(1)}\cdots R^{(n)}Q^{(n)}  (\cK_t, \hcK_t,\mathcal{D}) - R^{(1)}Q^{(1)}\cdots R^{(n)}Q^{(n)} (I_k\otimes  \cG_t, I_k\otimes\hat{\mathcal{G}}_t,\mathcal{D})\Big)\Big]    \\= \mathbb{E}\left[\Pi_{t,t}\right]
-\mathbb{E}\left[\Pi_{t,0}\right]=\int_0^t \frac{d}{du}\mathbb{E}\left[\Pi_{t,u}\right]du.
\end{multline*}
We compute the derivative and write 
\begin{align}
        \frac{d}{du} \mathbb{E}(\Pi_{t,u})= \frac{d}{dv}_{|_{v=u}} \mathbb{E}\left[R^{(1)}Q^{(1)}\cdots R^{(n)}Q^{(n)}
   ((I_k\otimes \cG_{t-u})\cK_v, \hcK_v (I_k\otimes \hcG_{t-u}),\mathcal{D})\right]\label{1st derivative term Fund theo} \\ -
    \frac{d}{dv}_{|_{v=t-u}} \mathbb{E}\left[
    R^{(1)}Q^{(1)}\cdots R^{(n)}Q^{(n)}((I_k\otimes \cG_{v})\cK_u,\hcK_u (I_k\otimes \hcG_{v}),\mathcal{D})\right]\label{2nd derivative term Fund theo} .
\end{align}
Note that the first term leads to stochastic calculus with respect to $\mathcal{K}$ and $\hat{\mathcal{K}}$ while the second term leads to stochastic calculus with respect to $I_k\otimes \cG$ and $I_k\otimes \hcG$.
In both cases, 
 we obtain according to \eqref{Itoproduit} and \eqref{Itoresolvante}, 
 $$dR^{(1)}Q^{(1)}\cdots R^{(n)}Q^{(n)}= d\Sigma^{(1)}+  d\Sigma^{(2)}$$
where
\begin{multline*}
  d\Sigma^{(1)} =    \sum_{i=1}^n R^{(1)}Q^{(1)}\cdots R^{(i)}dQ^{(i)}R^{(i+1)}\cdots R^{(n)}Q^{(n)}
 \\+ \sum_{i=1}^n R^{(1)}Q^{(1)}\cdots Q^{(i-1)}R^{(i)}d(P_iP_i^*)R^{(i)}Q^{(i)}\cdots R^{(n)}Q^{(n)},   
\end{multline*}
and 
\begin{multline}
   d\Sigma^{(2)}=    \sum_{i=1}^n R^{(1)}Q^{(1)}\cdots Q^{(i-1)}R^{(i)}\left(d\langle P_iP_i^*, P_iP_i^*\rangle \sharp R^{(i)}\right)  R^{(i)}Q^{(i)}\cdots R^{(n)}Q^{(n)} \\
+\sum_{i<j}N_t^{(1)}\cdots N^{(i-1)}\left(d<N^{(i)}, N^{(j)}> \sharp [N^{(i+1)}\cdots N^{(j-1)}]\right)N^{(j+1)}\cdots N^{(n)}.
\label{Sigma2} 
\end{multline}
Here, for $i=1,\ldots,n$, we have denoted  by $N^{(2i)}=Q^{(i)}$ and $N^{(2i-1)}=R^{(i)}$. Denote by $d\Sigma^{(1)}_{\cK,\hcK}$  and $d\Sigma^{(2)}_{\cK,\hcK}$ the terms we obtain when deriving with respect to $\mathcal{K}$ and $\hat{\mathcal{K}}$,   by $d\Sigma^{(1)}_{\cG,\hcG}$  and $d\Sigma^{(2)}_{\cG,\hcG}$ the terms we obtain when deriving with respect to
 $(I_k\otimes \cG)$ and $(I_k\otimes \hcG)$. With this notation, we write 
\begin{multline} \frac{d}{du} \mathbb{E}(\Pi_{t,u})\\= \frac{d}{dv}_{|_{v=u}}\mathbb{E}\int_0^v \left[d\Sigma^{(1)}_{\cK,\hcK}((I_k\otimes \cG_{t-u})\cK_s,\hcK_s (I_k\otimes \hcG_{t-u}),\mathcal{D}) +d\Sigma^{(2)}_{\cK,\hcK} ((I_k\otimes \cG_{t-u})\cK_s,\hcK_s (I_k\otimes \hcG_{t-u}),\mathcal{D})\right]\\- \frac{d}{dv}_{|_{v=t-u}}\mathbb{E}\int_0^v \left[d\Sigma^{(1)}_{\cG,\hcG} ((I_k\otimes \cG_s)\cK_u,\hcK_u (I_k\otimes \hcG_s),\mathcal{D})+d\Sigma^{(2)}_{\cG,\hcG}((I_k\otimes \cG_s)\cK_u,\hcK_u (I_k\otimes \hcG_s),\mathcal{D})\right],\end{multline}
where $s$ stands for the 
variable integration time. We now note that
\begin{multline*}   
d\Sigma^{(1)}_{\cK,\hcK}~\text{or} ~d\Sigma^{(1)}_{\cG,\hcG} = \text{a martingale part} \\+ \left\{ cR^{(1)}Q^{(1)}\cdots R^{(n)}Q^{(n)}+ 
 \sum_{i=1}^n c_i R^{(1)}Q^{(1)}\cdots Q^{(i-1)}R^{(i)}P_iP_i^*R^{(i)}Q^{(i)}\cdots R^{(n)}Q^{(n)}
\right\}dv \end{multline*}
for some constant complex numbers $c, c_i$'s that are identical in the terms $d\Sigma^{(1)}_{\cK,\hcK}$ and $d\Sigma^{(1)}_{\cG,\hcG}$. Hence, 
the contributions of $d\Sigma^{(1)}_{\cK,\hcK}$ and $d\Sigma^{(1)}_{\cG,\hcG}$ cancel each other in the computation of  $\frac{d}{du} \mathbb{E}(\Pi_{t,u})$. 
To compute the term $\Sigma^{(2)}$, we first recall that for any $\epsilon=(\epsilon_1,\epsilon_2)\in \{\cdot,\ast,-1\}^2$ and matrix $A$,  $A^{\epsilon}:=(A^{\epsilon_1})^{\epsilon_2},$
with $A^{\cdot}=A$. Now, according to the computations in Lemma \ref{Lemma:Calcul QuadraticCov},   $\forall \epsilon=(\epsilon_1,\epsilon_2), \hat \epsilon=(\hat \epsilon_1,\hat \epsilon_2)\in \{\cdot, -1,\ast\}^2$,
there exist noncommutative monomials $T_1^{ \epsilon,\hat \epsilon}$ and $T_2^{ \epsilon,\hat \epsilon}$  in two variables such that for any $N\times N$ adapted  matrix-valued process $V$,
\begin{equation}\label{bracketgeneral}d\langle G^\epsilon, G^{\hat \epsilon}\rangle_v \sharp V_v=T^{ \epsilon,\hat \epsilon}_1(G_v^{\epsilon},G_v^{\hat \epsilon} ) \tr_N(T^{ \epsilon,\hat \epsilon}_2(G_v^{\epsilon},G_v^{\hat \epsilon} ) V_v)dv.
\end{equation}
Then, as remarked after Lemma \ref{Lemma:Calcul QuadraticCov}, \eqref{tensor} readily yields that for any $Nk\times Nk$ matrix-valued process $V$,
\begin{equation}\label{bracketgeneraltensor}d\langle I_k\otimes G^\epsilon, I_k\otimes G^{\hat \epsilon}\rangle_v \sharp V_v=T^{ \epsilon,\hat \epsilon}_1(I_k\otimes G_v^{\epsilon},I_k\otimes G_v^{\hat \epsilon} ) (id_k \otimes \tr_N)(T^{ \epsilon,\hat \epsilon}_2(I_k\otimes G_v^{\epsilon},I_k\otimes G_v^{\hat \epsilon} ) V_v)dv,
\end{equation}
for the same noncommutative monomials $T_1^{ \epsilon,\hat \epsilon}$ and $T_2^{ \epsilon,\hat \epsilon}$. Then, considering the formulation in \eqref{compute brackets}, one can check that
\begin{multline*}d\Big\langle \int_0^{\cdot}((I_k\otimes G_{t-u})dK_s)^\epsilon, \int_0^{\cdot}((I_k\otimes G_{t-u})dK_s)^{\hat \epsilon}\Big\rangle_v \sharp V_v\\=P_1^{ \epsilon,\hat \epsilon}(I_k\otimes \cG_{t-u},\cK_v) \tr_{Nk}\left[P_2^{ \epsilon,\hat \epsilon}(I_k\otimes \cG_{t-u},\cK_v ) V_vP_3^{ \epsilon,\hat \epsilon}(I_k\otimes \cG_{t-u},\cK_v )\right] P_4^{ \epsilon,\hat \epsilon}(I_k\otimes \cG_{t-u},\cK_v )dv,%\label{brackettensordK}
\end{multline*}
and 
\begin{multline*}d\Big\langle \int_0^{\cdot}((I_k\otimes dG_s)K_u)^\epsilon, \int_0^{\cdot}((I_k\otimes dG_s)K_u)^{\hat \epsilon}\Big\rangle_v \sharp V_v\\=P_1^{ \epsilon,\hat \epsilon}(I_k\otimes \cG_v,\cK_u) \left((id_k\otimes \tr_N)\left[P_2^{ \epsilon,\hat \epsilon}(I_k\otimes \cG_v,\cK_u ) V_vP_3^{ \epsilon,\hat \epsilon}(I_k\otimes \cG_v,\cK_u )\right]\otimes I_N\right) P_4^{ \epsilon,\hat \epsilon}(I_k\otimes \cG_v,\cK_u )dv%\label{brackettensordG}
\end{multline*}
with the same monomials $P_i^{ \epsilon,\hat \epsilon}$'s.
Similarly 
\begin{multline*} d\Big\langle \int_0^{\cdot} (d\hat K_s (I_k\otimes \hat G_{t-u}))^\epsilon, \int_0^{\cdot}(d\hat K_s(I_k\otimes \hat G_{t-u}))^{\hat \epsilon}\Big\rangle_v \sharp V_v\\=\hat P_1^{ \epsilon,\hat \epsilon}(I_k\otimes \hcG_{t-u},\hcK_v) \tr_{Nk}\left[\hat P_2^{ \epsilon,\hat \epsilon}(I_k\otimes \hcG_{t-u},\hcK_v ) V_v\hat P_3^{ \epsilon,\hat \epsilon}(I_k\otimes \hcG_{t-u},\hcK_v )\right] \hat P_4^{ \epsilon,\hat \epsilon}(I_k\otimes \hcG_{t-u},\hcK_v )dv,%\label{brackettensordK}
\end{multline*}
and
\begin{multline*} d\Big\langle \int_0^{\cdot}(\hat K_u(I_k\otimes d\hat G_s))^\epsilon, \int_0^{\cdot}(\hat K_u(I_k\otimes dG_s))^{\hat \epsilon}\Big\rangle_v \sharp V_v\\=\hat P_1^{ \epsilon,\hat \epsilon}(I_k\otimes \hcG_v,\hcK_u) \left((id_k\otimes \tr_N)\left[\hat P_2^{ \epsilon,\hat \epsilon}(I_k\otimes \hcG_v,\hcK_u ) V_v\hat P_3^{ \epsilon,\hat \epsilon}(I_k\otimes \hcG_v,\hcK_u )\right] \otimes I_N\right)\hat P_4^{ \epsilon,\hat \epsilon}(I_k\otimes \hcG_v,\hcK_u )dv, %\label{brackettensordG}
\end{multline*}
with the same monomials $\hat P_i^{ \epsilon,\hat \epsilon}$'s. {Comparing the above terms, we observe that they are always identical with exception of having  $\tr_{NK}$ and $(id_k \otimes \tr_{N})$ when we consider terms coming from \eqref{1st derivative term Fund theo} and \eqref{2nd derivative term Fund theo}, respectively.} Therefore, we can deduce using  \eqref{Itoproduit} and \eqref{compute brackets} that for any monomials  $Q_1$ and $Q_2$ in  noncommutative indeterminates, for any finite products $P_1$, $P_2$ and $P_3$ of monomials or resolvents of nonnegative polynomials, and deterministic matrices $D_1, D_2, D_3, D_4$ (that are only introduced here for Step 3),
\begin{align*}  
&\frac{d}{dv}_{|_{v=u}}\Bigg[\int_0^{v}P_1(I_k\otimes \cG_{t-u}\cK_{s}, \hcK_{s}I_k\otimes \hcG_{t-u},\mathcal{D} )D_1 \\&\hspace{3cm} \Big\{  d \big\langle  Q_1 (I_k\otimes \cG_{t-u}\cK_{\cdot},
 \hcK_{\cdot}I_k\otimes \hcG_{t-u} , \mathcal{D}), Q_2 (I_k\otimes \cG_{t-u}\cK_{\cdot},
 \hcK_{\cdot}I_k\otimes \hcG_{t-u} , \mathcal{D})\big \rangle_s
 \\&\hspace{4cm} 
\sharp \big[D_2P_3(I_k\otimes \cG_{t-u}\cK_{s}, \hcK_{s}I_k\otimes \hcG_{t-u} ,\mathcal{D})D_3\big]\Big\}    D_4P_2(I_k\otimes \cG_{t-u}\cK_{s}, \hcK_{s}I_k\otimes \hcG_{t-u} ,\mathcal{D})\Bigg]
\\&-\frac{d}{dv}_{|_{v=t-u}}\Bigg[\int_0^{v} D_1P_1(I_k\otimes \cG_s\cK_{u}, \hcK_{u} I_k\otimes\hcG_s,\mathcal{D} )\\&\Big\{
d \big\langle  Q_1(I_k\otimes \cG_{\cdot}\cK_{u}, \hcK_{u} I_k\otimes\hcG_{\cdot} ), Q_2(I_k\otimes\cG_{\cdot}\cK_{u}, \hcK_{u} I_k\otimes\hcG_{\cdot},\mathcal{D} )  \big\rangle_s 
%\\ & \hspace{3cm}
\sharp \big[D_2P_3(I_k\otimes \cG_s\cK_{u}, \hcK_{u} I_k\otimes\hcG_s,\mathcal{D} ) D_3\big] \Big\} \\&
\hspace*{3cm} D_4 P_2(I_k\otimes \cG_s\cK_{u}, \hcK_{u} I_k\otimes\hcG_s,\mathcal{D})\Bigg]
\end{align*}
is a finite sum of terms of the form 
\begin{equation}\label{preuvecovariance}
   M_1\tr_{Nk}(M_2 )M_3- M_1 \left[(id_k\otimes tr_{N}) (M_2)\otimes I_N\right]M_3, \end{equation}
where the $M_i$'s are finite products of the $D_i$'s and of monomials or resolvents of nonnegative polynomials in the components of $\mathcal{D}$, $I_k\otimes  \cG_{t-u}, I_k\otimes\hat{\mathcal{G}}_{t-u}, \mathcal{K}_u, \hat{\mathcal {K}}_u$.
Proposition \ref{step1} readily follows from \eqref{Sigma2} using  \eqref{Itoresolvante} and \eqref{compute brackets}.
\end{proof}

\noindent 
\paragraph*{Step 2.} Having proven Proposition \ref{step1}, we now aim to express the terms of the form \eqref{Mu main prop} as covariances of normalized traces in the relevant $M$ and $\tilde{M}$ terms therein. We use an idea presented in Lemma \ref{trick} below, which is adapted from \cite{CGParraud} (see for instance Step 2 in the proof of Lemma 3.2 in \cite{ParraudGUE}). This trick enables us to rewrite the terms of \eqref{Mu main prop} obtained in Proposition \ref{step1} in terms of covariances.  Define \begin{equation}\label{def p ll'}
    p_{ll'}:=e_{ll'} \otimes I_N,
\end{equation} where $(e_{ll'})_{ll'}$ denotes the canonical basis of $M_k(\mathbb{C})$.  For any matrix-valued processes $A$ and $B$ in $M_{Nk}(\mathbb{C})$, we write
\begin{equation}\label{trick step 2}
\tr_{Nk}\big( A [(id_k\otimes \tr_{N}) (B)\otimes I_N] \big) = k \sum_{l, l'=1}^k \mathbb{E}\left[ \tr_{Nk}(A p_{ll'}) \tr_{Nk}( Bp_{l'l}) \right].    
\end{equation}
For a proof, see Step 2 in the proof of Lemma 3.2 in \cite{ParraudGUE} or Lemma 4.3  in \cite{CGParraud}. Relying on this, we obtain the following lemma, for which we provide a proof for the reader's convenience.
\begin{lemma}\label{trick}
Let $M$ and $\tilde M$ be finite products of monomials or resolvents of nonnegative polynomials in the components of $I_k\otimes  \cG_{t-u}, I_k\otimes\hat{\mathcal{G}}_{t-u}, \mathcal{K}_u, \hat{\mathcal {K}}_u$ and $I_k\otimes A^N$ defined above. Then
  \begin{eqnarray*}\lefteqn{\mathbb{E}_{\mathcal{K},\hat{\mathcal{K}}}\left[ \tr_{Nk}(M) \tr_{Nk}(\tilde M )\right]-\mathbb{E}_{\mathcal{K},\hat{\mathcal{K}}}\left[ \tr_{Nk}\left(\tilde M \{[(id_k\otimes \tr_{N}) M]\otimes I_N\} \right)\right]}\\
&=& \Cov_{\mathcal{K},\hat{\mathcal{K}}}(\tr_{Nk}(\tilde M), \tr_{Nk}( M))-k \sum_{l, l'=1}^k \Cov_{\mathcal{K},\hat{\mathcal{K}}}(\tr_{Nk}(\tilde Mp_{ll'}), \tr_{Nk}( Mp_{l'l})).
\end{eqnarray*}
\end{lemma}

\begin{proof}
Using \eqref{trick step 2}, we start by writing 
\begin{multline*} \E_{\mathcal{K},\hat{\mathcal{K}}}\big[ \tr_{Nk}(M) \tr_{Nk}(\tilde M )\big]  -\mathbb{E}_{\mathcal{K},\hat{\mathcal{K}}}\big[ \tr_{Nk}\big(\tilde M [(id_k\otimes \tr_{N}) (M)\otimes I_N] \big)\big]\\
= \sum_{l, l'=1}^k \mathbb{E}_{\mathcal{K},\hat{\mathcal{K}}}\left[ \tr_{Nk}(\tilde Mp_{ll}) \tr_{Nk}( Mp_{l'l'}) \right]
-k \sum_{l, l'=1}^k \mathbb{E}_{\mathcal{K},\hat{\mathcal{K}}}\left[ \tr_{Nk}(\tilde Mp_{ll'}) \tr_{Nk}( Mp_{l'l}) \right].
\end{multline*}
 Using the invariance of the distribution of $K,\hat K$ by unitary conjugation, we have that 
\begin{align*}
&\mathbb{E}_{\mathcal{K},\hat{\mathcal{K}}}\left[ \tr_{Nk}(  Mp_{ll}) \right]=\mathbb{E}_{\mathcal{K},\hat{\mathcal{K}}}\left[ \tr_{NK}( Mp_{11}) \right],   \qquad 
 \mathbb{E}_{\mathcal{K},\hat{\mathcal{K}}}\left[ \tr_{Nk}( Mp_{ll'}) \right]=0 \quad \text{if  } l\neq l' . \\
& \mathbb{E}_{\mathcal{K},\hat{\mathcal{K}}}\left[ \tr_{Nk}(\tilde Mp_{ll}) \right]=\mathbb{E}_{\mathcal{K},\hat{\mathcal{K}}}\left[ \tr_{NK}(\tilde Mp_{11}) \right],   \quad \,
 \mathbb{E}_{\mathcal{K},\hat{\mathcal{K}}}\left[ \tr_{Nk}(\tilde Mp_{ll'}) \right]=0 \quad \text{if  } l\neq l' . 
\end{align*} 
Therefore, we obtain that 
\begin{align*}&{\mathbb{E}_{\mathcal{K},\hat{\mathcal{K}}}\left[ \tr_{Nk}M \tr_{Nk}\tilde M \right]-\mathbb{E}_{\mathcal{K},\hat{\mathcal{K}}}\left[ \tr_{Nk}\left(\tilde M (id_k\otimes \tr_{N}) M \right)\right]}\\
&= \sum_{l, l'=1}^k \left\{\mathbb{E}_{\mathcal{K},\hat{\mathcal{K}}}\big[ \tr_{Nk}(\tilde Mp_{ll})\big]\mathbb{E}_{\mathcal{K},\hat{\mathcal{K}}}\left[ \tr_{Nk}( Mp_{l'l'}) \right]+\Cov_{\mathcal{K},\hat{\mathcal{K}}}(\tr_{Nk}(\tilde Mp_{ll}), \tr_{Nk}( Mp_{l'l'}))\right\}
\\&-k \sum_{l, l'=1}^k \left\{\mathbb{E}_{\mathcal{K},\hat{\mathcal{K}}}\big[ \tr_{Nk}(\tilde Mp_{ll'})\big]\mathbb{E}_{\mathcal{K},\hat{\mathcal{K}}}\left[ \tr_{Nk}( Mp_{l'l}) \right]+\Cov_{\mathcal{K},\hat{\mathcal{K}}}(\tr_{Nk}(\tilde Mp_{ll'}), \tr_{Nk}( Mp_{l'l}))\right\}\\
&=  k^2\mathbb{E}_{\mathcal{K},\hat{\mathcal{K}}}\big[ \tr_{Nk}(\tilde Mp_{11})\big]\mathbb{E}_{\mathcal{K},\hat{\mathcal{K}}}\left[ \tr_{Nk}( Mp_{11})\right]+\sum_{l, l'=1}^k \Cov_{\mathcal{K},\hat{\mathcal{K}}}(\tr_{Nk}(\tilde Mp_{ll}), \tr_{Nk}( Mp_{l'l'}))\\&-k^2 \mathbb{E}_{\mathcal{K},\hat{\mathcal{K}}}\big[ \tr_{Nk}(\tilde Mp_{11})\big]\mathbb{E}_{\mathcal{K},\hat{\mathcal{K}}}\left[ \tr_{Nk}(Mp_{11})\right]-k \sum_{l, l'=1}^k \Cov_{\mathcal{K},\hat{\mathcal{K}}}(\tr_{Nk}(\tilde Mp_{ll'}), \tr_{Nk}( Mp_{l'l}))\\
&=  \sum_{l, l'=1}^k \Cov_{\mathcal{K},\hat{\mathcal{K}}}(\tr_{Nk}(\tilde Mp_{ll}), \tr_{Nk}( Mp_{l'l'}))-k \sum_{l, l'=1}^k \Cov_{\mathcal{K},\hat{\mathcal{K}}}(\tr_{Nk}(\tilde Mp_{ll'}), \tr_{Nk}( Mp_{l'l}))\\
&=  \Cov_{\mathcal{K},\hat{\mathcal{K}}}(\tr_{Nk}(\tilde M), \tr_{Nk}( M))-k \sum_{l, l'=1}^k \Cov_{\mathcal{K},\hat{\mathcal{K}}}(\tr_{Nk}(\tilde Mp_{ll'}), \tr_{Nk}( Mp_{l'l})).
\end{align*}\end{proof}

\paragraph*{Step 3.} By Proposition \ref{step1}, Theorem \ref{theo:fundamental} follows once we prove that the covariances obtained in  Lemma \ref{trick} are of the order $(Nk)^{-2}.$  This will be the aim of this part. 

Recall the definitions of the % $I_k\otimes  \cG_{t-u}, I_k\otimes\hat{\mathcal{G}}_{t-u}, 
$\mathcal{K}_u, \hat{\mathcal {K}}_u$'s defined in %\eqref{Stoch Diff Eq sigma}, \eqref{stoch diff eq hat and tilde}, 
\eqref{K Stocg diff eq} and \eqref{K hat Stocg diff eq}.  Then we consider new independent multiplicative $(\lambda,\tau)$-Brownian motions $\chi_t$ and $\tilde{\chi}_t$ on $M_{Nk}(\mathbb{C})$ defined as a solution of the stochastic equations \eqref{K hat Stocg diff eq}, that have an opposite multiplication sense to $\mathcal{K}_t$. Similarly, we consider $\underline{\chi}_t$ and $\underline{\tilde{\chi}}_t$ that satisfy the equations \eqref{K Stocg diff eq}, which have an opposite multiplication sense to $\hat{\mathcal{K}}_t$.
\begin{proposition}\label{prop:est covariance}
Let $D_1$ and $D_2$ be $Nk \times Nk$ deterministic matrices and  $\Pi$ and $\tilde \Pi$ be finite products of monomials or resolvents of nonnegative polynomials in $
\mathcal{K}_u, \hat{\mathcal {K}}_u$ (defined above) and a set of deterministic $Nk \times Nk$ matrices $\mathcal {D}$. Then 
\[
\Cov\Big(\tr_{Nk}\big(\Pi(\mathcal{K}_u, \hat{\mathcal {K}}_u,\mathcal {D})D_1\big) , \tr_{Nk}\big(\tilde{\Pi}(\mathcal{K}_u, \hat{\mathcal {K}}_u,\mathcal {D})D_2\big) \Big)
\]
is a finite sum of terms of the form 
\begin{equation*}
\frac{1}{(Nk)^2}\int_0^u \E\big[ \tr_{Nk}\big(D_1 \Theta_1 D_2 \Theta_2 \big)\big] \ dv
\end{equation*}
where $\Theta_1$ and $\Theta_2$ are finite products of monomials or resolvents of nonnegative polynomials in the components of $
\mathcal{K}_v, \hat{\mathcal {K}}_v$, $\chi_{u-v}$, $\tilde{\chi}_{u-v}$,  $\underline{\chi}_{u-v}$, $\underline{\tilde{\chi}}_{u-v}$ and $\mathcal{D}$. In particular, for $D_1=D_2=I_{Nk}$,
and $Q, \tilde Q$ finite products of monomials or resolvents of nonnegative polynomials, we have that  
\begin{multline*}
\E\Cov_{|\cG,\hcG }\Big(\tr_{Nk}\big(Q(\mathcal{K}_u, \hat{\mathcal {K}}_u, I_k\otimes \mathcal{G}_{t-u}, I_k\otimes \hat{\mathcal{G}}_{t-u},I_k\otimes A^N)\big) , \tr_{Nk}\big(\tilde{Q}(\mathcal{K}_u, \hat{\mathcal {K}}_u, I_k\otimes \mathcal{G}_{t-u}, I_k\otimes \hat{\mathcal{G}}_{t-u},I_k\otimes A^N)\big) \Big)\\ = O\Big(\frac{1}{(Nk)^2}\Big).
\end{multline*}

\end{proposition}
\begin{proof}
To control the covariance, we will, as before, use interpolation. For this, we define for any $v \in [0,u]$, 
\[
h(v)= \E\Big[\tr_{Nk}\big(\Pi(\chi_{u-v}\mathcal{K}_{v}, \hat{\mathcal {K}}_v \underline{{\chi}}_{u-v},\mathcal {D})D_1\big)  \tr_{Nk}\big(\tilde{\Pi}(\tilde{\chi}_{u-v}\mathcal{K}_v, \hat{\mathcal {K}}_v\underline{\tilde{\chi}}_{u-v},\mathcal {D} )D_2\big)\Big],
\]
so that we write 
\[
\Cov\Big(\tr_{Nk}\big(\Pi(\mathcal{K}_u, \hat{\mathcal {K}}_u,\mathcal {D})D_1\big) , \tr_{Nk}\big(\tilde{\Pi}(\mathcal{K}_u, \hat{\mathcal {K}}_u,\mathcal {D})D_2\big) \Big)
=h(u)-h(0)= \int_0^u h'(v) \ dv.
\]
Computing the derivative, we have for any $s_1$ and 
\begin{align*}
    h'(v)=  &\frac{d}{ds}_{|_{s=v}} \mathbb{E}\left[\tr_{Nk}\big(\Pi(\chi_{u-v}\mathcal{K}_{s}, \hat{\mathcal {K}}_s \underline{{\chi}}_{u-v},\mathcal {D})D_1\big)  \tr_{Nk}\big(\tilde{\Pi}(\tilde{\chi}_{u-v}\mathcal{K}_s, \hat{\mathcal {K}}_s\underline{\tilde{\chi}}_{u-v},\mathcal {D} )D_2\big)\right]%\label{1st derivative term general Cov} 
    \\ &-
    \frac{d}{ds}_{|_{s=u-v}} \mathbb{E}\left[\tr_{Nk}\big(\Pi(\chi_{s}\mathcal{K}_{v}, \hat{\mathcal {K}}_v \underline{{\chi}}_{s},\mathcal {D})D_1\big)  \tr_{Nk}\big(\tilde{\Pi}(\tilde{\chi}_{s}\mathcal{K}_v, \hat{\mathcal {K}}_v\underline{\tilde{\chi}}_{s},\mathcal {D} )D_2\big)
    \right].% \label{2nd derivative term general Cov}
\end{align*}
We have that for any $s_1$ and $s_2$, 
\begin{multline*}\tr_{Nk}\big(\Pi(\chi_{s_1}\mathcal{K}_{s_2}, \hat{\mathcal {K}}_{s_2} \underline{{\chi}}_{s_1},\mathcal {D}) D_1\big) \tr_{Nk}\big(\tilde \Pi(\tilde{\chi}_{s_1}\mathcal{K}_{s_2}, \hat{\mathcal {K}}_{s_2}\underline{\tilde{\chi}}_{s_1},\mathcal {D} ) D_2\big)
\\=\frac{1}{Nk} \sum_{i,j=1}^{Nk}\tr_{Nk} \big(\mathcal{E}_{ji}\Pi(\chi_{s_1}\mathcal{K}_{s_2}, \hat{\mathcal {K}}_{s_2} \underline{{\chi}}_{s_1},\mathcal {D}) D_1 \mathcal{E}_{ij} D_2 \tilde \Pi (\tilde{\chi}_{s_1}\mathcal{K}_{s_2}, \hat{\mathcal {K}}_{s_2}\underline{\tilde{\chi}}_{s_1},\mathcal {D} )\big), \end{multline*}
where the $\mathcal E_{ij}$'s denote the unit matrices of $M_{Nk}(\mathbb{C})$.
$\Pi$ and $\tilde \Pi$ can be written, for some $p$ and $n$, as 
$$\Pi= Q^{(0)} R^{(1)}Q^{(1)}\cdots R^{(p-1)}Q^{(p-1)}, $$
$$\tilde \Pi= Q^{(p)} R^{(p+1)}Q^{(p+1)}\cdots R^{(n)}Q^{(n)}, $$
where the $Q^{(i)}$'s are monomials and the $R^{(i)}$'s are  resolvents of nonnegative polynomials $P_iP_i^*$'s.
Thus, as in the proof of Proposition \ref{step1}, deriving with respect to 
$\mathcal{K}$ and $\hat{\mathcal {K}}$ as well as deriving with respect to 
${\chi}, \tilde{\chi}, \underline{{\chi}}, \underline{\tilde{\chi}}$,
we obtain that 
$$d \Pi D_1 \mathcal{E}_{ij} D_2 \tilde \Pi = d\Sigma_1+d\Sigma_2,$$
where $\Sigma_1$ does not contribute in the computation of  $h'(v)$ and 
\begin{eqnarray*}
 d\Sigma^{(2)}&=&    \sum_{i=1, i\neq p}^n Q^{(0)}R^{(1)}Q^{(1)}\cdots Q^{(i-1)}R^{(i)}\left(d\langle P_iP_i^*, P_iP_i^*\rangle \sharp R^{(i)}\right)  R^{(i)}Q^{(i)}\cdots R^{(n)}Q^{(n)}\nonumber \\
&&+\sum_{i<j\leq 2p-2}N_t^{(1)}\cdots N^{(i-1)}\left(d<N^{(i)}, N^{(j)}> \sharp [N^{(i+1)}\cdots N^{(j-1)}]\right)N^{(j+1)}\cdots N^{(n)}\\
&&+\sum_{2p+2\leq i<j}N^{(1)}\cdots N^{(i-1)}\left(d<N^{(i)}, N^{(j)}> \sharp [N^{(i+1)}\cdots N^{(j-1)}]\right)N^{(j+1)}\cdots N^{(n)}\\
&&+\sum_{\substack{i\leq 2p-2\\j \geq 2p+2}}N^{(1)}\cdots N^{(i-1)}\left(d<N^{(i)}, N^{(j)}> \sharp [N^{(i+1)}\cdots N^{(j-1)}]\right)N^{(j+1)}\cdots N^{(n)}
\end{eqnarray*}
where, $N^{(2p)}= D_1 \mathcal{E}_{ij} D_2$, $N^{(2p-1)}=N^{(2p+1)}=1$
and 
for $i=0,\ldots,n$, $i\neq 2p-1, 2p, 2p +1,$ $N^{(2i)}=Q^{(i)}$ and $N^{(2i-1)}=R^{(i)}.$ Note that the deterministic matrix $ D_1 \mathcal{E}_{ij} D_2$ does not affect the following arguments.
Note that, the term  \eqref{preuvecovariance} vanishes for $k=1$, and therefore, the first three lines of $d\Sigma^{(2)}$ do not contribute to the computation of $h'(v)$. 
Moreover, the terms obtained by deriving  with respect to 
${\chi}, \tilde{\chi}, \underline{{\chi}}, \underline{\tilde{\chi}}$ in 
the last line vanish since $ {{\chi}},  \underline{{\chi}}$ and ${\tilde{\chi}},\underline{\tilde{\chi}}$ involve independent elliptic Brownian motions. Therefore, the only contributing terms are those obtained when deriving in the last line with respect to $\cK$ and $\hat{\mathcal {K}}$.

Hence, denoting by  $T_1 $ and $T_2$ either a monomial or a resolvent of some nonnegative polynomial and by the $\Xi_i$'s  finite products of monomials or resolvents of nonnegative polynomials in the components of $%I_k\otimes  \cG_{t-u}, I_k\otimes\hat{\mathcal{G}}_{t-u}, 
\mathcal{K}_s, \hat{\mathcal {K}}_s$, $\chi_{u-v}$, $\tilde{\chi}_{u-v}$,  $\underline{\chi}_{u-v}$, $\underline{\tilde{\chi}}_{u-v}$ and $\mathcal {D}$, we can see that $h'(v)$ is a finite sum of terms of the form
\begin{multline*}
 \frac{1}{Nk} \sum_{i,j=1}^{Nk} \E \tr_{Nk} \frac{d}{ds}_{|_{s=v}}\Big[ \int_0^s\mathcal{E}_{ji} \Xi_1 \ d\big\langle T_1 (\chi_{u-v}\mathcal{K}, \hat{\mathcal {K}} \underline{{\chi}}_{u-v},\mathcal {D}) , T_2 (\tilde{\chi}_{u-v}\mathcal{K}, \hat{\mathcal {K}}\underline{\tilde{\chi}}_{u-v},\mathcal {D} )D_2\big)  \big\rangle_s \\ \sharp  \big( \Xi_2 D_1 \mathcal{E}_{ij} D_2 \Xi_3\big) \Xi_4\Big],
\end{multline*}
which is in turn, by using \eqref{compute brackets}, \eqref{Itoproduit}, \eqref{Itoresolvante}  and  \eqref{crochetsgeneraux}, a finite sum of terms of the form
\[
\frac{1}{Nk} \sum_{i,j=1}^{Nk} \E \Big[  \tr_{Nk} \big( \mathcal{E}_{ji}\Theta_1 \big) \tr_{Nk} \big( D_1 \mathcal{E}_{ij} D_2\Theta_2 \big)\Big] = \frac{1}{(Nk)^2}\E \tr_{Nk} \big(D_2 \Theta_2 D_1 \Theta_1 \big).
\]  
where the $\Theta_i$'s are finite products of monomials or resolvents of nonnegative polynomials in the components of $\mathcal{K}_v, \hat{\mathcal {K}}_v$, $\chi_{u-v}$, $\tilde{\chi}_{u-v}$,  $\underline{\chi}_{u-v}$, $\underline{\tilde{\chi}}_{u-v}$ and $\mathcal {D}$.  
This ends the proof of the first statement.  In order to prove the second claim, we need to prove that for any finite products of monomials or resolvents of nonnegative polynomials in the components of $\mathcal{K}_v, \hat{\mathcal {K}}_v$, $\chi_{u-v}$, $\tilde{\chi}_{u-v}$,  $\underline{\chi}_{u-v}$, $\underline{\tilde{\chi}}_{u-v}$, $I_k \otimes \cG_{t-u} $,  $I_k \otimes \hcG_{t-u} $ and $I_k\otimes A^N$, there exists a constant $C_t$ such that 
\[
\mathbb{E} [{\tr_{Nk}(\Pi)}]\leq C_t.
\]
Note that $\Pi$ is an expression of the form 
\[\Pi= R_0^{\ell_0} \prod_{i=1}^m P_{i,1}\dots P_{i,m_i} R_i^{\ell_i} \]
where $m\geq 1$, $\ell_0,\dots ,\ell_m , k_1, \dots, k_m \in \mathbb{N}$, the $R_i$'s are resolvents at points $z_i$'s of nonnegative polynomials in  variables listed above while the $P_{i,j}$'s are  polynomials in only a single one of those variables. Then by H\"older's inequality we obtain
\begin{align}\label{bounddness}
\mathbb{E} [{\tr_{Nk}(\Pi)}]\leq \frac{1}{\left|\Im (z_0)\right|^{\ell_0}} \cdots \frac{1}{\left|\Im (z_m)\right|^{\ell_m}} \prod_{i=1}^m \prod_{j=1}^{k_i}\| P_{i,j}\|_{L_p},    
\end{align}
where $p$ is the sum of the degrees of the polynomials $P_{i,j}$ for $i=1,\dots,m$ and $j=1,\dots ,m_i$. We recall that the $A^N_i$'s are uniformly bounded over $N$ in operator norm and end the proof by noting that Corollary \ref{NormeLp} guarantees the existence of a constant $C_{i,j}(t)$ such that $\| P_{i,j}\|_{L_p} \leq C_{i,j}(t)$. 
\end{proof}

\begin{cor}
 Consider the matrices $p_{ll'}$  defined in \eqref{def p ll'} and let $M$ and $\tilde M$ be finite products of monomials or resolvents of nonnegative polynomials in the components of $I_k\otimes  \cG_{t-u}, I_k\otimes\hat{\mathcal{G}}_{t-u}, \mathcal{K}_u, \hat{\mathcal {K}}_u$ and $I_k\otimes A^N$, defined above. Denoting by $\mathbb{E}_{|{\cG, \hcG}}$, the conditional expectation with respect to $\{\cG, \hcG\}$, we obtain 
\begin{multline*}
 k \sum_{l, l'=1}^k \mathbb{E} \Cov_{|{\cG, \hcG}}(\tr_{Nk}(\tilde Mp_{ll'}), \tr_{Nk}( Mp_{l'l}))  
 \\
  =\E\int_0^u dv\frac{1}{N^2} \tr_N\left[\mathbb{E}_{|{\cG, \hcG}}\left[
(tr_k\otimes id_N) \left( \Pi_1\right) \right]\mathbb{E}_{|{\cG, \hcG}} \left[(\tr_k\otimes id_N) \left(\Pi_2\right)\right]\right]+O\Big(\frac{1}{Nk^2}\Big) ,
\end{multline*}
where the ${\Pi_i}$'s  are finite sums of finite  products of monomials or resolvents of nonnegative polynomials in the components of $I_k\otimes  A^N$, $ I_k\otimes  \cG_{t-u},$ $ I_k\otimes\hat{\mathcal{G}}_{t-u}, 
\mathcal{K}_v$, $\hat{\mathcal {K}}_v$, $\chi_{u-v}$, $\tilde{\chi}_{u-v}$,  $\underline{\chi}_{u-v}$ and $\underline{\tilde{\chi}}_{u-v}$.    
\end{cor}

\begin{proof}
Using Proposition \ref{prop:est covariance} with $D_1=p_{ll'}$ and $D_2=p_{l'l}$, we have that
\[ k \sum_{l, l'=1}^k \Cov_{|{\cG, \hcG}}(\tr_{Nk}(\tilde Mp_{ll'}), \tr_{Nk}( Mp_{l'l}))  \]
is a finite sum of terms of the form 
\begin{multline*}
   \frac{k}{(Nk)^2} \sum_{l, l'=1}^k\int_0^u \E_{|{\cG, \hcG}} \big[ \tr_{Nk}\big(p_{ll'} \Pi_1 p_{l'l} \Pi_2 \big)\big] \ dv \\= \frac{1}{N^2}\int_{0}^u \mathbb{E}_{|{\cG, \hcG}} \, tr_{N}
\big[(\tr_k\otimes id_N )\left( \Pi_1\right)(\tr_k\otimes id_N )\left( \Pi_2\right)\big]dv,
\end{multline*}
where $\Pi_1$ and $\Pi_2$ are finite products of monomials or resolvents of nonnegative polynomials in the components of  $I_k\otimes A^N$, $I_k\otimes  \cG_{t-u}, I_k\otimes\hat{\mathcal{G}}_{t-u}, 
\mathcal{K}_v, \hat{\mathcal {K}}_v$, $\chi_{u-v}$, $\tilde{\chi}_{u-v}$,  $\underline{\chi}_{u-v}$ and $\underline{\tilde{\chi}}_{u-v}$.  
 Now, denoting by the $E_{i,j}$'s the unit $N\times N$ matrices, we write 
\begin{multline*}
   \E_{|{\cG, \hcG}} 
\big[(\tr_k\otimes id_N )\left( \Pi_1\right)(\tr_k\otimes id_N )\left( \Pi_2\right)\big] 
\\= N^2 \sum_{i,j,i',j'=1}^N  \E_{|{\cG, \hcG}} \Big[ \tr_{Nk}\big(\Pi_1(I_k \otimes E_{j,i})\big) \tr_{Nk}\big(\Pi_2(I_k \otimes E_{j',i'})\big)\Big]E_{i,j}E_{i',j'},
\end{multline*}
so that
 \begin{align*}
    &\frac{1}{N^2} \E\big[ tr_{N}
\big[(\tr_k\otimes id_N )\left( \Pi_1\right)(\tr_k\otimes id_N )\left( \Pi_2\right)\big]\big]  \\&= \frac{1}{N}
\sum_{i,j=1}^N
\mathbb{E}\left[tr_{kN} \left( \Pi_1 (I_k\otimes E_{j,i})\right)tr_{kN} \left(  \Pi_2(I_k\otimes E_{i,j})\right)\right],
\\&=
\frac{1}{N}
\sum_{i,j=1}^N \E\Big[ \E_{|\cG,\hcG}\big[\tr_{Nk}\big(\Pi_1 (I_k\otimes E_{j,i}) \big)\big]\E_{|\cG,\hcG}\big[\tr_{Nk}\big(\Pi_2 (I_k\otimes E_{i,j}) \big)\big]\Big]
\\ &\qquad+\sum_{q\in \text{~finite set~} I} \frac{1}{N^3k^2} \sum_{i,j=1}^N \E \int_0^u \E_{|\cG,\hcG }\big[ \tr_{Nk}\big((I_k\otimes E_{j,i})\Theta_1^{(q)}(I_k\otimes E_{i,j}) \Theta_2^{(q)} \big)\big] \ dv ,
\end{align*}
where the last equality follows by Proposition \ref{prop:est covariance} with  $\Theta_1^{(q)}$'s and $\Theta_2^{(q)}$'s being finite products of monomials or resolvents of nonnegative polynomials in the components of $I_k\otimes A^N$, $I_k\otimes  \cG_{t-u}, I_k\otimes\hat{\mathcal{G}}_{t-u}, 
\mathcal{K}_v, \hat{\mathcal {K}}_v$, $\chi_{u-v}$, $\tilde{\chi}_{u-v}$,  $\underline{\chi}_{u-v}$ and $\underline{\tilde{\chi}}_{u-v}$. Then, by the fact that the $A^N_i$'s are uniformly bounded over $N$ in operator norm, Corollary \ref{NormeLp} and the same arguments used to establish the bound in \eqref{bounddness}, we have that
 \begin{align*}
    &\frac{1}{N^2} \E\big[ tr_{N}
\big[(\tr_k\otimes id_N )\left( \Pi_1\right)(\tr_k\otimes id_N )\left( \Pi_2\right)\big]\big]  \\& =
\frac{1}{N}
\sum_{i,j=1}^N \E\Big[ \E_{|\cG,\hcG}\big[\tr_{Nk}\big(\Pi_1 (I_k\otimes E_{j,i}) \big)\big]\E_{|\cG,\hcG}\big[\tr_{Nk}\big(\Pi_2 (I_k\otimes E_{i,j}) \big)\big]\Big]
+ O\Big(\frac{1}{Nk^2} \Big) 
\\& = \sum_{i,i',j,j'=1}^N \E\Big[ \E_{|\cG,\hcG}\big[\tr_{Nk}\big(\Pi_1 (I_k\otimes E_{j,i}) \big)\big]\E_{|\cG,\hcG}\big[\tr_{Nk}\big(\Pi_2 (I_k\otimes E_{j',i'}) \big)\big]\Big]\tr_{N}[E_{i,j}E_{i',j'}]
+ O\Big(\frac{1}{Nk^2} \Big)
\\&= \frac{1}{N^2}\E \tr_{N}\Big[\E_{|\cG,\hcG}\big[(\tr_k\otimes id_N)(\Pi_1)\big] \E_{|\cG,\hcG}\big[(\tr_k\otimes id_N)(\Pi_2)\big]\Big] + O\Big(\frac{1}{Nk^2} \Big).
\end{align*}

\paragraph*{Step 4.}
The last step of the proof consists of proving that the term 
\[
\E \tr_{N}\Big[\E_{|\cG,\hcG}\big[(\tr_k\otimes id_N)(\Pi_1)\big] \E_{|\cG,\hcG}\big[(\tr_k\otimes id_N)(\Pi_2)\big]\Big]
\]
is bounded uniformly in $k$. To do this, we proceed by applying  Cauchy-Schwarz inequality and use the fact that the conditional expectation $\E_{|\cG,\hcG}(\tr_k\otimes id_N)$ is a contraction to obtain 
\[
\Big| \E \tr_{N}\Big[\E_{|\cG,\hcG}\big[(\tr_k\otimes id_N)(\Pi_1)\big] \E_{|\cG,\hcG}\big[(\tr_k\otimes id_N)(\Pi_2)\big]\Big] \Big| \leq \big\| \Pi_1\big\|_{L_2}\big\| \Pi_2\big\|_{L_2} \leq C_t, 
\] 
for some constant $C_t>0$. The final inequality follows from Corollary \ref{NormeLp}, the fact the $A^N_i$'s are uniformly bounded in norm over $N$ and the same arguments used to establish the bound in \eqref{bounddness}. This concludes the proof of Theorem~\ref{theo:fundamental}, obtained by assembling the arguments presented in the preceding steps. 
\end{proof}

\section{Appendix}
\subsection{Proof of Theorem~\ref{th:weakconv}}\label{proofofweakconv}

\paragraph*{Pure trace polynomials.}
The proof follows the approach of~\cite{cebron2013free}, and uses as the main tool the algebra of \emph{pure trace polynomials}, which is the vector space of formal linear combinations of the form $\tr(P_1)\cdots \tr(P_n)$, where $P_1,\ldots,P_n$ are noncommutative polynomials in indeterminate variables and $\tr$ is an indeterminate linear functional. This space has been introduced by Procesi~\cite{procesi1976invariant} in 1976 for the study of unitarily invariant polynomials over matrices (see also~\cite{leron1976trace,razmyslov1974trace,razmyslov1987trace}), and it has been reintroduced regularly since then for the study of large random matrices and free probability~\cite{cebron2013free,cebron2014fluctuations,driver2013large,jekel2022tracial,kemp2016large,kemp2017heat,rains1997combinatorial,sengupta2008traces}.

Let $p\in \mathbb{N}$. Formally, the algebra $\tr\ \mathbb{C}\{X_{\ell},X_{\ell}^*,X_{\ell}^{-1},X_{\ell}^{-1,*}:1\leq \ell \leq p\}$ of \emph{pure trace polynomials} is defined as the symmetric tensor algebra over the vector space $\mathbb{C}\langle X_{\ell},X_{\ell}^*,X_{\ell}^{-1},X_{\ell}^{-1,*}:1\leq \ell \leq p\rangle$ of polynomials of $4p$~noncommutative indeterminates $X_{\ell},X_{\ell}^*,X_{\ell}^{-1},X_{\ell}^{-1,*}$ (with $1\leq \ell \leq p$). A pure tensor which is the tensor product of $n$~polynomials $P_1,\ldots,P_n \in \mathbb{C}\langle X_{\ell},X_{\ell}^*,X_{\ell}^{-1},X_{\ell}^{-1,*}:1\leq \ell \leq p\rangle $ will be written as $\tr(P_1)\cdots \tr(P_n)$. Thus, the standard basis of $\tr\ \mathbb{C}\{X_p,X_p^*,X_p^{-1},X_p^{-1,*}:1\leq \ell \leq p\}$ is the set
$$\{\tr(M_1)\cdots \tr(M_n): n \in \mathbb{N}, M_1, \cdots,M_n\text{ are monomials of }\mathbb{C}\langle X_{\ell},X_{\ell}^*,X_{\ell}^{-1},X_{\ell}^{-1,*}:1\leq \ell \leq p\rangle\}.$$
As $\tr\ \mathbb{C}\{X_{\ell},X_{\ell}^*,X_{\ell}^{-1},X_{\ell}^{-1,*}:1\leq \ell \leq p\}$ is the tensor algebra of the graded vector space of polynomials (graded by degree), it is itself a graded vector space. It allows to define canonically the \textit{degree} of any pure trace polynomial as follows: for all monomials $M_1, \cdots,M_n$, the degree of $\tr(M_1)\cdots \tr(M_n)$ is the product of the degrees of the monomials, and  the degree of any pure trace polynomial is the highest of the degrees of each non-zero component in the standard basis.
For any $d\in \mathbb{N}$, we will need to consider the vector space of pure trace polynomials with degrees less or equal to $d$. It is a finite-dimensional vector space, and we denote it by $E_d$.

Let us define the $\tr\ \mathbb{C}\{X_{\ell},X_{\ell}^*,X_{\ell}^{-1},X_{\ell}^{-1,*}:1\leq \ell \leq p\}$-calculus on any {*-probability} space~$(\mathcal{A},\varphi)$. For any $P=\tr(M_1)\cdots \tr(M_n)\in \tr\ \mathbb{C}\{X_{\ell},X_{\ell}^*,X_{\ell}^{-1},X_{\ell}^{-1,*}:1\leq \ell \leq p\}$ and any $p$-tuple $a=(a_1,\ldots,a_p)$ of invertible elements in $\mathcal{A}$, we denote by $P (a)$ the quantity
$$P (a) = \varphi\big(M_1(a)\big) \cdots \varphi\big(M_n (a)\big)\in \mathbb{C},$$
where $M_1(a), \ldots, M_p(a)$ are given by the polynomial calculus, and we extend this notation to all pure trace polynomials by linearity. In particular, given a multiplicative $(\lambda,\tau)$-Brownian motion $\G_t:= \big(G_t^{(1)} , \dots, G_t^{(p)} \big)$, a time $t\geq 0$, and a pure trace polynomial $P=\tr(P_1)\cdots \tr(P_n)$ we get
$$P (\G_t) = \tr_N\big(P_1(\G_t)\big) \cdots \tr_N\big(P_n (\G_t)\big).$$

\paragraph*{Computation of the generator of the Brownian motion.}
Using the stochastic calculus on $GL(N,\mathbb{C})$, and in particular Lemma~\ref{Lemma:Calcul QuadraticCov}, we can compute the evolution of the expectation of $\mathbb{E}[P (G_t)]$. 

First, let us define two operators $\Delta$ and $\tilde{\Delta}$ which act on $\tr\ \mathbb{C}\{X_{\ell},X_{\ell}^*,X_{\ell}^{-1},X_{\ell}^{-1,*}:1\leq \ell \leq p\}$ and which will appear in the computation of the generator of $G_t$.

Let $P= \tr(M_1)\cdots \tr(M_n)$ be a pure trace polynomial, where for each $i=1,\ldots,n$, the monomial $M_i$ is given by
$$M_i=X_{\ell_{1,i}}^{\epsilon_{1,i}}\cdots X_{\ell_{m_i,i}}^{\epsilon_{m_i,i}},$$
with $\ell_{u,v}\in\{1,\ldots,p\}$ and $\epsilon_{u,v}\in\{\cdot, *,-1, (-1,*)\}.$ Define the operator
\begin{align}\Delta(P)=& \sum_{\ell=1}^p \frac{\sigma_l^2}{2} \{d_\ell (\tau-\lambda)+d_\ell^* (\bar \tau -\lambda) +d_{\ell}^{(-1)} (\tau-\lambda) + d_{\ell}^{(-1,*)}(\bar\tau-\lambda)\}P\label{premiertermeDelta}\\&+\sum_{i=1}^n \prod_{i' \neq i} \tr(M_{i'})\sum_{k<k'} \sigma^2_{l_k,i} \delta_{\ell_{k,i},\ell_{k',i}}\nonumber\\ &\quad\tr\left(
X_{\ell_{1,i}}^{\epsilon_{1,i}}\cdots X_{\ell_{k-1,i}}^{\epsilon_{k-1,i}}
Q_1^{(\epsilon_{k,i}, \epsilon_{k',i})}(X_{\ell_{k,i}})
X_{\ell_{k'+1,i}}^{\epsilon_{k'+1,i}}\cdots X_{\ell_{m_i,i}}^{\epsilon_{m_i,i}}\right) \tr\left(
Q_2^{(\epsilon_{k,i}, \epsilon_{k',i})}(X_{\ell_{k,i}})
X_{\ell_{k+1,i}}^{\epsilon_{k+1,i}}\cdots X_{\ell_{k'-1,i}}^{\epsilon_{k'-1,i}}\right),\label{derniertermeDelta}
\end{align}
where, for $\ell\in\{1,\ldots,p\}$ and $\epsilon\in\{\cdot, *,-1, (-1,*)\},$
$d_\ell^{\epsilon}$ is the degree of $X_l^\epsilon$ in $ P$ and $Q_1^{(\epsilon_1,\epsilon_2)}$ and $Q_2^{(\epsilon_1,\epsilon_2)}$ are given by
%%%%%%%%%%%
{\small \begin{center}
\begin{tabular}{l|l|l}
$Q_1^{(\cdot,\cdot)}(X) = (\tau-\lambda) X$ & $Q_1^{(\cdot,-1)}(X)  = - (\tau-\lambda)  I $ &$Q_1^{(\cdot,\ast)}(X) = \lambda XX^\ast$
\\ 
  $Q_1^{(\cdot,({-1,\ast}))}(X)  = - \lambda X$ &
$Q_1^{(-1,\cdot)}(X)  = - (\tau-\lambda)  I$ & $Q_1^{(-1,-1)}(X) =  (\tau-\lambda) X^{-1} $
\\ 
$Q_1^{(-1,\ast)}(X)  = - \lambda  X^\ast$ & $Q_1^{(-1,({-1,\ast}))}=  \lambda  I$&
$Q_1^{(\ast,\cdot)}(X)  = \sigma^2 \lambda  I$\\  $Q_1^{(\ast,-1)}(X)  = - \lambda X^{-1} $
&
$Q_1^{(\ast,\ast)}(X) =  (\bar{\tau}-\lambda)  X^\ast$ & $Q_1^{(\ast,({-1,\ast}))}(X) = - (\bar{\tau}-\lambda) I$\\ 
{\small $Q_1^{(({-1,\ast}),\cdot)}(X)  =  - \lambda X^{-1,\ast} $} & {\small
$Q_1^{(({-1,\ast}), -1)}(X)  =  \lambda X^{-1,\ast} X^{-1} $}
& 
{\small $Q_1^{(({-1,\ast}),\ast)}(X)  = - (\bar{\tau}-\lambda)  I$ } \\ & {\small $Q_1^{(({-1,\ast}),({-1,\ast}))}(X)  =  (\bar{\tau}-\lambda)  X^{-1,\ast}$}
\end{tabular}

\end{center} }
and 
%%%%%%%%%%%%%%%%%%
{\small \begin{center}
\begin{tabular}{l|l|l|l}
$Q_2^{(\cdot,\cdot)}(X) =  X$ & $Q_2^{(\cdot,-1)}(X)  =  I $
&
$Q_2^{(\cdot,\ast)}(X) = I$ & $Q_2^{(\cdot,({-1,\ast}))}(X)  =  X^{-1,\ast}$\\
$Q_2^{(-1,\cdot)}(X)  =  I$ & $Q_2^{(-1,-1)}(X) =  X^{-1} $
& 
$Q_2^{(-1,\ast)}(X)  =  X^{-1}$ & $Q_2^{(-1,({-1,\ast}))}= X^{-1,\ast} X^{-1}$\\
$Q_2^{(\ast,\cdot)}(X)  = XX^\ast$ & $Q_2^{(\ast,-1)}(X)  =  X^{\ast} $
&
$Q_2^{(\ast,\ast)}(X) =   X^\ast$ & $Q_2^{(\ast,({-1,\ast}))}(X) =X^{-1,\ast}X^\ast$\\
{\small $Q_2^{(({-1,\ast}),\cdot)}(X)  =  X $} & {\small
$Q_2^{(({-1,\ast}), -1)}(X)  = I $}
& 
{\small $Q_2^{(({-1,\ast}),\ast)}(X)  =   I$ } & {\small $Q_2^{(({-1,\ast}),({-1,\ast}))}(X)  =  X^{-1,\ast}.$}
\end{tabular}

\end{center} }
We also define the operator $\tilde \Delta$ given by
\begin{align*}
    \tilde \Delta P&= 
\sum_{\ell=1}^p\sigma_{\ell}^2 \sum_{1\leq i<i'\leq n} \prod_{i^{\prime\prime}\neq i,i'}\tr(M_{i^{\prime\prime}})\\&\Bigg\{
(\tau-\lambda)\sum_{\substack{ M_i =AX_\ell B\\M_{i'} =CX_{\ell} D}}
\tr(AX_\ell DCX_\ell B)
%%%%%
+\lambda\sum_{\substack{M_i =AX_\ell B\\M_{i'} =CX_{\ell}^\ast D}}
\tr(AX_\ell X_\ell^\ast DCB )\\
&+(\lambda-\tau)\sum_{\substack{ M_i =AX_\ell B\\M_{i'} =CX_{\ell}^\ast D}}
\tr(A DC B)-\lambda
\sum_{\substack{M_i =AX_\ell B\\M_{i'} =CX_{\ell}^{-1,\ast} D}}
\tr(AX_\ell DCX_\ell^{-1,\ast} B)\\&
+\lambda\sum_{\substack{ M_i =AX_\ell^\ast B\\M_{i'} =CX_{\ell} D}}
\tr(ADCX_\ell X_\ell^\ast B)+(\tau-\lambda)\sum_{\substack{ M_i =AX_\ell^\ast B\\M_{i'} =CX_{\ell}^\ast D}}
\tr(AX_\ell^\ast DCX_\ell^\ast B)\\&
-\lambda\sum_{\substack{ M_i =AX_\ell^\ast B\\M_{i'} =CX_{\ell}^{-1} D}}
\tr(AX_\ell^{-1} DCX_\ell^\ast B)+(\lambda-\bar \tau)\sum_{\substack{ M_i =AX_\ell^\ast B\\M_{i'} =CX_{\ell}^{-1,\ast} D}}
\tr(ADCX_\ell^{-1,\ast} X_\ell^\ast B)\\
&+ (\lambda-\tau)\sum_{\substack{M_i =AX_\ell^{-1} B\\M_{i'} =CX_{\ell} D}}
\tr(A DC B)-\lambda \sum_{\substack{M_i =AX_\ell^{-1} B\\M_{i'} =CX_{\ell}^\ast D}}
\tr(AX_\ell^\ast DCX_\ell^{-1} B)\\
&+(\tau-\lambda)\sum_{\substack{M_i =AX_\ell^{-1} B\\M_{i'} =CX_{\ell}^{-1} D}}
\tr(AX_\ell^{-1} DCX_\ell^{-1} B)+\lambda \sum_{\substack{M_i =AX_\ell^{-1} B\\M_{i'} =CX_{\ell}^{-1,\ast} D}}
\tr(ADCX_\ell^{-1,\ast} X_\ell^{-1} B)\\
& -\lambda \sum_{\substack{M_i =AX_\ell^{-1,\ast} B\\M_{i'} =CX_{\ell} D}}
\tr(AX_\ell^{-1,\ast} DCX_\ell B)+(\lambda-\bar \tau)\sum_{\substack{M_i =AX_\ell^{-1,\ast} B\\M_{i'} =CX_{\ell}^\ast D}}
\tr(A DC B)\\
&+\lambda \sum_{\substack{M_i =AX_\ell^{-1,\ast} B\\M_{i'} =CX_{\ell}^{-1}D}}
\tr(AX_\ell^{-1,\ast} X_\ell^{-1}DC B)+(\bar \tau-\lambda)\sum_{\substack{M_i =AX_\ell^{-1,\ast} B\\M_{i'} =CX_{\ell}^{-1,\ast} D}}
\tr(AX_\ell^{-1,\ast} DCX_\ell^{-1,\ast} B)\Bigg\}.
\end{align*}
With the above definitions of the operators $\Delta$ and $\tilde\Delta$, we prove the following result.
\begin{lemma}\label{Generator}
We have
$$\frac{d}{ds}\mathbb{E}\left[P(\G_s)\right]=\mathbb{E}\left[\left((\Delta+\frac{1}{N^2}\tilde{\Delta})P\right)(\G_s)\right].$$
\end{lemma}
\begin{proof}
We have 
\begin{align*}
d\mathbb{E}\left[P(\G_s)\right]&= \frac{1}{N^n} \sum_{j_1,\ldots j_n=1}^N d
\mathbb{E}\left[\Tr \left( E_{j_n j_1} M_1(\G_s)E_{j_1 j_2} M_2(\G_s)\cdots E_{j_{n-1} j_n} M_n(\G_s)\right)\right]\\
&= \underbrace{\frac{1}{N^n} \sum_{j_1,\ldots j_n=1}^N
\sum_{i=1}^n\mathbb{E}\left[\Tr \left( E_{j_nj_1} M_1(\G_s)\cdots E_{j_{i-1} j_i} dM_i(\G_s) E_{j_{i} j_{i+1}} M_{i+1}(\G_s)\cdots \right)\right]}_{(I)}\\
 & \hspace{0.5cm} +
 \frac{1}{N^n}
\sum_{1\leq i<i'\leq n}
\mathbb{E}\left[\Tr \left( E_{j_n j_1} \cdots E_{j_{i-1} j_i} \right.\right.\\&
 \hspace{1cm}\cdot\left(d\langle M_i(\G)\otimes M_{i'}(G)\rangle_s \sharp \left[ E_{j_{i} j_{i+1}} M_{i+1}(\G_s)\cdots
E_{j_{i'-1} j_{i'}}\right]\right)  \cdot \left.\left.E_{j_{i'} j_{i'+1}} M_{i'+1}(\G_s)\cdots\right)\right]\\&\hspace{0.5cm}\underbrace{\hspace{14cm}}_{(II)}
\end{align*}
  where we recall that the $E_{ij}$'s are the unit $N\times N $ matrices and we used \eqref{Itoproduit} in the last equality. Now, for each $i\in\{1,\ldots, n\}$, we have by using \eqref{Itoproduit} \begin{eqnarray} 
 dM_i(\G_s)&=&\sum_{k=1}^{m_i} G_{\ell_{1,i}}^{\epsilon_{1,i}}(s)\cdots 
  G_{\ell_{k-1,i}}^{\epsilon_{k-1,i}}(s)
  (dG_{\ell_{k,i}}^{\epsilon_{k,i}}(s))
  G_{\ell_{k+1,i}}^{\epsilon_{k+1,i}}(s)\cdots G_{\ell_{m_i,i}}^{\epsilon_{m_i,i}}(s)\label{memenombredetraces}\\&&
  +  \sum_{\substack{1\leq k<k'\leq m_i\\\ell_k(i) =\ell_{k'}(i)}}
  G_{\ell_{1,i}}^{\epsilon_{1,i}}(s)\cdots 
  G_{\ell_{k-1,i}}^{\epsilon_{k-1,i}}(s)
  \nonumber\\&&\left\{d\langle G_{\ell_{k,i}}^{\epsilon_{k,i}}\otimes G_{\ell_{k',i}}^{\epsilon_{k',i}}\rangle_s \sharp 
  \left[ G_{\ell_{k+1,i}}^{\epsilon_{k+1,i}}(s)\cdots 
  G_{\ell_{k'-1,i}}^{\epsilon_{k'-1,i}}(s)\right]\right\} G_{\ell_{k'+1,i}}^{\epsilon_{k'+1,i}}(s)\cdots 
  G_{\ell_{m_i,i}}^{\epsilon_{m_i,i}}(s).
  \label{creationdunetrace}
  \end{eqnarray}
  The drift term of \eqref{memenombredetraces}, introduced in (I), provides obviously 
  the first term \eqref{premiertermeDelta} in the formula for $\Delta$. 
  It is easy to deduce from Lemma \ref{Lemma:Calcul QuadraticCov} that 
  \eqref{creationdunetrace}, introduced in (I), increases the number of traces by 1 and provides   the second  term \eqref{derniertermeDelta} in the formula for $\Delta$.  Using \eqref{compute brackets}  and Lemma \ref{Lemma:Calcul QuadraticCov}, one can easily check that (II) gives rise to $\frac{1}{N^2} \tilde \Delta$. The factor $\frac{1}{N^2}$ in front of $\tilde \Delta$ comes from the fact that the initial  number of traces in $P$ is reduced by 1  whereas the brackets of the $G_{\ell}^\epsilon$'s provide  an additional $1/N$ .
\end{proof}

\paragraph*{Large-$N$ limit.}
Let $d\in \mathbb{N}$ be a fixed degree.
Note that $\Delta$ and $\tilde{\Delta}$ are graded linear maps that leave invariant the finite-dimensional space $E_d$. As a consequence, for $s\geq 0$ and $N\in \mathbb{N}$, the exponentiations $e^{s\Delta}$ of $s\Delta$ and $e^{s(\Delta+\frac{1}{N^2}\tilde{\Delta})}$ of $s(\Delta+\frac{1}{N^2}\tilde{\Delta})$ are well-defined as operators acting on  $E_d$.

Using \cite[Lemma 2.5]{cebron2013free}, we get the following result as a direct corollary of Lemma~\ref{Generator}.

\begin{lemma}
Let $P\in E_d$, $s\geq 0$ and $N\in \mathbb{N}$. We have
$$\mathbb{E}\left[P(\G_s)\right]=\left(e^{s(\Delta+\frac{1}{N^2}\tilde{\Delta})}P\right)(I_N).$$
\end{lemma}
It allows one to get the large-$N$ limit with a control in $1/N^2$ as wanted.

\begin{proposition}\label{prop:convergence}
    Let $P\in E_d$ be pure trace polynomial, and let $t\geq 0$. There exists a constant $C>0$ such that, for all $0\leq s \leq t$ and all $N\in \mathbb{N}$, we have
$$\left|\mathbb{E}\left[P(\G_s)\right]-\left(e^{s\Delta}P\right)(1)\right|\leq \frac{C}{N^2}.$$
Moreover, there exists a 
constant $C'>0$ such that, for all $0\leq s \leq t$, we have
$$\left|\mathbb{E}\left[P(\G_s)\right]\right|\leq C'.$$
\end{proposition}
\begin{proof}We endow the finite-dimensional vector space $E_d$ with any Euclidean norm $\|\cdot\|_{E_d}$.
The Duhamel's formula (see for example \cite[Lemma 3.8.]{cebron2014fluctuations}) gives us that
$$e^{s(\Delta+\frac{1}{N^2}\tilde{\Delta})}-e^{s\Delta}= \frac{1}{N^2}\int_0^se^{u(\Delta+\frac{1}{N^2}\tilde{\Delta})}\tilde{\Delta}e^{(s-u)\Delta}du,$$
from which we can estimate the operator norm
$$\left\|e^{s(\Delta+\frac{1}{N^2}\tilde{\Delta})}-e^{s\Delta} \right\|\leq \frac{1}{N^2}e^{s\|\Delta\|+\frac{s}{N^2}\|\tilde{\Delta}\|}s\|\tilde{\Delta}\|\leq \frac{1}{N^2}e^{t\|\Delta\|+t\|\tilde{\Delta}\|}t\|\tilde{\Delta}\|.$$
As a consequence, setting $C'':=e^{t\|\Delta\|+t\|\tilde{\Delta}\|}t\|\tilde{\Delta}\|$, we have
$$\left\|e^{s(\Delta+\frac{1}{N^2}\tilde{\Delta})}P-e^{s\Delta}P \right\|_{E_d}\leq\frac{1}{N^2}C''\|P\|_{E_d}.$$
Because the traces are normalized, the maps $P\mapsto P(I_N)$ and $P\mapsto P(1)$ are the same linear form. We denote it by $\phi:E_d\to \mathbb{C}$. As it is a linear form on a finite-dimensional space, it is bounded, and can compute
\begin{align*}
    \left|\mathbb{E}\left[P(\G_s)\right]-\left(e^{s\Delta}P\right)(1)\right|&=\left|e^{s(\Delta+\frac{1}{N^2}\tilde{\Delta})}P(I_N)-e^{s\Delta}P(1)\right|\\
    &\leq \|\phi\|\cdot \left\|e^{s(\Delta+\frac{1}{N^2}\tilde{\Delta})}P-e^{s\Delta}P \right\|_{E_d}\\
    &\leq \frac{1}{N^2}\|\phi\|\cdot C''\cdot \|P\|_{E_d}
\end{align*}
which allows to conclude by setting $C:=\|\phi\|\cdot C''\cdot \|P\|_{E_d}$.    Similarly, we have
    $$|e^{s\Delta}P(1)|\leq \|\phi\|\cdot e^{s\|\Delta\|}\cdot \|P\|_{E_d}\leq \|\phi\|\cdot e^{t\|\Delta\|}\cdot \|P\|_{E_d},$$
    which allows to bound
    $$ \left|\mathbb{E}\left[P(\G_s)\right]\right|\leq  \left|\mathbb{E}\left[P(\G_s)\right]-\left(e^{s\Delta}P\right)(1)\right|+|e^{s\Delta}P(1)|\leq \frac{C}{N^2}+\|\phi\|\cdot e^{t\|\Delta\|}\cdot \|P\|_{E_d}:=C'$$
    uniformly for $s\leq t$ and $N\in \mathbb{N}$.
\end{proof}
\begin{cor}
    \label{NormeLp}
    Let $t\geq 0$, $q\in \mathbb{N}$ and let $P\in \mathbb{C}\langle X_{\ell},X_{\ell}^*,X_{\ell}^{-1},X_{\ell}^{-1,*}:1\leq \ell \leq p\rangle$ be any polynomial in $4p$ variables. There exists a constant $C>0$ such that, for all $0\leq s \leq t$ and all $N\in \mathbb{N}$, we have
    $$\left\|P(\G_s)\right\|_{L_{2q}}:=\left| \mathbb{E}\left[\tr_N\left(\Big(P(\G_s)P(\G_s)^*\Big)^q\right)\right]\right|^{1/2q} \leq C.$$
\end{cor}
\begin{proof}
    Note that if $d$ is the degree of $P$, there exists a pure trace polynomial $Q\in E_{2d}$ such that
$$\tr_N\left(\Big(P(\G_s)P(\G_s)^*\Big)^q\right)=Q(\G_s).$$
    Proposition~\ref{prop:convergence} then gives us a uniform bound for
    $$ \left|\mathbb{E}\left[\tr_N\left(\Big(P(\G_s)P(\G_s)^*\Big)^q\right)\right]\right|^{1/2q}=\left|\mathbb{E}\left[Q(\G_s)\right]\right|^{1/2q}.$$
\end{proof}
\paragraph*{Free It\^{o} calculus}
The foundation of stochastic integration with respect to  a free semicircular Brownian motion was laid by Biane and Speicher~\cite{BianeSpeicher}, whose work represents a milestone in the development of free stochastic calculus. In particular, they introduced the first version of a free It\^{o} product rule \cite[Proposition 4.3.4]{BianeSpeicher}: for any free semicircular Brownian motion $(s_t)_{t\geq 0} $ and any continuous adapted process $(V_t)_{t\geq 0}$ in $(\mathcal{A},\varphi)$, the quadratic variation process is computed following the rule
$$d\langle s,s\rangle_t\sharp V_t=\varphi(V_t)1_{\mathcal{A}}\ dt.$$
In other words, for any continuous adapted processes $(A_t)_{t\geq 0}$, $(B_t)_{t\geq 0}$, $(C_t)_{t\geq 0}$, $(D_t)_{t\geq 0}$, if the processes $(m_t)_{t\geq 0}$ and $(n_t)_{t\geq 0}$ satisfy the free SDEs
\begin{align*}
    dm_t&=A_t\cdot  ds_t\cdot  B_t\\
    dn_t&=C_t\cdot ds_t\cdot D_t,
\end{align*}
then the process $(m_tn_t)_{t\geq 0}$ satisfies the free SDE
\begin{equation}
    d(m_tn_t)=dm_t\cdot n_t+m_t\cdot dn_t+d\langle m,n\rangle_t\label{eq:quadcov} 
\end{equation}
where the quadratic variation process is
$$
    d\langle m,n\rangle_t =A_t\cdot d\langle s,s\rangle_t\sharp (B_tC_t)\cdot D_t\\
    =A_t\varphi(B_tC_t)D_t\ dt.$$
Note that if $(m_t)_{t\geq 0}$ or $(n_t)_{t\geq 0}$ is an ordinary integral, we have the classical product rule $d(m_tn_t)=dm_t\cdot n_t+m_t\cdot dn_t$, or equivalently $d\langle m,n\rangle_t=0$. As in \cite{Hall-Ho-23}, we require a generalization of this result to the setting of several freely independent semicircular Brownian motions: whenever $(m_t)_{t\geq 0}$ and  $(n_t)_{t\geq 0}$ satisfy free SDEs 
\begin{align*}
    dm_t&=A_t\cdot  dx_t\cdot  B_t\\
    dn_t&=C_t\cdot dy_t\cdot D_t,
\end{align*}
driven by freely independent semicircular Brownian motions $(x_t)_{t\geq 0}$ and $(y_t)_{t\geq 0}$, we have again $d\langle m,n\rangle_t=0$, i.e.$$d(m_tn_t)=dm_t\cdot n_t+m_t\cdot dn_t.$$

For $\ell=1,\ldots,p$,  we recall that $(z_t^{(\ell)})_{t\geq0}=(z_{\lambda, \tau}^{(\ell)}(t))_{t\geq0}$ is a family of freely independent free rotated elliptic Brownian motions with parameters $\lambda, \tau$. As in \cite[Section 4.3]{Hall-Ho-23}, the following It\^{o} rules for the processes $(z_t^{(\ell)})_{t\geq0}$ are obtained from~\eqref{eq:quadcov}:
\begin{align}
    d\langle {z^{(\ell)}},z^{(\ell')}\rangle_t\sharp V_t &=\delta_{\ell,\ell'}(\lambda-\tau)\varphi (V_t)1_\mathcal{A}\ dt, \label{crochetz}\\
    d\langle  {z^{(\ell)}}^*, {z^{(\ell')}}^*\rangle_t\sharp V_t &=\delta_{\ell,\ell'}(\lambda-\bar{\tau}) \varphi (V_t)1_\mathcal{A}\ dt, \label{crochetzstar}\\
    d\langle  {z^{(\ell)}}, {z^{(\ell')}}^*\rangle_t\sharp V_t &= d\langle  {z^{(\ell)}}^*, {z^{(\ell')}}\rangle_t\sharp V_t=
   \delta_{\ell,\ell'} \lambda \varphi (V_t)1_\mathcal{A}\ dt\label{crochetzzstar}
\end{align}
(observe the similarity with \eqref{crochetZ},  \eqref{crochetZstar} and  \eqref{crochetZZstar}). The free multiplicative $(\lambda,\tau)$-Brownian motions $(g_t^{(\ell)})_{t\geq0}=(g_{\lambda, \tau}^{(\ell)}(t))_{t\geq0}$ are defined for any $t\geq 0$ as the solutions of the free stochastic differential equation 
\begin{equation}\label{free Stoch Diff Eq with sigma}
    dg_t^{(\ell)} =\sigma_\ell \, g_t^{(\ell)} \big( i \, dz_t^{(\ell)}  -\frac{\sigma_\ell}{2}(\lambda-\tau)\, dt \big) \quad \text{with}\quad g_t^{(\ell)} =1_\A.
\end{equation}
Note that $({g_t^{(\ell)}}^{\ast})_{t\geq 0}$, $({g_t^{(\ell)}}^{-1})_{t\geq 0}$ and $(({g_t^{(\ell)}}^{-1})^\ast)_{t\geq 0}$ satisfy respectively the  multiplicative free SDEs:
\begin{align*}
  d{g_t^{(\ell)}}^*&=-i {\sigma}    d{z_t^{(\ell)}}^* {g_t^{(\ell)}}^*  -\frac{\sigma_{\ell}^2}{2}(\lambda-\bar{\tau}){g_t^{(\ell)}}^* \, dt,
\\ 
  d{g_t^{(\ell)}}^{-1} &=-i \sigma  \, d{z_t^{(\ell)}} {g_t^{(\ell)}}^{-1} -\frac{\sigma_{\ell}^2}{2}(\lambda-\tau){g_t^{(\ell)}}^{-1}dt ,
  \\
  d({g_t^{(\ell)}}^{-1})^\ast &=i \sigma \,({g_t^{(\ell)}}^{-1})^\ast d{z_t^{(\ell)}}^*  -\frac{\sigma_{\ell}^2}{2}(\lambda-\bar{\tau})({g_t^{(\ell)}}^{-1})^\ast dt ,
\end{align*}
all initiated at the identity $1_{\mathcal{A}}$  (the proof of~\cite[Proposition~4.17]{kemp2016large} for the SDE of $({g_t^{(\ell)}}^{-1})_{t\geq 0}$ works \emph{mutatis mutandis} for the $(\lambda,\tau)$-Brownian motion). Using~\eqref{crochetz}, \eqref{crochetzstar} and \eqref{crochetzzstar}, it yields the following quadratic covariation processes (note again the similarities with Lemma~\ref{Lemma:Calcul QuadraticCov}).

\begin{lemma}\label{freequadcov}Let $(g_t)_{t\geq 0}$ denote one of the free $(\lambda,\tau)$-Brownian motion $({g_t^{(\ell)}}^{-1})_{t\geq 0}$, and for convenience, we denote by $(g_t^{-1,\ast})_{t\geq 0}$ the process $(({g_t^{(\ell)}}^{-1})^\ast)_{t\geq 0}$.
For any adapted process $V_t$, we have
{\small \begin{center}
\begin{tabular}{l|l}
$d\langle g ,g \rangle_t \sharp V_t = \sigma_\ell^2 (\tau-\lambda) g_t \varphi (g_t V_t)dt$ & $d\langle g , g^{-1} \rangle_t \sharp V_t = -\sigma_\ell^2 (\tau-\lambda)  \varphi ( V_t)dt$
\\ 
$d\langle g , g^\ast \rangle_t \sharp V_t = \sigma_\ell^2 \lambda g_t g_t^\ast\varphi ( V_t)dt$ & $d\langle g , g^{-1,\ast} \rangle_t \sharp V_t = -\sigma_\ell^2 \lambda g_t \varphi (g_t^{-1,\ast} V_t)dt$\\
$d\langle g^{-1} , g \rangle_t \sharp V_t = -\sigma_\ell^2 (\tau-\lambda)  \varphi ( V_t)dt$ & $d\langle g^{-1} , g^{-1} \rangle_t \sharp V_t = \sigma_\ell^2 (\tau-\lambda) g_t^{-1} \varphi (g_t^{-1} V_t)dt$
\\ 
$d\langle g^{-1} , g^\ast \rangle_t \sharp V_t = -\sigma_\ell^2 \lambda  g_t^\ast\varphi (g_t^{-1} V_t)dt$ & $d\langle g^{-1} , g^{-1,\ast} \rangle_t \sharp V_t = \sigma_\ell^2 \lambda  \varphi (g_t^{-1} V_tg_t^{-1,\ast})dt$\\
$d\langle g^\ast , g \rangle_t \sharp V_t = \sigma_\ell^2 \lambda  \varphi (g_t^\ast V_tg_t)I_Ndt$ & $d\langle g^\ast , g^{-1} \rangle_t \sharp V_t = -\sigma_\ell^2 \lambda g_t^{-1} \varphi ( g_t^\ast V_t)dt$
\\ 
$d\langle g^\ast , g^\ast \rangle_t \sharp V_t = {\sigma_\ell}^2 (\bar{\tau}-\lambda)  g_t^\ast\varphi (g_t^\ast V_t)dt$ & $d\langle g^\ast , g^{-1,\ast} \rangle_t \sharp V_t = -{\sigma_\ell}^2 (\bar{\tau}-\lambda)  \varphi ( g_t^\ast V_t g_t^{-1,\ast})dt$\\
{\small $d\langle g^{-1,\ast} , g \rangle_t \sharp V_t =  -\sigma_\ell^2 \lambda g_t^{-1,\ast} \varphi (g_t V_t)dt$} & {\small
$d\langle g^{-1,\ast} , g^{-1} \rangle_t \sharp V_t = \sigma_\ell^2 \lambda g_t^{-1,\ast} g_t^{-1} \varphi ( V_t)dt$}
\\ 
{\small $d\langle g^{-1,\ast}  ,  g^\ast \rangle_t \sharp V_t = -{\sigma_\ell}^2 (\bar{\tau}-\lambda)  \varphi ( V_t)dt$ } & {\small $d\langle g^{-1,\ast} ,   g^{-1,\ast} \rangle_t  \sharp V_t = {\sigma_\ell}^2 (\bar{\tau}-\lambda)  g_t^{-1,\ast} \varphi  ( g_t^{-1,\ast}V_t) dt$}
\end{tabular}

\end{center} }
\end{lemma}
We are now ready to compute the generator of the process $\g_s$.
\begin{lemma}
    \label{freeGenerator}
For any pure trace polynomial $P$, we have
$$\frac{d}{ds}P(\g_s)=\Delta P(\g_s).$$
\end{lemma}
\begin{proof}
First-of-all, if $P=\tr(M_1)\cdots \tr(M_n)\in \tr\ \mathbb{C}\{X_{\ell},X_{\ell}^*,X_{\ell}^{-1},X_{\ell}^{-1,*}:1\leq \ell \leq p\}$ is an element of the standard basis, let us remark that
\begin{align*}
    \frac{d}{ds}P(\g_s)&=\frac{d}{ds}\left( \varphi\Big(M_1(\g_s)\Big) \cdots \varphi\Big(M_n (\g_s)\Big)\right)\\
    &=\sum_{i=1}^n\prod_{i' \neq i} \varphi(M_{i'}(\g_s))\cdot \frac{d}{ds}\varphi(M_{i}(\g_s)),
\end{align*}
and similarly, using the definition of $\Delta$,
\begin{align*}
   \Delta P(\g_s)&=\left(\sum_{i=1}^n\prod_{i' \neq i} \tr(M_{i'})\cdot \Delta(\tr(M_i))\right)(\g_s)\\
   &=\sum_{i=1}^n\prod_{i' \neq i} \varphi(M_{i'}(\g_s))\cdot \Big(\Delta(\tr(M_i))\Big)(\g_s),
\end{align*}
As a consequence, it suffices to prove the result in the case where $P=\tr(M)$ for a monomial~$M$. However, the computation of
    $\frac{d}{ds}P(\g_s)=\frac{d}{ds}\varphi[M(\g_s)]$ can be done explicitely. Indeed,
    $$d(P(\g_s))=d(\varphi[M(\g_s)])=\varphi[d(M(\g_s))]=\Delta P(\g_s) ds$$
    where the computation of $d(M(\g_s))$ is given by the free It\^{o} product rule, with the quadratic covariance processes given by Lemma~\ref{freequadcov} (we only need the quadratic covariance processes and the drifts because the martingale parts are vanishing under $\varphi$).
\end{proof}

Using \cite[Lemma 2.5]{cebron2013free}, we get the following result as a direct corollary of Lemma~\ref{freeGenerator}.

\begin{lemma}
Let $P\in E_d$ and $s\geq 0$. We have
$$P(\g_s)=\left(e^{s\Delta}P\right)(1_\mathcal{A}).$$
\end{lemma}
Because $\varphi$ is normalized, the maps $P\mapsto P(1_{\mathcal{A}})$ and $P\mapsto P(1)$ are the same linear form. Together with Proposition~\ref{prop:convergence}, the equality
$P(\g_s)=\left(e^{s\Delta}P\right)(1)$
yields Theorem~\ref{th:weakconv}. We also get a uniform bound for the $L_{2q}$-norms of any polynomial in $\g_s$.

\begin{cor}
    \label{NormeLp-g}
    Let $t\geq 0$, $q\in \mathbb{N}$ and let $P\in \mathbb{C}\langle X_{\ell},X_{\ell}^*,X_{\ell}^{-1},X_{\ell}^{-1,*}:1\leq \ell \leq p\rangle$ be any polynomial in $4p$ variables. There exists a constant $C>0$ such that, for all $0\leq s \leq t$, we have
    $$\left\|P(\g_s)\right\|_{L_{2q}}:=\left| \varphi\left(\Big(P(\g_s)P(\g_s)^*\Big)^q\right)\right|^{1/2q} \leq C.$$
\end{cor}
\begin{proof}
    Note that if $d$ is the degree of $P$, there exists a pure trace polynomial $Q\in E_{2d}$ such that
$$\varphi\left(\Big(P(\g_s)P(\g_s)^*\Big)^q\right)=Q(\g_s)=e^{s\Delta}Q(1)$$
which is bounded uniformly for $s\leq t$ as in the proof of Proposition~\ref{prop:convergence}.
\end{proof}
\subsection{The left-invariant and the right-invariant Brownian motions}
In this article, we are mainly interested with the \emph{left-invariant} multiplicative $(\lambda,\tau)$-Brownian motion $(G_{\lambda, \tau}(t))_{t\geq 0}$ satisfying the following SDE:
\begin{equation*}
    dG_{\lambda, \tau}(t)= \, G_{\lambda, \tau}(t)\big( i \, dZ_{\lambda, \tau}(t) -\frac{1}{2}(\lambda-\tau)\, dt \big) \quad \text{with}\quad G_{\lambda, \tau}(0)=I_N.
\end{equation*}
However, it corresponds to a diffusion on the general linear group (see \cite{mckean2024stochastic}), and the density of the distribution of $G_{\lambda, \tau}(t)$ (for any fixed $t\geq 0$) is the heat kernel, i.e. the solution of the heat equation relative to some \emph{left-invariant} Laplacian. 

In the proof of the main result, we also need to consider the \emph{right-invariant} multiplicative $(\lambda,\tau)$-Brownian motion $(\widetilde{G}_{\lambda, \tau}(t))_{t\geq 0}$ satisfying the following SDE:
\begin{equation*}
d\widetilde{G}_{\lambda, \tau}(t)= \, \big( i \, d\widetilde{Z}_{\lambda, \tau}(t) -\frac{1}{2}(\lambda-\tau)\, dt \big)\widetilde{G}_{\lambda, \tau}(t) \quad \text{with}\quad \widetilde{G}_{\lambda, \tau}(0)=I_N.
\end{equation*}
It corresponds also to a diffusion on the general linear group (see \cite{mckean2024stochastic}), with the density of the distribution of $\widetilde{G}_{\lambda, \tau}(t)$ (for any fixed $t\geq 0$) given by (\emph{a priori}) another heat kernel, i.e. the solution of the heat equation relative to some \emph{right-invariant} Laplacian. 

Now, because $(Z_{\lambda, \tau}(t))_{t\geq 0}$ and $(\widetilde{Z}_{\lambda, \tau}(t))_{t\geq 0}$ have the same law (a rotated elliptic Brownian motion), the left-invariant Laplacian and the right-invariant Laplacian are relative to one same and single metric. As a consequence, the two solutions to the two heat equations (one for each Laplacian) starting at $I_N$ are the same (see
\cite[Theorem 2.7]{driver1995kakutani}) and for any fixed $t\geq 0$, the random matrices $G_{\lambda, \tau}(t)$ and $\widetilde{G}_{\lambda, \tau}(t)$ have the same probability distribution. Of course, a simultaneous change of time do not change the equality in law, and taking $p$ independent copies neither. The following proposition summarizes the above discussion.

\begin{proposition}\label{samelaw}
    We let, for any $\ell=1,\dots, p$,  $(Z_t^{(\ell)})_{t\geq0}:=(Z_{\lambda, \tau}^{(\ell)}(t))_{t\geq0}$ be a family of independent rotated elliptic Brownian motions and consider the \emph{left-invariant} multiplicative $(\lambda,\tau)$-Brownian motions $(G_t^{(\ell)})_{t\geq0}:=(G_{\lambda, \tau}^{(\ell)}(t))_{t\geq0}$ defined for any $t\geq 0$ as the solution of the  stochastic differential equation 
\begin{equation*}
    dG_t^{(\ell)}=\sigma_\ell \, G_t^{(\ell)} \big( i \, dZ_t^{(\ell)}  -\frac{\sigma_\ell}{2}(\lambda-\tau)\, dt \big) \quad \text{with}\quad G_0^{(\ell)}=I_N,
\end{equation*}
for some positive real number $\sigma_\ell$. Similarly,  we let, for any $\ell=1,\dots, p$,  $(\widetilde{Z}_t^{(\ell)})_{t\geq0}:=(\widetilde{Z}_{\lambda, \tau}^{(\ell)}(t))_{t\geq0}$ be a family of independent rotated elliptic Brownian motions and consider the  \emph{right-invariant} multiplicative $(\lambda,\tau)$-Brownian motions $(\widetilde{G}_t^{(\ell)})_{t\geq0}:=(\widetilde{G}_{\lambda, \tau}^{(\ell)}(t))_{t\geq0}$ defined for any $t\geq 0$ as the solution of the  stochastic differential equation 
\begin{equation*}
    d\widetilde{G}_t^{(\ell)}=\sigma_\ell \, \big( i \, d\widetilde{Z}_t^{(\ell)}  -\frac{\sigma_\ell}{2}(\lambda-\tau)\, dt \big)\widetilde{G}_t^{(\ell)}  \quad \text{with}\quad \widetilde{G}_0^{(\ell)}=I_N.
\end{equation*}
Then, for any $t\geq 0$, the $p$-tuples 
\begin{equation*}
    \G_t:= \big(G_t^{(1)} , \dots, G_t^{(p)} \big)
    \quad \text{and} \quad
    \widetilde{\G}_t:= \big(\widetilde{G}_t^{(1)} , \dots, \widetilde{G}_t^{(p)} \big)
\end{equation*}
have the same probability distribution.
\end{proposition}

Now, we proved in section~\ref{proofofweakconv} that the noncommutative  distribution of the free multiplicative $(\lambda,\tau)$-Brownian motions $\g_t$  is the limit of the noncommutative distribution of $\G_t$. Similarly, if $\tilde{\g}_t:=(\tilde{g}_t^{(1)},\ldots,\tilde{g}_t^{(p)})$ is a $p$-tuple of  \emph{right-invariant} free multiplicative $(\lambda,\tau)$-Brownian motions, i.e.  the solutions of the free stochastic differential equations 
\begin{equation*}
    d\tilde{g}_t^{(\ell)} = \sigma_\ell \big( i \, d\tilde{z}_t^{(\ell)}  -\frac{\sigma_\ell}{2}(\lambda-\tau)\, dt \big) \tilde{g}_t^{(\ell)} \quad \text{with}\quad \tilde{g}_t^{(\ell)} =1_\A,
\end{equation*}
the noncommutative  distribution of $\tilde{\g}_t$  is the large-$N$ limit of the noncommutative distribution of $\widetilde{G}_t$ (the proof is a straightforward adaptation of section~\ref{proofofweakconv}). As a consequence, we get the following proposition.
\begin{proposition}\label{Prop:verylast}
    For any $t\geq 0$, the noncommutative distribution of $\g_t$ and $\tilde{\g}_t$ coincide.
\end{proposition}

\section*{Acknowledgement}
This project started during the meeting "Non-commutative Function Theory and Free Probability" at the MFO Oberwolfach Research Institute for Mathematics. The authors would like to thank the organizers for the invitation and for providing an inspiring research environment. This work was supported by the NYU Abu Dhabi Grant "Free Probability for Neural Networks and Stochastic Differential Equations." 
\bibliographystyle{plain}
\bibliography{ref}
\end{document}